\documentclass[11pt,a4paper]{article}
\usepackage[utf8]{inputenc}
\usepackage{amsmath}
\usepackage{amsfonts}
\usepackage{amssymb}
\usepackage{extarrows}
\usepackage{amsthm}
\usepackage{enumerate}
\PassOptionsToPackage{normalem}{ulem}
\usepackage{ulem}
\usepackage[margin=0.8in]{geometry} 
\usepackage{array}
\usepackage{cases}
\usepackage{mathrsfs}
\usepackage{longtable}
\allowdisplaybreaks[4]

\usepackage{color}
\usepackage{dsfont}
\usepackage{marginnote}
\usepackage{enumitem}
\usepackage{verbatim}
\usepackage{graphicx}
\usepackage{algorithm}
\usepackage[noend]{algpseudocode}
\usepackage[colorlinks,linkcolor=blue,anchorcolor=blue,citecolor=blue]{hyperref}
\theoremstyle{plain}
\newtheorem{theorem}{\protect Theorem}[section]
\newtheorem{prop}[theorem]{\protect Proposition}
\newtheorem{definition}[theorem]{\protect Definition}

\newtheorem{lemma}[theorem]{\protect Lemma}
\newtheorem{remark}[theorem]{\protect Remark}
\newtheorem{assumption}{\protect Assumption}

\newcommand{\Rb}{\mathbb{R}}

\title{Policy Iteration Achieves Regularized Equilibrium\\
under Time Inconsistency}

\author{Yu-Jui Huang\thanks{Department of Applied Mathematics, University of Colorado, Boulder, CO, USA. \url{yujui.huang@colorado.edu}}
\and Xiang Yu\thanks{Department of Applied Mathematics,  The Hong Kong Polytechnic University, Kowloon, Hong Kong. \url{xiang.yu@polyu.edu.hk}}
	\and Keyu Zhang\thanks{
	Department of Applied Mathematics,  The Hong Kong Polytechnic University, Kowloon, Hong Kong.  \url{keyu1.zhang@polyu.edu.hk}}}
\date{\vspace{-0.3cm}}

\begin{document}
	\maketitle

	\begin{abstract}
	 For a general entropy-regularized time-inconsistent stochastic control problem, we propose a policy iteration algorithm (PIA) and establish its convergence to an equilibrium policy with an exponential convergence rate. The design of the PIA is based on a coupled system of non-local partial differential equations, called the {\it exploratory equilibrium Hamilton--Jacobi--Bellman (EEHJB) equation}. As opposed to the standard time-consistent case, policy improvement fails in general and the target value function (now an equilibrium value function) is not even known to exist a priori. To overcome these, 
     we prove that the value functions generated by the PIA form a Cauchy sequence in a specialized Banach space, hence admit a limit, and the rate of convergence is exponential, on the strength of the Bismut--Elworthy--Li formula of stochastic representation. The limiting value function is shown to fulfill the EEHJB equation, which induces an equilibrium policy in a Gibbs form. Such convergence in value additionally implies uniform convergence of the generated policies to the equilibrium policy, again with an exponential rate.
     As a byproduct, the PIA gives a constructive proof of the global existence and uniqueness of a classical solution to the general EEHJB equation. {\color{black}We further extend the convergence analysis to the infinite-horizon setting under some additional conditions on discounting. Numerical examples of two financial applications are presented to illustrate the theoretical findings.}

       

        
	\end{abstract}

	\section{Introduction}\label{sec:intro}
    	A policy iteration algorithm (PIA), originated from ideas of Bellman \cite{Bellman1955Functional,Bellman1957Dynamic} and Howard \cite{Howard1960Dynamic}, recursively updates a current policy in a {\it greedy} manner (i.e., by maximizing {\it immediate} reward). In practice, one performs the PIA in the hope that it will eventually lead to an optimal policy. 
        For classical stochastic control problems in diffusion models, 
        Krylov \cite{Krylov2008Controlled} and Puterman \cite{Puterman1982On} showed that the PIA recovers the optimal value for a specific control problem with a compact space-time domain by classical estimates of partial differential equations (PDEs); Jacka and Mijatovi\'c \cite{Jacka2017On} outlined a general set of assumptions ensuring the PIA's convergence to an optimal policy; 
        Kerimkulov et al.\ \cite{kerimkulov2020exponential}
        used the theory of backward stochastic differential equations to establish the PIA's convergence with an appropriate convergence rate, when only the drift coefficient of the state process is controlled; Possama\"{i} and Tangpi \cite{DylanT24} addressed the PIA's convergence in a non-Markovian setting through successive approximations by linear-quadratic (LQ) control problems.

	Due to recent developments in continuous-time reinforcement learning (RL), significant attention has been drawn to \textit{entropy-regularized} stochastic control problems: exploration in RL is modeled 
   by the use of {\it relaxed control policies} (i.e., measure-valued processes), which randomize actual control actions; the entropy of the relaxed control policies is incorporated into the objective function, 
   where a temperature parameter $\lambda > 0$ governs the tradeoff between exploitation (reward maximization) and exploration (action randomization). This framework was first introduced by 
   Wang et al.\ \cite{wang2020reinforcement_continuous_time},
   where a detailed analysis in the LQ case was carried out. It has been extended in subsequent studies to the settings of mean-field games, mean-field control, optimal stopping, and optimal regime switching; see      
    \cite{reisinger2021regularity, tang2022exploratory_hjb_sicon, jia2022policy_gradient_continuous_time, bo2025optimal, guo2022entropy, jia2023q, pham25, wei2025continuous, wyy2024, dong, DFX24, dai2024learning, HLYZ25}, just to name a few.

    With entropy regularization, it is natural to ask if the PIA can still converge desirably.  
    For mean-variance portfolio selection, Wang and Zhou \cite{wang2020continuous} showed explicitly that the PIA yields an optimal policy within two iterations, relying on the inherent LQ structure. By  sophisticated PDE estimates, Huang et al.\ \cite{huang2025convergence} established the PIA's convergence in a general (non-LQ) diffusion model, when the coefficients are bounded and only the drift is controlled. Tran et al.\ \cite{tran2025policy} generalized this to two situations: one with unbounded coefficients; the other keeps coefficients bounded but allows the diffusion to be slightly controlled (under a ``smallness'' condition). Recently, Ma et al. \cite{ma2025convergence} examined the PIA via probabilistic representations of solutions to the iterative PDEs and their derivatives, obtaining a super-exponential convergence rate. In the context of optimal stopping, Dong \cite{dong} transformed the stopping problem into a regular control problem, where the PIA's convergence was established. Under a different type of entropy regularization, Dianetti et al.\ \cite{DFX24} proposed a singular control formulation of optimal stopping and showed that the PIA converges in some examples. For optimal regime switching, Huang et al.\ \cite{HLYZ25} obtained the PIA's convergence by analyzing the system of exploratory Hamilton-Jacobi-Bellman (HJB) equations. All these studies, in a nutshell, rely on the \textit{policy improvement property} plus \textit{compactness arguments}: the value functions generated by the PIA are first shown to be monotonically improving, so that their limit is well-defined; 
    once uniform bounds for these value functions and their derivatives are established, the limiting  function inherits sufficient regularity by compactness arguments---hence must be the optimal value function (by verification).

    This line of arguments, nonetheless, fails under {\it time inconsistency}. Many realistic financial and economic models are time-inconsistent (i.e., a policy deemed optimal today may not be optimal at future dates), due to e.g., non-exponential discounting, dependence of rewards on initial time and state, or non-linear objectives in expectations (like the mean-variance objective). As there is no longer a dynamically optimal policy on the entire planning horizon, a prevailing alternative is to view the control problem as an intra-personal game among the continuum of an agent's current and future selves and seek a \textit{subgame perfect Nash equilibrium}, 
    defined as a policy that cannot be improved by any one-shot deviation made by the current self; 
    see e.g., \cite{Bjork2017, BjorkKhapkoMurgoci2021} for details.
    

    While the PIA can still be formally stated under time inconsistency, analyzing its convergence faces two fundamental challenges. First, as noted in Dai et al.\ \cite{dai2023learning}, {\it policy improvement fails} in general under time inconsistency, because an agent's goal is no longer value improvement but equilibrium achievement. Hence, the standard proof arguments mentioned above---which hinge on monotonicity of generated values---simply break down. Second and most critically, time inconsistency deprives the PIA of a single, well-specified target. In the time-consistent case, the optimal value $V^*$, well-defined a priori, anchors the analysis of the PIA: one can first analyze the regularity of $V^*$ and inversely identify a suitable function space that facilitates compactness arguments; or, one can directly estimate the error between the $n^{th}$ iterate $V^n$ and the target $V^*$. By contrast, the target under time inconsistency is an equilibrium value, which depends on an equilibrium policy that is unknown a priori. 
    That is, the PIA now needs to reach a target without any knowledge of the target (even its existence)---a much more formidable task than that in the time-consistent case. 
    
    To the best of our knowledge, the PIA's convergence under time inconsistency was studied only by Dai et al.\ \cite{dai2023learning}: for mean-variance portfolio selection, the PIA was shown therein to converge to a specific equilibrium policy, if the initial policy chosen is close enough to that equilibrium. This result, however, relies crucially on the inherent LQ structure and cannot be easily generalized. Also, its choice of initial policy requires a known equilibrium, which may not be available in general. 
    
		
    In this paper, we consider a general entropy-regularized stochastic control problem under time inconsistency. In particular, the model is non-LQ and covers various forms of time inconsistency all at once; see \eqref{J^pi} below. Our goal is to establish the PIA's general convergence to an equilibrium policy (unknown a priori) with a suitable convergence rate. 
    

    First, we extend the {\it equilibrium HJB} equation in \cite{Yong2012, lei_nonlocal_2023, lei_nonlocality_2024, liang2025}, stated for traditional time-inconsistent problems (without entropy regularization), to the case of entropy regularization. The derived equation, called the \textit{exploratory} equilibrium HJB (EEHJB) equation, is new to the literature. It is a coupled PDE system with non-local terms for two auxiliary value functions $(V^{\hat\pi,1},V^{\hat\pi,2})$, which jointly characterize an equilibrium policy $\hat\pi$ via a Gibbs measure; see the system in \eqref{eq:EEHJB} and Gibbs measure in \eqref{eq:gibbs_policy}-\eqref{Z}. Leveraging on this Gibbs-measure relation, we tailor-make a PIA for our time-inconsistent setting, which iterates $(V^{\hat\pi,1},V^{\hat\pi,2})$ jointly; see Section~\ref{subsect:pia}. It is worth noting that, besides the Gibbs-measure structure, our equation allows dependence on initial time and state as well as additional nonlinearity (through $G$ in \eqref{eq:EEHJB}), while the nonlocal PDE in \cite{Yong2012, lei_nonlocal_2023, lei_nonlocality_2024, liang2025} only allows dependence on initial time. 

    To show that the PIA converges in the absence of policy improvement (due to time inconsistency), we turn away from proving any monotonicity of the iterates $\{(V^{\pi_n,1},V^{\pi_n,2})\}_{n\ge 1}$ and instead focus on whether the iterates form a \textit{Cauchy sequence} in a suitable Banach space. 
    When the initial input of the PIA is sufficiently smooth (but otherwise general), we show that $\{(V^{\pi_n,1},V^{\pi_n,2})\}_{n\ge 1}$ satisfy a recursive PDE with appropriate H{\"o}lder bounds, using the Bismut--Elworthy--Li representation formula (inspired by \cite{ma2025convergence}); see Proposition \ref{prop:1}. Based on this, we estimate in detail the norm of $(V^{\pi_{n+1},1}-V^{\pi_n,1}, V^{\pi_{n+1},2}-V^{\pi_n,2})$ in a specialized Banach space, and find that the norm decreases exponentially in $n$. This shows that $\{(V^{\pi_n,1},V^{\pi_n,2})\}_{n\ge 1}$ is indeed Cauchy in the Banach space---hence admitting a limit $(V^{*,1},V^{*,2})$ in the same space---and it converges to $(V^{*,1},V^{*,2})$ exponentially. This directly implies sufficient regularity of $(V^{*,1},V^{*,2})$, which are shown to satisfy the EEHJB equation and thus yield the existence of an equilibrium policy $\pi^*$. 
    Moreover, 
    the policy iterates $\pi^n$ converge uniformly to $\pi^*$, also exponentially. See the details in Theorem~\ref{thm:convergence}, the main result of this paper. {\color{black}Furthermore, under some additional conditions on the discount functions, we establish estimates that are uniform over time horizon and extend the exponential convergence of the PIA to the infinite-horizon setting; see Proposition~\ref{prop:ih}.} 


    Comparing our study with \cite{ma2025convergence} well illustrates the distinct challenges posed by time inconsistency. For a classical time-consistent problem, \cite{ma2025convergence} directly estimates the error between the $n^{th}$ iterate $V^n$ and the optimal value function $V^*$, which is a priori well-defined and known to be the unique classical solution to an exploratory HJB equation (thereby inducing an optimal policy already). 
    Under time inconsistency, the standard exploratory HJB equation becomes our EEHJB system, whose existence of solutions is  unknown from the literature. Without a clearly-specified target value function, we can no longer discuss the error between an iterate and the target, and are left to work solely with the iterates themselves. Ultimately, we prove that the iterates are Cauchy, and it is at this point we finally uncover the target value function (i.e., the limit of the Cauchy sequence) and its induced equilibrium policy. {\color{black}In addition, because the iterative policies $\pi^{n+1}$ depend on $V_x^{\pi_n,1}(t, t, x, x)$, ensuring that the PIA is well-defined requires establishing regularity of $V^{\pi_n,1}(\tau, t, y, x)$ with respect to the reference variables $(\tau, y)$. The nonlinear expectation term further introduces the second auxiliary function $V^{\pi_n,2}$ into the iteration, and the estimates for the resulting pair of auxiliary functions $(V^{\pi_n,1}, V^{\pi_n,2})$ require a delicate analysis;} see Remark~\ref{remark:comparison_Ma} for details. 

    Our study has three main contributions. First, we prove the PIA's convergence in a general time-inconsistent setting, with no need of policy improvement or a well-defined target value function. Second, as a byproduct, the PIA itself gives a {\it constructive} proof of global existence and uniqueness of a classical solution to the coupled, non-local EEHJB equation. To the best of our knowledge, this is the {first well-posedness result} for this class of equilibrium HJB equations. As mentioned before, prior studies (e.g., \cite{Yong2012, lei_nonlocal_2023, lei_nonlocality_2024, liang2025}) focus on equilibrium HJB equations (without entropy regularization) that allow only dependence on initial time, while our EEHJB equation also allows dependence on initial state and additional nonlinearity. Such generality was supposed to pose significant analytical complexity, but it was circumvented by the PIA under entropy regularization; see Remark~\ref{rem:PIA non-local}.
    
 
    The third contribution is a new way of solving time-inconsistent stochastic control problems. 
    The literature has focused on deriving a so-called {\it extended HJB} equation
    and solving it case by case through guesses of a clever ansatz, from which an equilibrium policy can be deduced; see \cite{Bjork2017,BjorkKhapkoMurgoci2021,BMZ14, EP08,EMP12,dong2014time,HS26}, among many others. By contrast, we just arbitrarily pick initial functions to initiate the PIA, and the iterations will take over the rest to produce an equilibrium policy---{\it with no need of case-by-case guessing}. This is reminiscent of the fixed-point approach for time-inconsistent stopping problems in \cite{HN18,HNZ20,HY21,HZ20,HZ22}, where recursive game-theoretic reasoning is shown to yield intra-personal equilibrium stopping rules. Our PIA performs the same recursive intra-personal game-theoretic reasoning, but for entropy-regularized stochastic control problems; see Remark~\ref{rem:fixed-point approach}. 
    
 
    Our analysis focuses on the case where only the drift coefficient of the state process is controlled, similarly to \cite{huang2025convergence, kerimkulov2020exponential}. When the diffusion coefficient is also controlled, {\color{black}the updated policy depends on second-order spatial derivatives of the auxiliary functions, which leads to a loss of derivatives in the estimates based on the Bismut--Elworthy--Li formula; see Remark~\ref{rem:controlled_diffusion}. In fact,} the study of the PIA, even in the time-consistent case, requires much more restrictive model assumptions (e.g., the state process is one-dimensional, or the diffusion is controlled only ``slightly'') as well as very different proof techniques; see e.g. \cite{tran2025policy, ma2025convergence}. Hence, we would like to leave this for future research. 
   

 The rest of the paper is organized as follows. 
 Section \ref{sec:formulation} introduces the general entropy-regularized time-inconsistent stochastic control problem to be studied. In Section \ref{subsect:eehjb}, we derive the EEHJB equation and characterize an equilibrium policy in the Gibbs form. Section \ref{subsect:pia} formulates the PIA. Section \ref{sec:reg} proves a key regularity result for the PIA. Based on this, the convergence analysis of the PIA {\color{black}(including an infinite-horizon extension under sufficiently strong discounting)} is carried out in detail in Section \ref{sec:convergence}. {\color{black}Finally, Section~\ref{sec:numerical} illustrates our theoretical findings through two numerical examples in financial applications under non-exponential discounting: equilibrium effort for a given incentive contract and retirement account withdrawals over an uncertain lifetime.}

	\subsection{Notation throughout the paper}\label{subsect:notation}
    We write $x \cdot y$ for the usual inner product in Euclidean spaces, and use the same notation for  matrix-vector and matrix-matrix products. 
    Given $u \in \mathbb{R}^m$ and $M \in \mathbb{R}^{m \times d \times d}$, define $u \otimes M \in \mathbb{R}^{d \times d}$  by
	\[
	[u \otimes M]_{ij} := \sum_{k=1}^m u_k M_{kij}.
	\]
	  
	  For any $0 \leq s_1 \leq s_2 < \infty$, we define the trapezoidal domain
	\begin{equation*}
		\Delta[s_1, s_2] := \left\{ (\tau, t) :0\le\tau\le t,\ s_1\le t\le s_2 \right\}.
	\end{equation*}
	In particular, for any $T>0$, $\Delta[0, T] = \{ (\tau, t) : 0 \leq \tau \leq t \leq T \}$. 
 We use both $\partial_{\cdot}\phi$ and $\phi_{\cdot}$ to denote derivatives, whichever is more convenient.  
 For any $\alpha \in (0,1)$ and $\varphi: [s_1, s_2] \times \mathbb{R}^d \to \mathbb{R}$, we define the following parabolic norms and semi-norms (in line with e.g., \cite{ladyzenskaja_linear_nodate}):
 \begin{align*}
 	&\|\varphi\|^{(0)} = \sup_{(t,x) \in [s_1,s_2]\times \mathbb{R}^d} |\varphi(t,x)|, \quad
 	\|\varphi\|^{(2)} = \|\varphi\|^{(0)} + \|\varphi_t\|^{(0)} + \sum_{i=1}^d \|\varphi_{x_i}\|^{(0)} + \sum_{i,j=1}^d \|\varphi_{x_i x_j}\|^{(0)}, \\
 	&\langle \varphi \rangle_t^{\left(\frac{\alpha}{2}\right)} = \sup_{\substack{t, t' \in [s_1, s_2], x \in \mathbb{R}^d \\ t \neq t'}} \frac{|\varphi(t, x) - \varphi(t', x)|}{|t - t'|^{\frac{\alpha}{2}}}, \quad
 	\langle \varphi \rangle_x^{(\alpha)} = \sup_{\substack{t \in [s_1, s_2], x, x' \in \mathbb{R}^d \\ 0 < |x - x'| \leq 1}} \frac{|\varphi(t, x) - \varphi(t, x')|}{|x - x'|^\alpha}, \\
 	&\langle \varphi \rangle^{(\alpha)} = \langle \varphi \rangle_t^{\left(\frac{\alpha}{2}\right)} + \langle \varphi \rangle_x^{(\alpha)}, \quad
 	\|\varphi\|^{(\alpha)} = \|\varphi\|^{(0)} + \langle \varphi \rangle^{(\alpha)}, \\
 	&\|\varphi\|^{(2+\alpha)} = \|\varphi\|^{(2)} + \langle \varphi_t \rangle^{(\alpha)} + \sum_{i,j=1}^d \langle \varphi_{x_i x_j} \rangle^{(\alpha)} + \sum_{i=1}^d \langle \varphi_{x_i} \rangle_t^{\left(\frac{1 + \alpha}{2}\right)}.
 \end{align*}
 For a vector-valued function $\Psi = (\psi^1, \dots, \psi^m)$, its norm is defined as the sum of the norms of its components, e.g., $\|\Psi\|^{(k+\alpha)} := \sum_{j=1}^m \|\psi^j\|^{(k+\alpha)}$. We sometimes emphasize a domain $K \subseteq [0, T] \times \mathbb{R}^d$ by adding a subscript to the norm, i.e., $\|\varphi\|_K^{(0)}$ for example. In this paper, the finiteness of each parabolic norm above is understood to imply the existence and joint continuity of all derivatives involved in its definition.
 For $k \in \{0, 2\}$ and $\beta \in [0, 1)$, $C^{k+\beta}([s_1, s_2] \times \mathbb{R}^d)$ denotes the space of real- or vector-valued functions $\phi$ satisfying $\|\phi\|^{(k+\beta)}<\infty$; the spaces $C^{k+\beta}(\mathbb{R}^d)$ (resp. $C^{k+\beta}([0,T])$) are defined analogously for functions independent of the temporal (resp. spatial) variable.

 For functions $\psi(\tau, t, y, x)$ defined on $\Delta[s_1, s_2] \times \mathbb{R}^d \times \mathbb{R}^d$, we define for $\gamma \in \{0, \alpha, 2, 2+\alpha\}$ the norms
 \begin{align}
 	&[\psi]^{(\gamma)}_{[s_1, s_2]} := \sup_{\tau \in [0, s_2], y \in \mathbb{R}^d} \|\psi(\tau, \cdot, y, \cdot)\|^{(\gamma)}_{[\tau \vee s_1, s_2] \times \mathbb{R}^d}, \notag\\
 	&\|\psi\|_{[s_1, s_2]}^{(\gamma)} := \sup_{\tau \in [0, s_2], y \in \mathbb{R}^d} \left\{ \|\psi(\tau, \cdot, y, \cdot)\|_{[\tau \vee s_1, s_2] \times \mathbb{R}^d}^{(\gamma)} + \|\psi_{\tau}(\tau, \cdot, y, \cdot)\|_{[\tau \vee s_1, s_2] \times \mathbb{R}^d}^{(\gamma)} \right. \notag\\
 	&\qquad\qquad \left. + \|\psi_y(\tau, \cdot, y, \cdot)\|_{[\tau \vee s_1, s_2] \times \mathbb{R}^d}^{(\gamma)} \right\}.\label{holder norm}
 \end{align}
 These norms induce the Banach spaces\footnote{Similar spaces were considered in \cite{lei_nonlocal_2023, lei_nonlocality_2024}. For readers' convenience, we briefly verify the completeness of $\Theta_{[s_1, s_2]}^{(2)}$, which will be used in the proof of Theorem \ref{thm:convergence}. Let $\{\psi_n\}_{n \geq 1}$ be a Cauchy sequence in $\Theta_{[s_1, s_2]}^{(2)}$. For any fixed pair $(\tau, y)$, the sequence $\{\psi_n(\tau, \cdot, y, \cdot)\}$ is a Cauchy sequence in the Banach space $C^2([\tau \vee s_1, s_2] \times \mathbb{R}^d)$. By the completeness of this space, it converges to a unique limit $\psi(\tau, \cdot, y, \cdot) \in C^2([\tau \vee s_1, s_2] \times \mathbb{R}^d)$. As the Cauchy condition holds uniformly over all $(\tau, y)$ due to the supremum in the definition of the norm $[\cdot]^{(2)}_{[s_1, s_2]}$, $\{\psi_n\}$ must converge to $\psi$ under this norm. Furthermore, this uniform convergence guarantees that $\psi$ and all derivatives ($\psi_t, \psi_x, \psi_{xx}$) preserve joint continuity on $\Delta[s_1, s_2] \times \mathbb{R}^d \times \mathbb{R}^d$ and satisfy $[\psi]^{(2)}_{[s_1, s_2]} < \infty$. Thus, $\psi \in \Theta_{[s_1, s_2]}^{(2)}$, confirming the completeness.}
\begin{align*}
	\Theta_{[s_1, s_2]}^{(2)} &:= \Big\{ \psi : [\psi]_{[s_1, s_2]}^{(2)} < \infty, \text{ and } {\psi, \psi_t, \psi_x, \psi_{xx} \in C(\Delta[s_1, s_2] \times \mathbb{R}^d \times \mathbb{R}^d)} \Big\}, \\
	\Xi_{[s_1, s_2]}^{(2+\alpha)} &:= \left\{ \psi \in C(\Delta[s_1, s_2] \times \mathbb{R}^d \times \mathbb{R}^d) : \|\psi\|_{[s_1, s_2]}^{(2+\alpha)} < \infty \right\}.
\end{align*}


\section{Problem Formulation}\label{sec:formulation}
Fix $d,\ell\in\mathbb N$ and a finite time horizon $T>0$. Let $(\Omega, \mathcal{F}, \mathbb{P})$ be a complete probability space where a standard $d$-dimensional Brownian motion $W = \{W_s\}_{0 \leq s \leq T}$ is defined, and we denote by $\mathbb{F} = \{\mathcal{F}_s\}_{0 \leq s \leq T}$ the $\mathbb{P}$-augmented natural filtration generated by $W$. Consider an action space $\mathbb{A} \subseteq \mathbb R^\ell$ 
with a positive finite volume, i.e., $0 < |\mathbb{A}| < \infty$, and let $\mathscr{P}_0(\mathbb{A})$ be the set of all probability density functions $\rho$ on $\mathbb{A}$. The Shannon entropy of $\rho\in\mathscr{P}_0(\mathbb{A})$ is given by $\mathcal{H}(\rho) := -\int_{\mathbb{A}} \rho(a) \ln \rho(a) da$. For any measurable $\phi: [0, T] \times \mathbb{R}^d \times \mathbb{A} \to \mathbb R$ (or $\mathbb R^d$) and $\rho\in\mathscr{P}_0(\mathbb{A})$, we define
	\begin{equation*}
		\tilde{\phi}(t, x, \rho) := \int_{\mathbb{A}} \phi(t, x, a) \rho(a) da, \quad (t, x) \in [0, T] \times \mathbb{R}^d.
	\end{equation*}

Let $\Pi$ be the set of exploratory (relaxed) policies $\pi: [0, T] \times \mathbb{R}^d \to \mathscr{P}_0(\mathbb{A})$. For any $(t, x) \in [0, T] \times \mathbb{R}^d$ and $\pi \in \Pi$, the exploratory state dynamics $\{X_s^\pi\}_{s \in [t, T]}$ is governed by the SDE
\begin{equation}\label{sde_1}
	dX_s^\pi = \tilde{b}(s, X_s^\pi, \pi(s, X_s^\pi)) ds + \sigma(s, X_s^\pi) dW_s, \quad X_t^\pi = x \in \mathbb{R}^d,
\end{equation}
where $b: [0, T] \times \mathbb{R}^d \times \mathbb{A} \to \mathbb{R}^d$ and $\sigma: [0, T] \times \mathbb{R}^d \to \mathbb{R}^{d \times d}$ are measurable functions. 

Now, consider the objective function\footnote{The objective function \eqref{J^pi} is chosen to be 
in line with standard problem formulation under exponential discounting (e.g., \cite{wang2020reinforcement_continuous_time,jia2022policy_gradient_continuous_time,tang2022exploratory_hjb_sicon,huang2025convergence,ma2025convergence}). In fact, our analysis in this paper readily applies to an objective function of the form
	\[
	J^{\pi}(t,x) := \mathbb{E}_{t,x} \Big[ \int_{t}^{T} \big( \tilde r(t,x,s,X_s^\pi,\pi(s,X_s^\pi)) + \lambda \mathcal H(\pi(s,X_s^\pi)) \big)\, ds + F(t,x,X_T^\pi) \Big]
	+ G\big(t,x,\mathbb{E}_{t,x}[h(X_T^\pi)]\big).
	\]
}
\begin{align}
	J^{\pi}(t,x) := \mathbb{E}_{t,x} &\left[ \int_{t}^{T} \delta(s-t) \left( \tilde{r}(x, s, X^\pi_s, \pi(s, X_s^\pi)) + \lambda \mathcal{H}(\pi(s, X_s^\pi)) \right) ds + F(t, x, X^\pi_T) \right] \notag\\
	&+ G(t, x, \mathbb{E}_{t,x}[h(X^\pi_T)]),\label{J^pi}
\end{align}
where $\delta: [0, T] \to (0, \infty)$ is a general discount function (i.e., decreasing with $\delta(0)=1$), $\lambda > 0$ is an exogenous temperature parameter, and $r:\mathbb{R}^d \times [0, T] \times \mathbb{R}^d \times \mathbb{A}\to \mathbb R$, 
$F:[0, T] \times \mathbb{R}^d \times \mathbb{R}^d\to \mathbb R$, $G:[0, T] \times \mathbb{R}^d \times \mathbb{R}^m\to\mathbb R$, and $h:\mathbb{R}^d \to \mathbb{R}^m$ are given measurable functions.

{\color{black}
\begin{definition}[Admissible Policy]\label{def:adm}
	A policy $\pi \in \Pi$ is said to be \textit{admissible} if the following conditions hold:
\begin{enumerate}
    \item[(i)] The mapping $(t,x,a)\mapsto\pi(t,x,a)$ is Borel measurable.
    For each $a\in\mathbb A$, $\pi$ is right-continuous in $t$ and continuous
    in $x$, in the following sense:
    \[
        \lim_{\substack{(s,z)\to(t,x)\\s\in[t,T]}}
        \pi(s,z,a)=\pi(t,x,a),
        \qquad (t,x)\in[0,T]\times\mathbb R^d.
    \]

    \item[(ii)] For every $(t,x)\in[0,T)\times\mathbb R^d$,
    SDE \eqref{sde_1} admits a unique strong solution
    $\{X_s^\pi\}_{s\in[t,T]}$.

    \item[(iii)] For every $(t,x)\in[0,T]\times\mathbb R^d$,
    \begin{align*}
        \mathbb E_{t,x}\Bigg[
        &\int_t^T \delta(s-t)\int_{\mathbb A}
        \Bigl(
            |r(x,s,X_s^\pi,a)|\pi(s,X_s^\pi,a)
            +\lambda\pi(s,X_s^\pi,a)
            |\ln\pi(s,X_s^\pi,a)|
        \Bigr)\,da\,ds\\
        &\qquad
        +|F(t,x,X_T^\pi)|+|h(X_T^\pi)|
        \Bigg]<\infty.
	\end{align*}
    \item[(iv)] For every $t\in[0,T)$, there exists
    $\tilde t\in(t,T]$ such that, for every $x\in\mathbb R^d$,
    \begin{equation}\label{eq:local_unif_bound}
        \mathbb E_{t,x}\bigg[
        \int_{\mathbb A}\sup_{s\in[t,\tilde t]}
        \left\{
            \pi(s,X_s^\pi,a)
            \bigl(1+|\ln\pi(s,X_s^\pi,a)|\bigr)
        \right\}\,da
        \bigg]<\infty.
    \end{equation}
\end{enumerate}
    The set of all admissible policies is denoted by $\mathcal{A}$.
\end{definition} }

Due to the potential non-exponential structure of $\delta$, the dependence of $r$, $F$, and $G$ on initial time and/or state $t,x$, and the potential nonlinearity of $G$ in its third variable, the stochastic control problem $\sup_{\pi\in\mathcal A}J^\pi(t,x)$ is time-inconsistent, i.e., an optimal policy $\pi\in\mathcal A$ found at $(t,x)$ may no longer be optimal at a future time-state pair $(t',x')$. In response, an agent considers an intra-personal game among her current and future selves, and looks for an equilibrium policy as defined below. 
{\color{black}
\begin{definition}[Equilibrium Policy]\label{def:equilibrium}
We say $\hat\pi\in\mathcal A$ is a (regularized) equilibrium policy
for \eqref{J^pi} if, for any $(t,x)\in[0,T)\times\mathbb R^d$
and any $\pi_0\in\mathcal A$ such that the perturbed policy
$\hat\pi_{t,\epsilon,\pi_0}$ defined by
\begin{equation}\label{perturbed}
    \hat\pi_{t,\epsilon,\pi_0}(s,z):=
    \begin{cases}
        \pi_0(s,z),&s\in[0,t+\epsilon),\\
        \hat\pi(s,z),&s\in[t+\epsilon,T],
    \end{cases}
    \qquad z\in\mathbb R^d,
\end{equation}
remains in $\mathcal A$ for all sufficiently small
$\epsilon\in(0,T-t)$, we have
\begin{equation}\label{equilibrium_con}
    \limsup_{\epsilon\downarrow0}
    \frac{
        J^{\hat\pi_{t,\epsilon,\pi_0}}(t,x)-J^{\hat\pi}(t,x)
    }{\epsilon}\leq0.
\end{equation}
\end{definition}
In Definition~\ref{def:equilibrium}, the condition $\hat\pi_{t,\epsilon,\pi_0}\in\mathcal A$ for all sufficiently small $\epsilon\in(0,T-t)$ is imposed because the integrability condition~(iii) in Definition~\ref{def:adm} may not be preserved under the concatenation \eqref{perturbed}, even if $\hat\pi,\pi_0\in\mathcal A$. Nevertheless, under Assumption~\ref{assump:regularity}, we will show in the proof of Theorem~\ref{thm:convergence} that the equilibrium policy $\pi^*$ satisfies $\pi^*_{t,\epsilon,\pi_0}\in\mathcal A$ for every $t\in[0,T)$, $\pi_0\in\mathcal A$, and $\epsilon\in(0,T-t)$.
}

\subsection{Derivation of the Exploratory Equilibrium HJB (EEHJB) Equation}\label{subsect:eehjb}
{\it Equilibrium HJB} equations, introduced in \cite{Yong2012} and recently studied in \cite{lei_nonlocal_2023, lei_nonlocality_2024, liang2025}, characterize equilibrium policies under time inconsistency induced by reward functions' dependence on initial time. 
As our objective \eqref{J^pi} further involves entropy regularization as well as time inconsistency induced by reward functions' dependence on initial state and additional nonlinearity, we will derive below a new (and much more general) equilibrium HJB equation (i.e., \eqref{eq:EEHJB} below).  
To begin with, we introduce 
\begin{equation}\label{V^pi}
   V^{\pi}(\tau, t, y, x):= V^{\pi,1}(\tau, t, y, x) + G(\tau, y, V^{\pi,2}(t, x)),  
\end{equation}
defined on $\Delta[0, T] \times \mathbb{R}^d \times \mathbb{R}^d$, where
\begin{align}
	&V^{\pi,1}(\tau, t, y, x) := \mathbb{E}_{t,x} \left[ \int_t^T \delta(s-\tau) \left( \tilde{r}(y, s, X_s^\pi, \pi(s, X_s^\pi)) + \lambda \mathcal{H}(\pi(s, X_s^\pi)) \right) ds + F(\tau, y, X_T^\pi) \right], \notag\\
	&V^{\pi,2}(t, x) := \mathbb{E}_{t,x} [h(X_T^\pi)].\label{V^1, V^2}
\end{align}
By construction, the objective function \eqref{J^pi} is recovered on the diagonal: $J^\pi(t, x) = V^\pi(t, t, x, x)$.
By \eqref{J^pi} and the law of iterated expectations, under the perturbed policy $\hat{\pi}_{t, \epsilon, \pi_0}$ in \eqref{perturbed}, 
\begin{align*}
	J^{\hat{\pi}_{t, \epsilon, \pi_0}}(t, x) &= \mathbb{E}_{t,x} \bigg[ \int_t^{t+\epsilon} \delta(s-t) \left( \tilde{r}(x, s, X_s^{\pi_0}, \pi_0(s, X_s^{\pi_0})) + \lambda \mathcal{H}(\pi_0(s, X_s^{\pi_0})) \right) ds \\
	&\quad + V^{\hat{\pi},1}(t, t+\epsilon, x, X_{t+\epsilon}^{\pi_0}) +  G\left(t, x, \mathbb{E}_{t,x}\left[V^{\hat{\pi},2}(t+\epsilon, X_{t+\epsilon}^{\pi_0})\right]\right)\bigg].
\end{align*} 
Assume heuristically that the auxiliary functions $V^{\hat{\pi},1}$ and $V^{\hat{\pi},2}$ possess sufficient regularity with respect to (w.r.t.) the \textit{flow variables} (i.e., the second and fourth arguments for $V^{\hat{\pi},1}$, and the first and second for $V^{\hat{\pi},2}$), so that we can apply Itô's formula to $V^{\hat{\pi},1}(t, \cdot, x, \cdot)$ and $ V^{\hat{\pi},2}(\cdot, \cdot)$ along the state process $X^{\pi_0}$ over the interval $[t, t+\epsilon]$. Utilizing the first-order Taylor expansion for the function $G(t,x,\cdot)$ and taking the limit as $\epsilon \downarrow 0$, we exploit the condition \eqref{equilibrium_con} to get the inequality
{\small
\begin{align}
	&\mathcal{I}(t, x, \pi_0) := V_t^{\hat{\pi},1}(t, t, x, x) + \tilde{b}(t, x, \pi_0(t, x)) \cdot  V_x^{\hat{\pi},1}(t, t, x, x) + \frac{1}{2} \text{tr}\left(\sigma \sigma^\top(t, x)  V_{xx}^{\hat{\pi},1}(t, t, x, x)\right) \nonumber \\
	&\quad + \tilde{r}(x, t, x, \pi_0(t, x)) + \lambda \mathcal{H}(\pi_0(t, x)) \nonumber \\
	&\quad + G_z(t, x, V^{\hat{\pi},2}(t, x))\cdot \left[  V_t^{\hat{\pi},2}(t, x) + \tilde{b}(t, x, \pi_0(t, x)) \cdot  V_x^{\hat{\pi},2}(t, x) + \frac{1}{2} \text{tr}\left(\sigma \sigma^\top(t, x)  V_{xx}^{\hat{\pi},2}(t, x)\right) \right] \leq 0,\label{I} 
\end{align}
}where $V_t^{\hat{\pi},1}$ denotes the partial derivative w.r.t.\ the second argument, while $V_x^{\hat{\pi},1}$ and $V_{xx}^{\hat{\pi},1}$ denote the gradient and the Hessian matrix w.r.t.\ the fourth argument, respectively. {\color{black}Moreover, the term $\tilde{r}(x, t, x, \pi_0(t, x)) + \lambda \mathcal{H}(\pi_0(t, x))$ in \eqref{I} is a result of
{\small\begin{align*}
    \lim_{\epsilon\downarrow 0}\frac{1}{\epsilon}\mathbb{E}_{t,x} \bigg[ \int_t^{t+\epsilon} \delta(s-t) \left( \tilde{r}(x, s, X_s^{\pi_0}, \pi_0(s, X_s^{\pi_0})) + \lambda \mathcal{H}(\pi_0(s, X_s^{\pi_0})) \right) ds\bigg]=\tilde{r}(x, t, x, \pi_0(t, x)) + \lambda \mathcal{H}(\pi_0(t, x)),
\end{align*}}which is induced by the continuity and integrability conditions in Definition~\ref{def:adm}~(i) and (iv) and the dominated convergence theorem.}

Note that  $\mathcal{I}(t, x, \pi_0)$ depends on  $\pi_0$ only through its value $\pi_0(t, x) \in \mathscr{P}_0(\mathbb{A})$. Furthermore, it is evident that $\mathcal{I}(t, x, \hat{\pi}) = 0$. As a result, if the admissible policy $\hat{\pi}$ is chosen such that it pointwise maximizes $\mathcal{I}(t, x, \pi_0)$ over $\mathscr{P}_0(\mathbb{A})$, the condition $\mathcal{I}(t, x, \pi_0) \leq 0$ is satisfied for any $\pi_0$, so that the definition of an equilibrium policy (i.e., \eqref{equilibrium_con}) is fulfilled. Solving this maximization problem leads to an {equilibrium policy in the form} of a \textit{Gibbs measure}:
\begin{equation}
	\hat{\pi}(t, x)(a) = \Gamma(t, x, Z(t,x), a) := \frac{\exp\left( \frac{1}{\lambda} \left[b(t, x, a) \cdot Z(t,x) + r(x, t, x, a)\right] \right)}{\int_{\mathbb{A}} \exp\left( \frac{1}{\lambda} \left[b(t, x, a') \cdot Z(t,x) + r(x, t, x, a')\right] \right) da'}, \label{eq:gibbs_policy}
\end{equation}
where 
\begin{equation}\label{Z}
	Z(t,x) := V_x^{\hat{\pi},1}(t, t, x, x) + G_z(t, x, V^{\hat{\pi},2}(t, x)) \cdot V_x^{\hat{\pi},2}(t, x).
\end{equation}
   Under the equilibrium policy $\hat{\pi}$, the pair $(V^{\hat{\pi},1},V^{\hat{\pi},2})$ satisfies the coupled PDE system
\begin{equation}\label{eq:EEHJB}
	\begin{cases}
		 V_t^{\hat{\pi},1}(\tau, t, y, x) + \frac{1}{2} \text{tr}\left(\sigma \sigma^\top(t, x)  V_{xx}^{\hat{\pi},1}(\tau, t, y, x)\right)  + \tilde{b}(t, x, \hat{\pi}(t,x)) \cdot \left[  V_x^{\hat{\pi},1}(\tau, t, y, x) - \delta(t-\tau) Z(t, x) \right] \\
		\qquad\qquad\qquad\, + \delta(t-\tau)\left[ H(t, x, Z(t, x)) + \left( \tilde{r}(y, t, x,  \hat{\pi}(t, x)) - \tilde{r}(x, t, x,  \hat{\pi}(t, x)) \right)\right] = 0, \\
	    V_t^{\hat{\pi},2}(t, x)  + \frac{1}{2} \text{tr}\left(\sigma \sigma^\top(t, x)  V_{xx}^{\hat{\pi},2}(t, x)\right)+ \tilde{b}(t, x, \hat{\pi}(t, x)) \cdot  V_x^{\hat{\pi},2}(t, x) = 0, \\
		V^{\hat{\pi},1}(\tau, T, y, x) = F(\tau, y, x), \quad V^{\hat{\pi},2}(T, x) = h(x),
	\end{cases}
\end{equation}
where 
\begin{equation}\label{hfunction}
	H(t, x, z) := \lambda \ln \left( \int_{\mathbb{A}} \exp\left( \frac{1}{\lambda} \left[ b(t, x, a) \cdot z + r(x, t, x, a) \right] \right) da \right).
\end{equation}
We call the system \eqref{eq:EEHJB} the {\it exploratory equilibrium HJB (EEHJB)} equation. Note that it is inherently non-local, as $V_1^{\hat{\pi}}(\tau, t, y, x)$ explicitly depends on its diagonal values at $(t, t, x, x)$ through $Z(t, x)$ in \eqref{Z}. Using \eqref{hfunction}, it can be checked that
\[H_{z}(t,x,Z(t,x))=\tilde{b}(t,x,\hat{\pi}(t,x)).\]
Consequently, the EEHJB equation \eqref{eq:EEHJB} can be rewritten as
\begin{equation*}
	\begin{cases}
		V_t^{\hat{\pi},1}(\tau, t, y, x) + \frac{1}{2} \text{tr}\left(\sigma \sigma^\top(t, x) V_{xx}^{\hat{\pi},1}(\tau, t, y, x)\right) + H_{z}(t, x, Z(t, x)) \cdot \left[ V_x^{\hat{\pi},1}(\tau, t, y, x) - \delta(t-\tau) Z(t, x) \right] \\
		\qquad\qquad\qquad\, + \delta(t-\tau)\left[ H(t, x, Z(t, x)) + \int_{\mathbb{A}} \left( r(y, t, x, a) - r(x, t, x, a) \right) \Gamma(t, x, Z(t, x), a) da\right] = 0, \\
		V_t^{\hat{\pi},2}(t, x)  + \frac{1}{2} \text{tr}\left(\sigma \sigma^\top(t, x) V_{xx}^{\hat{\pi},2}(t, x)\right)+ H_{z}(t, x, Z(t, x))\cdot V_x^{\hat{\pi},2}(t, x) = 0, \\
		V^{\hat{\pi},1}(\tau, T, y, x) = F(\tau, y, x), \quad V^{\hat{\pi},2}(T, x) = h(x).
	\end{cases}
\end{equation*}

\begin{remark}[Relation to the extended HJB equation in \cite{Bjork2017, BjorkKhapkoMurgoci2021}]\label{rem:bjork}
	If the auxiliary functions $V^{\hat{\pi},1}$ and $V^{\hat{\pi},2}$ are sufficiently regular (i.e., $V^{\hat{\pi},1}$ is $C^{1,2}$ w.r.t.\ the reference variables $\tau$ and $y$), the equilibrium value function $J^{\hat{\pi}}(t,x) = V^{\hat{\pi},1}(t, t, x, x) + G(t, x, V^{\hat{\pi},2}(t, x))$ will satisfy the equation
	\begin{align} \label{eq:Bjork_structure}
		\sup_{\pi \in \mathscr{P}_0(\mathbb{A})} \bigg\{ & \mathcal{L}^\pi J^{\hat{\pi}}(t, x) + \tilde{r}(x, t, x, \pi) + \lambda \mathcal{H}(\pi) \nonumber \\
		& - \Big( \mathcal{L}^\pi V^{\hat{\pi},1}_{\text{diag}}(t, x) - \mathcal{L}^\pi [V^{\hat{\pi},1}_{t,x}](t, x) \Big) 
		- \Big( \mathcal{L}^\pi G^\diamond(t, x) - \mathcal{G}^\pi[V^{\hat{\pi},2}](t, x) \Big) \bigg\} = 0,
	\end{align}
    with the terminal condition
        $J^{\hat{\pi}}(T, x) = F(T, x, x) + G(T, x, h(x))$,
	where
	\begin{enumerate}
		\item $\mathcal{L}^\pi$, the infinitesimal generator acting on flow variables $(t, x)$, is given by 
		\[ \mathcal{L}^\pi \phi(t,x) := \partial_t \phi(t,x) + \tilde{b}(t, x, \pi) \cdot \partial_x \phi(t,x) + \frac{1}{2} \text{tr}\left(\sigma \sigma^\top(t, x) \partial_{xx}^2 \phi(t,x)\right),\quad \forall \phi \in C^{1,2}([0,T]\times\mathbb{R}^d); \]
		\item $V^{\hat{\pi},1}_{\text{diag}}(t, x) := V^{\hat{\pi},1}(t, t, x, x)$, and the term $\mathcal{L}^\pi V^{\hat{\pi},1}_{\text{diag}}(t,x)$ involves total derivatives, including derivatives w.r.t.\ the reference variables $\tau$ and $y$ evaluated at the diagonal $(\tau,y)=(t,x)$;
		\item $V^{\hat{\pi},1}_{\tau,y}(t, x) := V^{\hat{\pi},1}(\tau, t, y, x)$ represents the function with fixed reference variables $\tau$ and $y$, and  $\mathcal{L}^\pi[V^{\hat{\pi},1}_{t,x}](t, x)$ means that the generator $\mathcal{L}^\pi$ acts  on only the flow variables $(t, x)$, with the reference variables fixed at $(\tau, y) = (t, x)$;
		\item $G^\diamond(t, x) := G(t, x, V^{\hat{\pi},2}(t, x))$ and  
        $\mathcal{G}^\pi[V^{\hat{\pi},2}](t, x) := G_z(t, x, V^{\hat{\pi},2}(t, x)) \cdot \mathcal{L}^\pi V^{\hat{\pi},2}(t, x)$. 
	\end{enumerate}
	The PDE system consisting of \eqref{eq:Bjork_structure} and \eqref{eq:EEHJB}, jointly for $(J^{\hat \pi}, V^{\hat{\pi},1}, V^{\hat{\pi},2})$, corresponds precisely to the extended HJB equation in \cite[Definition 15.6]{BjorkKhapkoMurgoci2021}. As \eqref{eq:Bjork_structure} involves differentiation w.r.t.\ the reference parameters $\tau$ and $y$ (see point 1.\ above), the extended HJB equation is inherently non-local. However, the issue of non-localness has not been widely recognized or discussed in \cite{Bjork2017, BjorkKhapkoMurgoci2021} and related studies (e.g., \cite{EP08,EMP12, BMZ14, dong2014time, HS26}, among others). This is because in this line of literature, the extended HJB equation is typically solved by making a clever ansatz of $(J^{\hat \pi}, V^{\hat{\pi},1}, V^{\hat{\pi},2})$ (designed for the specific financial/economic problem at hand), which greatly simplifies the extended HJB equation to the extent that non-localness disappears. 

    For the general time-inconsistent problem \eqref{J^pi}, as there is no way of making a clever simplifying ansatz, we have to confront the non-localness inherent in \eqref{eq:Bjork_structure} and \eqref{eq:EEHJB}. Notably, as we will find out, it suffices to focus on only \eqref{eq:EEHJB}, which yields some technical advantages (see Remark~\ref{remark:adv}); moreover, policy iteration can be used to circumvent the non-localness (see Remark~\ref{rem:PIA non-local}).  
    \end{remark}

\subsection{Policy Iteration Algorithm}\label{subsect:pia}
In the standard time-consistent case, the policy iteration algorithm (PIA) successfully converges to an optimal policy under entropy regularization; see e.g., \cite{huang2025convergence, tran2025policy, ma2025convergence}. In the present time-inconsistent setting, we aim to develop a corresponding PIA that can lead us to an equilibrium policy.  

In view of \eqref{V^1, V^2} and \eqref{eq:gibbs_policy}-\eqref{Z}, we initiate our PIA from sufficiently smooth initial functions $(V^{0,1}, V^{0,2})$; next, for $n=0,1,2,\cdots$, we recursively follow the two-step loop:
\begin{enumerate}
	\item \textbf{Policy Update:} 
	Given the current iterates $(V^{n,1}, V^{n,2})$, the updated policy $\pi^{n+1}$ is 
	\begin{equation}\label{updated}
		\pi^{n+1}(t, x)(a) = \Gamma\left(t, x, Z^n(t, x), a\right), \quad \forall (t, x, a) \in [0, T] \times \mathbb{R}^d \times \mathbb{A},
	\end{equation}
	where $\Gamma$ is defined as in \eqref{eq:gibbs_policy}, with
	\begin{align*}
		Z^n(t, x) :=  V_x^{n,1}(t, t, x, x) + G_z(t, x, V^{n,2}(t, x))\cdot  V_x^{n,2}(t, x).
	\end{align*}
	
	\item \textbf{Policy Evaluation:} 
    Compute the next pair of auxiliary functions $(V^{n+1,1}, V^{n+1,2})$ via 
	\begin{align}
		V^{n+1,1}(\tau, t, y, x) &= \mathbb{E}_{t,x} \bigg[ \int_t^T \delta(s-\tau) \left( \tilde{r}(y, s, X^{n+1}_s, \pi^{n+1}(s, X^{n+1}_s))+ \lambda \mathcal{H}(\pi^{n+1}(s, X^{n+1}_s)) \right) ds \notag\\
		&\qquad + F(\tau, y, X^{n+1}_T) \bigg], \notag\\
		V^{n+1,2}(t, x) &= \mathbb{E}_{t,x} [h(X^{n+1}_T)],\label{V^1, V^2, n}
	\end{align}
	where $\{X^{n+1}_s\}_{s\in[t,T]}$ evolves according to \eqref{sde_1} under the updated policy $\pi^{n+1}$.
\end{enumerate}
The goal of this paper is to establish the convergence of the PIA to an equilibrium policy with a suitable convergence rate. As a preparation, we will start with a regularity result for the PIA.

\section{A Regularity Result}\label{sec:reg}
Before we can investigate the convergence of our PIA (the focus of Section~\ref{sec:convergence}), we need to first ensure that all generated policies $\{\pi^n\}_{n\ge 1}$ are admissible and appropriate regularity can be maintained across all the iterates  $\{(V^{n,1}, V^{n,2})\}_{n\ge 1}$. To this end, we impose the following. 

\begin{assumption} \label{assump:regularity}
The following regularity conditions hold:
	\begin{enumerate}
		\item[(i)] For any fixed $a \in \mathbb{A}$ and $y \in \mathbb{R}^d$, the function $\phi \in \{b(\cdot, \cdot, a), r(y, \cdot, \cdot, a), r_y(y, \cdot, \cdot, a), \sigma(\cdot, \cdot)\}$ belongs to $C^{2}([0, T] \times \mathbb{R}^d)$ and there exists $C_0>0$, independent of $(a,y)$, such that $\|\phi\|^{(2)} \leq C_0$. {For each fixed $a \in \mathbb{A}$, $r(\cdot, \cdot, \cdot, a)$, $r_t(\cdot, \cdot, \cdot, a)$, $r_x(\cdot, \cdot, \cdot, a)$, and $r_y(\cdot, \cdot, \cdot, a)$ are jointly continuous in $(y, t, x)$.}
		\item[(ii)]  $\sigma$ is uniformly non-degenerate, i.e., $\sigma \sigma^\top(t, x) \geq \frac{1}{C_0} I_d$ for all $(t, x) \in [0, T] \times \mathbb{R}^d$. 
	\item[(iii)] The derivative  $\phi \in \{G_t, G_z, G_{zz}, G_{zt}, G_{zx}, G_{zxz}, G_{zzz}\}$ of $G(t, x, z)$  is continuous and locally bounded in $z$ uniformly in $(t, x)$; specifically, $|\phi(t, x, z)| \leq C(z)$ for all $(t, x, z) \in [0, T] \times \mathbb{R}^d \times \mathbb{R}^m$, where $C(\cdot)$ is a locally bounded function (i.e., bounded on any compact subset of $\mathbb{R}^m$).
		\item[(iv)] There exists $\alpha \in (0, 1)$ such that $F \in \Xi_{[0, T]}^{(2+\alpha)}$\footnote{Here $F(\tau, y, x)$ is independent of the flow time $t$, with $(\tau, y)$ acting as the reference variables.} and $h \in C^{2+\alpha}(\mathbb{R}^d)$.
      \item[(v)] The discount function $\delta$ belongs to $C^2([0, T])$. 
	\end{enumerate}
\end{assumption}
{\color{black}Under Assumption~\ref{assump:regularity}, the functions $r$, $F$, and $h$ are uniformly bounded. Since $\delta$ is continuous and strictly positive on $[0,T]$, a policy $\pi\in\Pi$ is admissible if and only if it satisfies Definition~\ref{def:adm}~(i), (ii), and (iv), along with
\begin{equation}\label{eq:entropy_integrability}
    \mathbb{E}_{t,x}\left[
        \int_t^T\int_{\mathbb{A}}
        \pi(s,X_s^\pi,a)|\ln\pi(s,X_s^\pi,a)|\,da\,ds
    \right]<\infty,
    \qquad (t,x)\in[0,T)\times\mathbb{R}^d.
\end{equation}
}
By straightforward calculations, Assumption \ref{assump:regularity} directly implies the following estimates of $H$ in \eqref{hfunction} and $\Gamma$ in \eqref{eq:gibbs_policy}.

\begin{lemma} \label{lemma:H_estimates}
	Under Assumption~\ref{assump:regularity}, the following properties hold {\color{black}for a constant $C>0$ depending only on $d$, $C_0$, $\lambda$, and $|\mathbb{A}|$}:
	\begin{itemize}
		\item[(i)] The derivatives of $H$, including $H_t, H_x, H_z, H_{zz}, H_{xz}, H_{tz}, H_{zzx}, H_{zzz}$, are continuous and satisfy
		\begin{align*}
			& |H_z| \leq C, \quad 0 \leq H_{zz} \leq C I_d, \quad |H_{zzz}| \leq C, \\
			& [ |H| + |H_t| + |H_x| + |H_{xz}| + |H_{tz}| + |H_{zzx}| ](t, x, z) \leq C(1 + |z|).
		\end{align*}
		\item[(ii)] The derivatives of $\Gamma$, including $\Gamma_t, \Gamma_x, \Gamma_z, \Gamma_{zz}, \Gamma_{xz}$, are continuous in $(t, x, z)$ and satisfy
		\begin{align*}
			 |\Gamma_z| + |\Gamma_{zz}| \leq C \Gamma,\qquad |\Gamma_t| + |\Gamma_x| + |\Gamma_{xz}|\leq C(1 + |z|) \Gamma.
		\end{align*}
	\end{itemize}
\end{lemma}
{\color{black}\begin{proof}
It follows from \eqref{eq:gibbs_policy} and \eqref{hfunction} that
\begin{equation}\label{eq:ln_Gamma}
	\lambda \ln \Gamma(t, x, z, a) = b(t, x, a) \cdot z + r(x, t, x, a) - H(t, x, z), \qquad \int_{\mathbb{A}} \Gamma(t, x, z, a) \, da = 1.
\end{equation}
Since $|b(t, x, a) \cdot z + r(x, t, x, a)| \leq C_0(1 + |z|)$ by Assumption~\ref{assump:regularity}~(i), it holds that
\[
	|H(t, x, z)| \leq C_0(1 + |z|) + \lambda \big|\ln |\mathbb{A}|\big| \leq C(1 + |z|).
\]
Under Assumption~\ref{assump:regularity} and $0 < |\mathbb{A}| < \infty$, differentiation under the integral for $H$ in \eqref{hfunction} is justified by the dominated convergence theorem. The same argument, together with the strict positivity of the normalizing integral, gives continuity of the resulting derivatives of $H$. The continuity of the derivatives of $\Gamma$ then follows directly from \eqref{eq:ln_Gamma} and Assumption~\ref{assump:regularity}. It therefore remains to establish the bounds in (i) and (ii).

Differentiating \eqref{hfunction} and \eqref{eq:ln_Gamma} with respect to $z_i$ ($i \in \{1, \ldots, d\}$) yields
\[
	H_{z_i} = \int_{\mathbb{A}} b_i \Gamma \, da, \qquad \Gamma_{z_i} = \frac{1}{\lambda} (b_i - H_{z_i}) \Gamma.
\]
Similarly, differentiation with respect to $t$ and $x_k$ ($k \in \{1, \ldots, d\}$) gives
\[
\begin{aligned}
	H_t &= \int_{\mathbb{A}} (b_t \cdot z + r_t) \Gamma \, da, \qquad &H_{x_k} &= \int_{\mathbb{A}} (b_{x_k} \cdot z + r_{y_k} + r_{x_k}) \Gamma \, da, \\
	\Gamma_t &= \frac{1}{\lambda} (b_t \cdot z + r_t - H_t) \Gamma, \qquad &\Gamma_{x_k} &= \frac{1}{\lambda} (b_{x_k} \cdot z + r_{y_k} + r_{x_k} - H_{x_k}) \Gamma.
\end{aligned}
\]
By Assumption~\ref{assump:regularity}~(i), these identities readily imply
\[
	|H_z| \leq C, \quad |H_t| + |H_x| \leq C(1 + |z|), \qquad |\Gamma_z| \leq C \Gamma, \quad |\Gamma_t| + |\Gamma_x| \leq C(1 + |z|) \Gamma.
\]

As $\int_{\mathbb{A}} \Gamma \, da = 1$ in \eqref{eq:ln_Gamma}, $H_{z_i} = \int_{\mathbb{A}} b_i \Gamma \, da$ gives
\[
	\int_{\mathbb{A}} (b_i - H_{z_i}) \Gamma \, da = 0, \qquad i \in \{1, \ldots, d\}.
\]
Using this equality, we obtain, for $i, j, k \in \{1, \ldots, d\}$,
\[
	H_{z_i z_j} = \frac{1}{\lambda} \int_{\mathbb{A}} (b_i - H_{z_i})(b_j - H_{z_j}) \Gamma \, da, \qquad H_{z_i z_j z_k} = \frac{1}{\lambda^2} \int_{\mathbb{A}} (b_i - H_{z_i})(b_j - H_{z_j})(b_k - H_{z_k}) \Gamma \, da.
\]
The first expression shows that $H_{zz}$ is positive semidefinite. As $b$ and $H_z$ are bounded, it follows from these identities that
\[
	0 \leq H_{zz} \leq C I_d, \qquad |H_{zzz}| \leq C.
\]
Moreover, differentiating $H_{z_i}$ with respect to $t$ and $x_k$ gives
\[
\begin{aligned}
	H_{z_i t} &= \int_{\mathbb{A}} (b_i)_t \Gamma \, da + \frac{1}{\lambda} \int_{\mathbb{A}} (b_i - H_{z_i}) (b_t \cdot z + r_t - H_t) \Gamma \, da, \\
	H_{z_i x_k} &= \int_{\mathbb{A}} (b_i)_{x_k} \Gamma \, da + \frac{1}{\lambda} \int_{\mathbb{A}} (b_i - H_{z_i}) (b_{x_k} \cdot z + r_{y_k} + r_{x_k} - H_{x_k}) \Gamma \, da,
\end{aligned}
\]
which implies that $|H_{tz}| + |H_{xz}| \leq C(1 + |z|)$. Differentiating the formula of $H_{z_i z_j}$ with respect to $x_k$, and again using $\int_{\mathbb{A}} (b_i - H_{z_i}) \Gamma \, da = 0$, we obtain
\[
\begin{aligned}
	H_{z_i z_j x_k} &= \frac{1}{\lambda} \int_{\mathbb{A}} \left[ (b_i)_{x_k}(b_j - H_{z_j}) + (b_i - H_{z_i}) (b_j)_{x_k} \right] \Gamma \, da + \frac{1}{\lambda} \int_{\mathbb{A}} (b_i - H_{z_i})(b_j - H_{z_j}) \Gamma_{x_k} \, da.
\end{aligned}
\]
Thus, $|H_{zzx}| \leq C(1 + |z|)$, which completes the estimates in (i).

Finally, differentiating the formula for $\Gamma_{z_i}$ yields
\[
	\Gamma_{z_i z_j} = \left[ \frac{(b_i - H_{z_i})(b_j - H_{z_j})}{\lambda^2} - \frac{H_{z_i z_j}}{\lambda} \right] \Gamma, \qquad \Gamma_{z_i x_k} = \frac{1}{\lambda}\big((b_i)_{x_k} - H_{z_i x_k}\big) \Gamma + \frac{1}{\lambda}(b_i - H_{z_i}) \Gamma_{x_k}.
\]
Combining these expressions with Assumption~\ref{assump:regularity}~(i) and the bounds on $H_z, H_{zz}, H_{xz}$, and $\Gamma_x$, we arrive at 
\[
	|\Gamma_{zz}| \leq C \Gamma, \qquad |\Gamma_{xz}| \leq C(1 + |z|) \Gamma,
\]
which completes the proof.
\end{proof}
}

The next result ensures that our PIA is well-defined, by showing that the generated policies $\{\pi^n\}_{n\ge 1}$ are all admissible. Moreover, it establishes appropriate regularity of the generated values $\{(V^{n,1}, V^{n,2})\}_{n \geq 1}$ across all iterations.

\begin{prop}\label{prop:1}
	Assume Assumption \ref{assump:regularity}. For any 
    \begin{equation}\label{V^0 condition}
    (V^{0,1}, V^{0,2}) \in (\Xi_{[0, T]}^{(2+\alpha)}\cap\Theta^{(2)}_{[0,T]})\times C^{2+\alpha}([0, T] \times \mathbb{R}^d)\quad\hbox{with}\quad \hbox{$V_{x\tau}^{0,1} = V_{\tau x}^{0,1}$ and $V_{xy}^{0,1} = V_{yx}^{0,1}$}, 
    \end{equation}
    the PIA (Section~\ref{subsect:pia}) fulfills the following: 
    \begin{itemize}[leftmargin=0.4in]
    \item [(i)] For each $n \geq 0$, we have $\pi^{n+1}\in\mathcal A$ and 
    \begin{align}
    &(V^{n+1,1}, V^{n+1,2})\in (\Xi_{[0, T]}^{(2+\alpha)}\cap\Theta^{(2)}_{[0,T]}) \times C^{2+\alpha}([0, T] \times \mathbb{R}^d),\label{V^n condition}\\
    &\hspace{0.5in}\hbox{$V_{x\tau}^{n+1,1} = V_{\tau x}^{n+1,1}$,  $V_{xy}^{n+1,1} = V_{yx}^{n+1,1}$}.\label{V^n condition'} 
    \end{align}
   Moreover, $(V^{n+1,1}, V^{n+1,2})$ is the unique solution in {$(\Xi_{[0, T]}^{(2+\alpha)}\cap\Theta^{(2)}_{[0,T]}) \times C^{2+\alpha}([0, T] \times \mathbb{R}^d)$} to 
	\begin{equation} \label{eq:recursive}
		\begin{cases}
			 V_t^{n+1,1}(\tau, t, y, x) + \frac{1}{2} \text{tr}\left(\sigma \sigma^\top(t, x)  V_{xx}^{n+1,1}(\tau, t, y, x)\right)  + H^n_{z}(t,x) \cdot \left[  V_x^{n+1,1}(\tau, t, y, x) - \delta(t-\tau) Z^n(t,x) \right]  \\
		\qquad\qquad\qquad\quad + \delta(t-\tau)\left[ H^n(t, x) + \int_{\mathbb{A}} \left(r(y, t, x, a) - r(x, t, x, a) \right)\Gamma^n(t,x,a)da \right] = 0, \\
			 V_t^{n+1,2}(t, x) + \frac{1}{2} \text{tr}\left(\sigma \sigma^\top(t, x)  V_{xx}^{n+1,2}(t, x)\right) + H^n_{z}(t,x) \cdot  V_{x}^{n+1,2}(t, x) = 0, \\
			V^{n+1,1}(\tau, T, y, x) = F(\tau, y, x), \quad V^{n+1,2}(T, x) = h(x),
		\end{cases}
	\end{equation}  
    which is a (recursive) linear PDE system that involves
	\begin{align*}
		&Z^n(t,x) :=  V_{x}^{n,1}(t, t, x, x) + G_z(t, x, V^{n,2}(t, x))\cdot  V_x^{n,2}(t, x), \\
		&H^n(t,x) := H(t, x, Z^n(t,x)), \quad
		H^n_z(t,x) := H_z(t, x, Z^n(t,x)), \\
		&\Gamma^n(t,x,a) := \Gamma(t, x, Z^n(t,x), a).
	\end{align*}
    \item [(ii)] The sequence $\{(V^{n,1}, V^{n,2})\}_{n \geq 1}$ is uniformly bounded in the sense that {\color{black}  there exists a constant
$C(T)>0$ such that
	\begin{equation}\label{uniformbounded}
		\sup_{n \geq 1} \left( [V^{n,1}]_{[0,T]}^{(2)} + [V_{xy}^{n,1}]_{[0,T]}^{(0)} + \|V^{n,2}\|^{(2)} \right) \leq C(T).
	\end{equation}}
    \end{itemize}
\end{prop}


\subsection{Proof of Proposition~\ref{prop:1} (i)}
{\bf Step 1: The case $n=0$.} {\color{black}Let $C > 0$ denote a generic constant that may change from line to line, which depends only on $T, d, m, \alpha, \lambda, |\mathbb{A}|, C_0$, the norms $\|F\|_{[0, T]}^{(2+\alpha)}$, $\|h\|^{(2+\alpha)}$, $\|\delta\|^{(2)}$, $\|V^{0,1}\|_{[0, T]}^{(2+\alpha)}$, $\|V^{0,2}\|^{(2+\alpha)}$, and the bounds in Assumption~\ref{assump:regularity}~(iii) on $|z| \leq \|V^{0,2}\|^{(0)}$.} We first show that $\pi^1 \in \mathcal{A}$. {\color{black}As $0 < |\mathbb{A}| < \infty$} and the functions $b, r$ are bounded, the policy $\pi^1(t, x, a)=\Gamma(t, x, Z^0(t, x), a)$ in \eqref{updated} is well-defined. With $(V^{0,1}, V^{0,2}) \in \Theta_{[0,T]}^{(2)} \times C^{2+\alpha}([0, T] \times \mathbb{R}^d)$, $Z^0(t, x)$ is continuous in $(t, x)$, which implies that $\pi^1$ is jointly continuous in $(t, x)$ for each $a \in \mathbb{A}$. {\color{black}Under Assumption~\ref{assump:regularity}, in view of Definition~\ref{def:adm}, it suffices to prove that \eqref{sde_1} has a unique strong solution under $\pi^1$, and that both \eqref{eq:entropy_integrability} and \eqref{eq:local_unif_bound} hold for $\pi^1$.} 

By the chain rule, 
	$\partial_x \Gamma^n(t, x, a) = \Gamma_x(t, x, Z^n(t,x), a) + \Gamma_z(t, x, Z^n(t,x), a)\cdot Z_x^n(t, x)$.
This, along with the estimates for $\Gamma_x$ and $\Gamma_z$ in Lemma \ref{lemma:H_estimates}, implies
\begin{equation} \label{eq:Gamma_x_est}
	|\partial_x \Gamma^n(t, x, a)| \leq C \left( 1 + |Z^n(t,x)| + |Z_x^n(t,x)| \right) \Gamma^n(t, x, a).
\end{equation}
For any $x, x' \in \mathbb{R}^d$, it follows from Assumption \ref{assump:regularity} that
\begin{align*}
	|\tilde{b}(t, x, \pi^1(t,x)) &- \tilde{b}(t, x', \pi^1(t,x'))| 
	\leq \int_{\mathbb{A}} \left| b(t, x, a)\Gamma^0(t, x, a) - b(t, x', a)\Gamma^0(t, x', a) \right| da \\
	&\leq \int_{\mathbb{A}} |b(t, x, a) - b(t, x', a)| \Gamma^0(t, x, a) da 
    + \int_{\mathbb{A}} |b(t, x', a)| \left| \Gamma^0(t, x, a) - \Gamma^0(t, x', a) \right| da \\
	&\leq C|x-x'| + C \int_{\mathbb{A}} \left| \int_{0}^{1} \partial_x \Gamma^0(t, x' + s(x-x'), a) \cdot (x-x') ds \right| da.
\end{align*}
Plugging the estimate \eqref{eq:Gamma_x_est} with $n=0$ into the above integral yields
\begin{align}\label{b Lipschitz}
	|\tilde{b}(t, x, \pi^1(t,x)) - \tilde{b}(t, x', \pi^1(t,x'))| 
	&\leq C \left( 1 + \|Z^0\|^{(0)} + \|Z_x^0\|^{(0)} \right) |x-x'|.
\end{align}
Observe that
\begin{align*}
 Z_x^0(t, x) &= \left[  V_{xx}^{0,1}(t, t, x, x) + V_{xy}^{0,1}(t, t, x, x) \right] + G_z(t, x, V^{0,2}(t,x))\otimes V_{xx}^{0,2}(t,x)\\& + \left[ G_{zx}(t, x, V^{0,2}(t,x)) + G_{zz}(t, x, V^{0,2}(t,x))\cdot V_x^{0,2}(t,x) \right]\cdot V_x^{0,2}(t,x).
\end{align*}
Hence, by $(V^{0,1}, V^{0,2}) \in  \Xi_{[0, T]}^{(2+\alpha)} \times C^{2+\alpha}([0, T] \times \mathbb{R}^d)$, {$V_{xy}^{0,1} = V_{yx}^{0,1}$}, and Assumption \ref{assump:regularity} (iii), we get $\|Z^0\|^{(0)} + \|Z_x^0\|^{(0)} < \infty$. In view of \eqref{b Lipschitz}, $\tilde{b}(t, x, \pi^1(t,x))$ is Lipschitz continuous in $x$ uniformly in $t$. This, along with the regularity of $\sigma$, gives a unique strong solution to \eqref{sde_1} under $\pi^1$. {\color{black}Next, by \eqref{eq:ln_Gamma}, Assumption~\ref{assump:regularity}~(i), and Lemma~\ref{lemma:H_estimates}~(i),
\begin{align}\label{admiss_0}
	\mathbb{E}_{t,x} &\left[ \int_t^T \int_{\mathbb{A}} \pi^1(s, X_s^1, a) |\ln \pi^1(s, X_s^1, a)| \, da \, ds \right] \nonumber\\
	&= \frac{1}{\lambda} \mathbb{E}_{t,x} \left[ \int_t^T \int_{\mathbb{A}} \Gamma^0(s, X_s^1, a) \left| b(s, X_s^1, a) \cdot Z^0(s, X_s^1) + r(X_s^1, s, X_s^1, a) - H^0(s, X_s^1) \right| da \, ds \right] \nonumber\\
	&\leq C \left( 1 + \|Z^0\|^{(0)} \right) < \infty,
\end{align}
which confirms \eqref{eq:entropy_integrability}. Similarly, taking $\tilde{t} = T$ in \eqref{eq:local_unif_bound}, we get that
{\small
\begin{align} \label{admiss_1}
     \mathbb{E}_{t,x} &\left[ \int_{\mathbb{A}} \sup_{s \in [t, T]} \left\{ \pi^1(s, X_s^1, a) (1 + |\ln \pi^1(s, X_s^1, a)|) \right\} da \right] \nonumber\\
     &=  \mathbb{E}_{t,x} \left[ \int_{\mathbb{A}} \sup_{s \in [t, T]} \left\{ \Gamma^0(s, X^1_s, a) \left( 1 + \left| \frac{1}{\lambda} \left( b(s, X^1_s, a) \cdot Z^0(s, X^1_s) + r(X^1_s, s, X^1_s, a) - H^0(s, X^1_s) \right) \right| \right) \right\} da \right] \nonumber\\
     &\leq \exp\left( C \left(1 + \|Z^0\|^{(0)}\right) \right) \left[ 1 + C\left(1 + \|Z^0\|^{(0)}\right) \right] < \infty,
\end{align}}which establishes \eqref{eq:local_unif_bound}. It thus holds that $\pi^1 \in \mathcal{A}$.}

To show $(V^{1,1}, V^{1,2}) \in (\Xi_{[0, T]}^{(2+\alpha)}\cap\Theta^{(2)}_{[0,T]}) \times C^{2+\alpha}([0, T] \times \mathbb{R}^d)$, consider the recursive linear PDE \eqref{eq:recursive} with $n=0$ and the following auxiliary linear systems 
\begin{align} 
	&\begin{cases} \label{eq:recursive_tau}
		w_t^{1,1}(\tau, t, y, x) + \frac{1}{2} \text{tr}\left(\sigma \sigma^\top(t, x) w_{xx}^{1,1}(\tau, t, y, x)\right) + H_z^0(t, x) \cdot \left[ w_x^{1,1}(\tau, t, y, x) + \delta'(t-\tau) Z^0(t,x) \right]\\
		\quad\qquad\qquad\quad  - \delta'(t-\tau) \left[ H^0(t, x) + \int_{\mathbb{A}} \left(r(y, t, x, a) - r(x, t, x, a) \right) \Gamma^0(t,x,a) da \right] = 0, \\
		w^{1,1}(\tau, T, y, x) = F_\tau(\tau, y, x),
	\end{cases} \\
	&\begin{cases} \label{eq:recursive_y}
		v_t^{1,1}(\tau, t, y, x) + \frac{1}{2} \text{tr}\left(\sigma \sigma^\top(t, x) v_{xx}^{1,1}(\tau, t, y, x)\right)  + H_z^0(t, x) \cdot v_x^{1,1}(\tau, t, y, x) \\\qquad\qquad\quad\ \   + \delta(t-\tau) \int_{\mathbb{A}} r_y(y, t, x, a) \Gamma^0(t,x,a) da = 0, \\
		v^{1,1}(\tau, T, y, x) = F_y(\tau, y, x),
	\end{cases}
\end{align}
where $w^{1,1}$ and $v^{1,1}$ are candidates for the $\tau$-derivative and $y$-gradient, respectively, of $V^{1,1}$. 
To apply the standard parabolic theory (e.g., \cite{ladyzenskaja_linear_nodate}) to these PDEs, we claim that their coefficients satisfy
 \begin{align}\label{holder_estimate}
 	\|a_{ij}\|^{(\alpha)} < \infty, \quad \|\bar{b}_j\|^{(\alpha)} < \infty, \quad \|\bar{f}\|^{(\alpha)}_{[0, T]} < \infty, 
 \end{align}
 with $a(t,x) := \sigma \sigma^\top(t,x)$, $\bar{b}(t,x) := H_z^0(t,x)$, and
 \begin{align*}
 	&\bar{f}(\tau, t, y, x) := \delta(t-\tau) \left[ \int_{\mathbb{A}} \left(r(y, t, x, a) - r(x, t, x, a) \right) \Gamma^0(t,x,a) da + H^0(t, x)  - H_z^0(t,x) \cdot Z^0(t,x) \right].
 \end{align*}
Let us focus on proving $\|\bar{f}\|^{(\alpha)}_{[0, T]}<\infty$ in \eqref{holder_estimate}. 
As a preparation, we will first establish the H{\"o}lder continuity of $Z^0$. For any fixed $t \in [0, T]$ and $x, x' \in \mathbb{R}^d$, 
\begin{align*}\label{Z_xholder}
	|&Z^0(t,x)-Z^0(t,x')| \leq |V_x^{0,1}(t, t, x, x)- V_x^{0,1}(t, t, x', x')| \nonumber \\
	&\hspace{1.5in}\quad + |G_z(t, x, V^{0,2}(t, x))\cdot V_x^{0,2}(t, x) - G_z(t, x', V^{0,2}(t, x'))\cdot V_x^{0,2}(t, x')| \nonumber \\
	&\leq | V_x^{0,1}(t, t, x, x) - V_x^{0,1}(t, t, x, x') | + | V_x^{0,1}(t, t, x, x') - V_x^{0,1}(t, t, x', x') | \nonumber \\
	&\quad + | G_z(t, x, V^{0,2}(t, x)) - G_z(t, x', V^{0,2}(t, x')) |  |V_x^{0,2}(t, x)| \nonumber \\
	&\quad + | G_z(t, x', V^{0,2}(t, x')) |  | V_x^{0,2}(t, x) - V_x^{0,2}(t, x') | \nonumber \\
	&\leq \bigg( [V^{0,1}_{xx}]^{(0)}_{[0,T]} + [V^{0,1}_{xy}]^{(0)}_{[0,T]} + \|V_x^{0,2}\|^{(0)} (\|G^0_{zx}\|^{(0)} + \|G^0_{zz}\|^{(0)} \|V_x^{0,2}\|^{(0)}) \nonumber 
    + \|G^0_z\|^{(0)} \|V_{xx}^{0,2}\|^{(0)} \bigg) |x-x'|,
\end{align*}
where for each $\phi \in \{G_z, G_{zx}, G_{zz}\}$, we write $\phi^0(t,x) := \phi(t, x, V^{0,2}(t, x))$. Under Assumption \ref{assump:regularity}, these derivatives are continuous and locally bounded in $z$ uniformly in $(t,x)$. As $V^{0,2}$ is bounded, this implies $\|\phi^0\|^{(0)}<\infty$. On the other hand, for any $t,s \in [0, T]$ with $t < s$ and $x \in \mathbb{R}^d$, 
\begin{align*}
	|Z^0(t,x)-Z^0(s,x)| &\leq |V_x^{0,1}(t, t, x, x)- V_x^{0,1}(s, s, x, x)| \nonumber \\
	&\quad + |G_z(t, x, V^{0,2}(t, x))\cdot V_x^{0,2}(t, x) - G_z(s, x, V^{0,2}(s, x))\cdot V_x^{0,2}(s, x)| \nonumber \\
	&\leq | V_x^{0,1}(t, t, x, x) - V_x^{0,1}(t, s, x, x) | + | V_x^{0,1}(t, s, x, x) - V_x^{0,1}(s, s, x, x) | \nonumber \\
	&\quad + | G_z(t, x, V^{0,2}(t, x)) - G_z(s, x, V^{0,2}(s, x)) |  |V_x^{0,2}(t, x)| \nonumber \\
	&\quad + | G_z(s, x, V^{0,2}(s, x)) |  | V_x^{0,2}(t, x) - V_x^{0,2}(s, x) | \nonumber \\
	&\leq \bigg(\langle V^{0,1}_{x}(t, \cdot, x, \cdot) \rangle_t^{(\frac{1+\alpha}{2})} + \|G^0_z\|^{(0)} \langle V^{0,2}_{x} \rangle_t^{(\frac{1+\alpha}{2})} \bigg) |s-t|^{\frac{1+\alpha}{2}} \nonumber \\
	&\quad + \bigg([V^{0,1}_{x\tau}]^{(0)}_{[0,T]} + \|V_x^{0,2}\|^{(0)} (\|G^0_{zt}\|^{(0)} + \|G^0_{zz}\|^{(0)} \|V_t^{0,2}\|^{(0)})\bigg) |s-t|.
\end{align*}
We then conclude from the previous two inequalities that 
\begin{equation}\label{Z_holder}
	\|Z^0\|^{(\alpha)} \leq C.
\end{equation}
By Assumption~\ref{assump:regularity}, \eqref{V^0 condition}, and Lemma~\ref{lemma:H_estimates}, we readily have $[\bar{f}]^{(0)}_{[0,T]} \leq C$. For any fixed $(\tau,y,x) \in [0, T]\times\Rb^d\times\Rb^d$ and $\tau \leq t < s \leq T$, observe that
 {\small\begin{align}
 	&|\bar{f}(\tau, t, y, x) - \bar{f}(\tau, s, y, x)| \nonumber \\
 	&\leq \underbrace{\left| \delta(t-\tau) - \delta(s-\tau) \right|  \left| \int_{\mathbb{A}} (r(y, t, x, a) - r(x, t, x, a)) \Gamma^0(t,x,a) da + H^0(t, x) - H_z^0(t,x) \cdot Z^0(t,x) \right|}_{\text{I}} \nonumber \\
 	&\quad + \underbrace{\delta(s-\tau) \left| H^0(t, x) - H^0(s, x) \right|}_{\text{II}} + \underbrace{\delta(s-\tau) \left| H_z^0(t, x) \cdot Z^0(t, x) - H_z^0(s, x) \cdot Z^0(s, x) \right|}_{\text{III}} \nonumber \\
 	&\quad + \underbrace{\delta(s-\tau) \left| \int_{\mathbb{A}} (r(y, t, x, a) - r(x, t, x, a)) \Gamma^0(t,x,a) da - \int_{\mathbb{A}} (r(y, s, x, a) - r(x, s, x, a)) \Gamma^0(s,x,a) da \right|}_{\text{IV}}. \label{eq:h_decomp_full}
 \end{align}}For term I, it follows from Assumption \ref{assump:regularity} (i) and (v) and Lemma \ref{lemma:H_estimates} (i) that
\begin{equation}
	\text{I} \leq C\Big( 1 + \|Z^0\|^{(0)} \Big) |t-s|. \label{eq:I}
\end{equation}
For term II, the linear growth of $H_t$ and the boundedness of $H_z$ yield
\begin{align}
	|H^0(t, x) - H^0(s, x)| 
	&\leq |H(t, x, Z^0(t,x)) - H(s, x, Z^0(t,x))| + |H(s, x, Z^0(t,x)) - H(s, x, Z^0(s,x))| \nonumber \\
	&\leq C\left(1 +\|Z^0\|^{(0)}\right) |t-s| + C |Z^0(t,x) - Z^0(s,x)|. \label{eq:H_inc_step}
\end{align}
Similarly, for terms III and IV, using Assumption \ref{assump:regularity} and Lemma \ref{lemma:H_estimates}, we obtain 
\begin{align}
	\text{III} &\leq \delta(s-\tau) \left| H_z^0(t, x) - H_z^0(s, x) \right|  |Z^0(t, x)| + \delta(s-\tau) |H_z^0(s, x)|  \left| Z^0(t, x) - Z^0(s, x) \right| \nonumber \\
	&\leq C \left( (1 + \|Z^0\|^{(0)}) |t-s| + |Z^0(t, x) - Z^0(s, x)| \right) \|Z^0\|^{(0)} + C |Z^0(t, x) - Z^0(s, x)|, \label{eq:term3_step} \\
	\text{IV} &\leq \delta(s-\tau) \int_{\mathbb{A}} \left| (r(y, t, x, a) - r(x, t, x, a)) - (r(y, s, x, a) - r(x, s, x, a)) \right| \Gamma^0(t, x, a) da \nonumber \\
	&\quad + \delta(s-\tau) \int_{\mathbb{A}} \left| r(y, s, x, a) - r(x, s, x, a) \right| \left| \Gamma^0(t, x, a) - \Gamma^0(s, x, a) \right| da \nonumber \\
	&\leq C |t-s| + C \left( (1 + \|Z^0\|^{(0)}) |t-s| + |Z^0(t, x) - Z^0(s, x)| \right). \label{eq:term4_step}
\end{align}
Plugging \eqref{Z_holder} and \eqref{eq:I}--\eqref{eq:term4_step} into \eqref{eq:h_decomp_full} leads to
	$\sup_{\tau\in[0,T], y\in\Rb^d}\langle \bar{f}(\tau, \cdot, y, \cdot) \rangle_t^{(\alpha/2)} \leq C$. 
It can be proved analogously that $\sup_{\tau\in[0,T], y\in\mathbb{R}^d}\langle \bar{f}(\tau, \cdot, y, \cdot) \rangle_x^{(\alpha)} \leq C$. We therefore obtain
\begin{equation*}
    \sup_{\tau\in[0,T], y\in\mathbb{R}^d}\|\bar{f}(\tau, \cdot, y, \cdot)\|_{[\tau, T] \times \mathbb{R}^d}^{(\alpha)} \leq C.
\end{equation*}
By following the above arguments, we can analogously prove that 
\[
\sup_{\tau \in [0,T], y \in \mathbb{R}^d}\|\bar{f}_{\tau}(\tau, \cdot, y, \cdot)\|_{[\tau, T] \times \mathbb{R}^d}^{(\alpha)}\le C\quad \hbox{and}\quad \sup_{\tau \in [0,T], y \in \mathbb{R}^d} \|\bar{f}_y(\tau, \cdot, y, \cdot)\|_{[\tau, T] \times \mathbb{R}^d}^{(\alpha)}\le C. 
\]
By recalling the definition of $\|\bar{f}\|_{[0, T]}^{(\alpha)}$ in \eqref{holder norm}, we then conclude $\|\bar{f}\|_{[0, T]}^{(\alpha)}<\infty$. As the other two estimates in \eqref{holder_estimate} can be proved in a similar (and in fact, simpler) manner, \eqref{holder_estimate} is established. 

By uniform non-degeneracy of $\sigma$ (Assumption~\ref{assump:regularity} (ii)) and H{\"o}lder continuity \eqref{holder_estimate} of coefficients, standard parabolic theory (e.g., \cite{ladyzenskaja_linear_nodate}) implies that for any fixed $(\tau, y) \in [0, T] \times \mathbb{R}^d$, the linear PDEs \eqref{eq:recursive}, \eqref{eq:recursive_tau}, and \eqref{eq:recursive_y} admit unique solutions $(\tilde{V}^{1,1}(\tau, \cdot, y, \cdot),\tilde{V}^{1,2}(\cdot,\cdot))\in C^{2+\alpha}([\tau, T] \times \mathbb{R}^d) \times C^{2+\alpha}([0, T] \times \mathbb{R}^d)$, $w^{1,1}(\tau, \cdot, y, \cdot)\in C^{2+\alpha}([\tau, T] \times \mathbb{R}^d)$, and $v^{1,1}(\tau, \cdot, y, \cdot) \in C^{2+\alpha}([\tau, T] \times \mathbb{R}^d)$, with the Schauder estimate
\begin{equation}\label{schauder}
	[\tilde{V}^{1,1}]^{(2+\alpha)}_{[0, T]} + \|\tilde{V}^{1,2}\|^{(2+\alpha)} + [w^{1,1}]^{(2+\alpha)}_{[0, T]} + [v^{1,1}]^{(2+\alpha)}_{[0, T]} \leq C \left( \|F\|_{[0, T]}^{(2+\alpha)} + \|h\|^{(2+\alpha)} + \|\bar{f}\|_{[0, T]}^{(\alpha)} \right).
\end{equation}
 Now, we claim that
\begin{equation}\label{a claim}
	w^{1,1} = \partial_\tau \tilde{V}^{1,1},\ \ \tilde V_{x\tau}^{1,1} = \tilde V_{\tau x}^{1,1};\quad   \quad v^{1,1} = \partial_y \tilde{V}^{1,1},\ \ \tilde V_{xy}^{1,1} = \tilde V_{yx}^{1,1}.
\end{equation}
Let us focus on the first half of \eqref{a claim}. {\color{black}Consider the difference quotient operator with respect to $\tau$, defined for $\eta \neq 0$ by $\Delta_\eta \phi(\tau, \cdot) := {(\phi(\tau+\eta, \cdot) - \phi(\tau, \cdot))}/{\eta}$.} By smoothly extending the discount function $\delta$ to a slightly larger interval $[-\epsilon, T]$ for some $\epsilon > 0$, {\color{black}for each $(\tau, y) \in [0, T] \times \mathbb{R}^d$ and any $\eta \in (-\epsilon, \epsilon) \setminus \{0\}$ with $\tau + \eta \in [0, T]$ (using one-sided difference quotients at the endpoints $\tau = 0$ and $\tau = T$), the backward linear PDE governing $\tilde{V}^{1,1}(\tau+\eta, \cdot, y, \cdot)$ remains well-posed on $[(\tau+\eta-\epsilon) \vee 0, T] \times \mathbb{R}^d \supseteq [\tau, T] \times \mathbb{R}^d$. Hence, we can extend $\tilde{V}^{1,1}(\tau+\eta, \cdot, y, \cdot)$ to $[\tau, T] \times \mathbb{R}^d$, so that both $\Delta_\eta F(\tau, y, x)$ and $\Delta_\eta \tilde{V}^{1,1}(\tau, \cdot, y, \cdot)$ are well-defined on $[\tau, T] \times \mathbb{R}^d$.} It follows that $\Delta_\eta \tilde{V}^{1,1}$ satisfies the linear PDE
\begin{equation} \label{eq:diff_quotient_pde}
	\begin{cases}
		\partial_t (\Delta_\eta \tilde{V}^{1,1})(\tau, t, y, x) + \frac{1}{2} \text{tr}\left(\sigma \sigma^\top(t, x) \partial_{xx}^2 (\Delta_\eta \tilde{V}^{1,1})(\tau, t, y, x)\right) \\
		\qquad\qquad + H^0_{z}(t, x) \cdot \partial_x (\Delta_\eta \tilde{V}^{1,1})(\tau, t, y, x) + \mathcal{R}_\eta(\tau, t, y, x) = 0, \\
		\Delta_\eta \tilde{V}^{1,1}(\tau, T, y, x) = \Delta_\eta F(\tau, y, x),
	\end{cases}
\end{equation}
where 
\begin{align*}
	\mathcal{R}_\eta(\tau, t, y, x) := &\frac{\delta(t-\tau-\eta)-\delta(t-\tau)}{\eta} \bigg[\int_{\mathbb{A}} (r(y, t, x, a) - r(x, t, x, a))\Gamma^0(t,x,a)da\\&+H^0(t, x) - H^0_{z}(t, x) \cdot Z^0(t, x) \bigg].
\end{align*}
We may estimate the H{\"o}lder norms of $\mathcal{R}_\eta$ and $\Delta_\eta F$ by the same arguments for proving \eqref{holder_estimate}, specifically using the $C^2$-regularity of $\delta$ (Assumption~\ref{assump:regularity} (v)), and the H{\"o}lder regularity of $F$ (Assumption~\ref{assump:regularity} (iv)). 
We can thus obtain from classical Schauder estimates that
\begin{equation*}
	\|\Delta_\eta \tilde{V}^{1,1}(\tau, \cdot, y, \cdot)\|^{(2+\alpha)}_{[\tau, T] \times \mathbb{R}^d} \leq C, \quad \forall (\tau, y) \in [0, T] \times \mathbb{R}^d,
\end{equation*}
where $C>0$ here is independent of $\eta$. It then follows from the Arzel\`a-Ascoli theorem that the family $\{\Delta_\eta \tilde{V}^{1,1}(\tau,\cdot,y,\cdot)\}_{\eta\neq 0}$ is relatively compact in $C^2_{\mathrm{loc}}([\tau,T]\times\mathbb{R}^d)$. Thus, for any sequence $\eta_k \to 0$ {\color{black}with $\tau + \eta_k \in [0, T]$}, there exist a subsequence (without relabeling) and a limit function $W(\tau, \cdot, y, \cdot)$ such that $\Delta_{\eta_k} \tilde{V}^{1,1}\to W$ in $C^2(K)$ for any compact subset $K \subset [\tau, T] \times \mathbb{R}^d$. By taking  $k \to \infty$ in PDE \eqref{eq:diff_quotient_pde} (with $\eta=\eta_k$ therein), we find that $W$ satisfies PDE \eqref{eq:recursive_tau}. By the uniqueness of classical solutions, $W \equiv w^{1,1}$ on $[\tau, T] \times \mathbb{R}^d$. Now, as every convergent subsequence has the same limit $w^{1,1}$, the whole sequence converges, i.e.,
\begin{equation*}
	\lim_{\eta \to 0} \Delta_\eta \tilde{V}^{1,1}(\tau, t, y, x) = w^{1,1}(\tau, t, y, x), \quad \forall (\tau, t, y, x) \in \Delta[0, T] \times \mathbb{R}^d\times \mathbb{R}^d,
\end{equation*}
which gives the existence of the partial derivative $\partial_\tau \tilde{V}^{1,1} = w^{1,1}$. {Moreover, the local uniform convergence of the spatial derivatives yields}
\begin{equation*} \label{eq:mixed_xtau_consist}
	\partial_x (\partial_\tau \tilde{V}^{1,1}) = \partial_x \left( \lim_{\eta \to 0} \Delta_\eta \tilde{V}^{1,1} \right) = \lim_{\eta \to 0} \Delta_\eta (\partial_x \tilde{V}^{1,1}) = \partial_\tau (\partial_x \tilde{V}^{1,1}).
\end{equation*}
Hence, the first half of \eqref{a claim} is proved. As the other half of \eqref{a claim} can be proved  analogously, we conclude that \eqref{a claim} is true.\footnote{By identical reasoning, the order of differentiation can be exchanged between the reference variables $(\tau, y)$ and all other flow variables (i.e., $\partial_t$ and $\partial_{xx}^2$).} By \eqref{a claim} and the uniform bound $[w^{1,1}]_{[0,T]}^{(2)} + [v^{1,1}]_{[0,T]}^{(2)} < \infty$, $\tilde{V}^{1,1}$ and its flow derivatives $(\tilde{V}_t^{1,1}, \tilde{V}_x^{1,1}, \tilde{V}_{xx}^{1,1})$ are Lipschitz continuous w.r.t.\ $\tau$ (resp.\ $y$), uniformly in other variables. Thanks to their H{\"o}lder continuity in $(t, x)$, these functions are jointly continuous on $\Delta[0, T] \times \mathbb{R}^d \times \mathbb{R}^d$. Then, as $\tilde{V}^{1,1}$, $w^{1,1}=\partial_\tau \tilde{V}^{1,1}$, and $v^{1,1}=\partial_y \tilde{V}^{1,1}$ satisfy \eqref{schauder}, we obtain 
\begin{equation*}
	(\tilde{V}^{1,1}, \tilde{V}^{1,2}) \in (\Xi_{[0, T]}^{(2+\alpha)} \cap \Theta_{[0,T]}^{(2)}) \times C^{2+\alpha}([0, T]\times\Rb^d).
 \end{equation*} 
Finally, Feynman--Kac's formula implies that  $(\tilde{V}^{1,1}, \tilde{V}^{1,2})$ has a probabilistic representation as $(V^{1,1}, V^{1,2})$ in \eqref{V^1, V^2, n} (with $n=0$). Hence, as we have proved above, $(V^{1,1}, V^{1,2}) = (\tilde{V}^{1,1}, \tilde{V}^{1,2})$ satisfies \eqref{V^n condition}-\eqref{V^n condition'} (with $n=0$) and is the unique solution in $(\Xi_{[0, T]}^{(2+\alpha)} \cap \Theta_{[0,T]}^{(2)}) \times C^{2+\alpha}([0, T]\times\Rb^d)$ to \eqref{eq:recursive} (with $n=0$). 


{\bf Step 2: Other cases $n\ge 1$:} As $(V^{1,1}, V^{1,2})$ satisfies \eqref{V^n condition}-\eqref{V^n condition'} (with $n=0$), we may repeat the same arguments in Step 1 above to show that $\pi^2\in\mathcal A$ and $(V^{2,1}, V^{2,2})$ satisfies \eqref{V^n condition}-\eqref{V^n condition'} (with $n=1$) and is the unique solution in $(\Xi_{[0, T]}^{(2+\alpha)} \cap \Theta_{[0,T]}^{(2)}) \times C^{2+\alpha}([0, T]\times\Rb^d)$ to \eqref{eq:recursive} (with $n=1$). Continuing this procedure shows that all the properties hold for all $n\ge 1$. 

\subsection{Proof of Proposition~\ref{prop:1} (ii)} 
{\color{black}Throughout this proof, let $C \ge 1$ be a generic constant independent of both $n$ and $T$, which may change from line to line; it only depends on $d, m, C_0, \lambda, |\mathbb{A}|$, and the norms $\|h\|^{(2)}$, $\|\delta\|^{(0)}$. When a constant depends on $T$, we denote it by $C(T)$, which may additionally depend on the norms $\|F\|_{[0,T]}^{(2)}$, $\|V^{0,1}\|_{[0,T]}^{(2)}$, $\|V^{0,2}\|^{(2)}$, and the bounds of $G$ in Assumption~\ref{assump:regularity}~(iii) on $|z| \le \max\big\{ \|V^{0,2}\|^{(0)}, \|h\|^{(0)} \big\}$.}
First, it is clear  that $\|V^{n+1,2}\|^{(0)} \leq \|h\|^{(0)}$. To establish uniform estimates for higher-order spatial and temporal derivatives, we will use probabilistic representation formulae, as proposed in \cite{ma2025convergence}. To this end, fix $(t, x) \in [0, T) \times \mathbb{R}^d$ and consider the auxiliary state process
\begin{equation*}
	X_s^{t, x} = x + \int_t^s \sigma\left(l, X_l^{t, x}\right) d W_l, \quad s \in [t, T].
\end{equation*}
As $(V^{n+1,1}, V^{n+1,2})$ is the solution to \eqref{eq:recursive}, Feynman--Kac's formula implies
\begin{align}
	V^{n+1,1}(\tau, t, y, x) &= \mathbb{E} \left[ F(\tau, y, X_T^{t, x}) + \int_t^T f^{n+1,1}(\tau, s, y, X_s^{t, x}) ds \right], \label{eq:rep1} \\
	V^{n+1,2}(t, x) &= \mathbb{E} \left[ h(X_T^{t, x}) + \int_t^T f^{n+1,2}(s, X_s^{t, x}) ds \right], \label{eq:rep2} \\
	V_y^{n+1,1}(\tau, t, y, x) &= \mathbb{E} \left[ F_y(\tau, y, X_T^{t, x}) + \int_t^T f_y^{n+1,1}(\tau, s, y, X_s^{t, x}) ds \right], \nonumber
\end{align}
where 
\begin{align}
	f^{n+1,1}(\tau, s, y, x) &:= H_z^n(s, x) \cdot \left[ V_x^{n+1,1}(\tau, s, y, x) - \delta(s-\tau) Z^n(s, x) \right]\nonumber \\
	&\qquad + \delta(s-\tau) \left[ H^n(s, x) + \int_{\mathbb{A}} (r(y, s, x, a) - r(x, s, x, a)) \Gamma^n(s, x, a) da \right],\label{eq:f1} \\
	f^{n+1,2}(s, x) &:= H_z^n(s, x) \cdot V_x^{n+1,2}(s, x).\label{eq:f2}
\end{align}
By applying the Bismut--Elworthy--Li formula in \cite{Bismut1984LargeDeviation, elworthy1994formulae, ma2002representation} and following the arguments in \cite{ma2025convergence}, we obtain representation formulae for 
derivatives of $V^{n+1,1}$ and $V^{n+1,2}$ in $x$, as well as 
$V_{yx}^{n+1,1}$, i.e., 
\begin{align}
	V_x^{n+1,1}(\tau,t, y,x) &= \mathbb{E}\left[ (\nabla X_T^{t, x})^\top F_x(\tau, y, X_T^{t, x}) + \int_t^T f^{n+1,1}(\tau, s, y, X_s^{t, x}) N_s^{t, x} d s \right], \label{rep_first1} \\
	V_{xx}^{n+1,1}(\tau,t,y, x) &= \mathbb{E}\left[ (\nabla X_T^{t, x})^\top F_{xx}(\tau, y, X_T^{t, x}) \nabla X_T^{t, x} + F_x(\tau, y, X_T^{t, x}) \otimes \nabla^2 X_T^{t, x} \right. \nonumber \\
	&\quad \left. + \int_t^T \left[ N_s^{t, x} \left( f_x^{n+1,1}(\tau, s, y, X_s^{t, x}) \right)^\top \nabla X_s^{t, x} + f^{n+1,1}(\tau, s, y, X_s^{t, x}) \nabla N_s^{t, x} \right] d s \right], \label{rep_second1} \\
	V_{yx}^{n+1,1}(\tau,t, y,x) &= \mathbb{E}\left[ (\nabla X_T^{t, x})^\top F_{yx}(\tau, y, X_T^{t, x}) + \int_t^T f_y^{n+1,1}(\tau, s, y, X_s^{t, x}) N_s^{t, x} d s \right], \label{rep_first1y} \\
	V_x^{n+1,2}(t,x) &= \mathbb{E}\left[ (\nabla X_T^{t, x})^\top h_x(X_T^{t, x}) + \int_t^T f^{n+1,2}(s, X_s^{t, x}) N_s^{t, x} d s \right], \label{rep_first2} \\
	V_{xx}^{n+1,2}(t, x) &= \mathbb{E}\left[ (\nabla X_T^{t, x})^\top h_{xx}(X_T^{t, x}) \nabla X_T^{t, x} + h_x(X_T^{t, x}) \otimes \nabla^2 X_T^{t, x} \right. \nonumber \\
	&\quad \left. + \int_t^T \left[ N_s^{t, x} \left( f_x^{n+1,2}(s, X_s^{t, x}) \right)^\top \nabla X_s^{t, x} + f^{n+1,2}(s, X_s^{t, x}) \nabla N_s^{t, x} \right] d s \right]. \label{rep_second2} 
\end{align}
Here, the $i$-th column of $\nabla X\in\mathbb R^{d\times d}$ is $\partial_{x_i} X$ and $\nabla^2 X \in \mathbb{R}^{d \times d \times d}$ is a tensor with $\nabla_j^2 X \in \mathbb{R}^{d \times d}$ denoting $\partial_{x_j} \nabla X$. Moreover, by using Einstein summation for repeated indices,
\begin{equation*}
	\begin{aligned}
		& \nabla X_s^{t, x} = I_d + \int_t^s \sigma_x^i\left(l, X_l^{t, x}\right) \nabla X_l^{t, x} d W_l^i, \\
		& \nabla_j^2 X_s^{t, x} = \int_t^s \left[ \sigma_{xx_k}^i\left(l, X_l^{t, x}\right) \left(\nabla X_l^{t, x}\right)^{kj} \nabla X_l^{t, x} + \sigma_x^i\left(l, X_l^{t, x}\right) \nabla_j^2 X_l^{t, x} \right] d W_l^i, \\
		& \left[ \phi_x\left(X_s^{t, x}\right) \otimes \nabla^2 X_s^{t, x} \right]_{ij} := \phi_x\left(X_s^{t, x}\right) \cdot \left( \nabla_j^2 X_s^{t, x} \right)^i, \quad \text{for } \phi = F, h,
	\end{aligned}
\end{equation*}
where $\sigma^i$ is the $i$-th column of $\sigma$ and $(\nabla_j^2 X)^i$ is the $i$-th column of $\nabla_j^2 X$. Similarly, by setting $\check{\sigma}:=\sigma^{-1}$ to be the inverse matrix, we define
\begin{equation*}
	\begin{aligned}
		& N_s^{t, x}:=\frac{1}{s-t} \int_t^s\left(\check{\sigma}\left(l, X_l^{t, x}\right) \nabla X_l^{t, x}\right)^{\top} d W_l, \\
		& \nabla_i N_s^{t, x}:=\frac{1}{s-t} \int_t^s\left(\left(\nabla X_l^{t, x}\right)^{ji} \check{\sigma}_{x_j}\left(l, X_l^{t, x}\right) \nabla X_l^{t, x}+\check{\sigma}\left(l, X_l^{t, x}\right) \nabla_i^2 X_l^{t, x}\right)^{\top} d W_l.
	\end{aligned}
\end{equation*}
Furthermore, standard estimates (see, e.g., \cite{ma2025convergence}) show that 
\begin{equation}\label{estimate_1}
	\mathbb{E}\left[\left|\nabla X_s^{t, x}\right|^2+\left|\nabla^2 X_s^{t, x}\right|^2\right] \leq C e^{C(s-t)}, \quad \mathbb{E}\left[\left|N_s^{t, x}\right|^2+\left|\nabla N_s^{t, x}\right|^2\right] \leq \frac{C e^{C(s-t)}}{s-t}.
\end{equation}
Now, to establish a uniform bound for $\|V_{x}^{n,2}\|^{(0)}$, consider the local domain $[T-\epsilon,T]\times\mathbb R^d$
for $0<\epsilon\leq T$. 
For any $(s, x) \in [T-\epsilon, T] \times \mathbb{R}^d$, by the boundedness of $H_z$ in Lemma \ref{lemma:H_estimates}, $f^{n+1,2}$ in \eqref{eq:f2} satisfies
\begin{align}
	|f^{n+1,2}(s, x)| \leq C \|V_x^{n+1,2}\|^{(0)}_{[T-\epsilon, T]\times\Rb^d}. \label{eq:source2_growth2}
\end{align}
Combining \eqref{rep_first2}, \eqref{estimate_1}, and \eqref{eq:source2_growth2}, we obtain 
\begin{align*}
	|V^{n+1,2}_{x}(t, x)| &\leq \mathbb{E}[|\nabla X_T^{t, x}|]  \|h_x\|^{(0)} + C \|V_x^{n+1,2}\|^{(0)}_{[T-\epsilon, T]\times\Rb^d}\int_t^T \mathbb{E}\left[|N_s^{t, x}| \right] ds \nonumber \\
	&\leq C e^{C(T-t)} \left[ 1 + \|V_x^{n+1,2}\|^{(0)}_{[T-\epsilon, T]\times\Rb^d} \int_t^T \frac{1}{\sqrt{s-t}} ds \right] \\
	&\leq C e^{C\epsilon} \left[ 1 + \sqrt{\epsilon}\|V_x^{n+1,2}\|^{(0)}_{[T-\epsilon, T]\times\Rb^d} \right].
\end{align*}
{\color{black}Choosing $0 < \epsilon \leq \min\big\{ T, \frac{1}{C}, \frac{1}{9 e^2 C^2} \big\}$, we get that $C e^{C\epsilon} \sqrt{\epsilon} \leq C e \frac{1}{3eC} = \frac{1}{3} \leq \frac{1}{2}$. Rearranging terms yields the local bound}
\begin{equation*}
	\|V_x^{n+1,2}\|^{(0)}_{[T-\epsilon, T]\times\Rb^d} \leq 2 Ce^{C\epsilon}.
\end{equation*}
To extend this bound to the global domain $[0, T] \times \mathbb{R}^d$, {\color{black}let us partition $[0,T]$ into $K := \lfloor T/\epsilon \rfloor + 1$ equal subintervals $0 = t_0 < t_1 < \dots < t_K = T$, such that $t_i - t_{i-1} = T/K < \epsilon$.} We proceed by backward induction from $i=K-1$ down to $i=1$. Assume that for the $i$-th segment $[t_{i-1}, t_i]$, the gradients on  $[t_i, T] \times \mathbb{R}^d$ are uniformly bounded by $A_i := \sup_{n \geq 1} \| V_{x}^{n,2}\|^{(0)}_{[t_i, T]\times \mathbb{R}^d} < \infty$. For any $(t, x) \in [t_{i-1}, t_i] \times \mathbb{R}^d$, by \eqref{rep_first2} and the inductive hypothesis,
\begin{align*}
	|V_{x}^{n+1,2}(t, x)| &\leq \mathbb{E}\left[ | \nabla X_{t_i}^{t, x} | | V_x^{n+1,2}( t_i, X_{t_i}^{t, x}) | \right] + \mathbb{E} \left[ \int_t^{t_i} |f^{n+1,2}(s, X^{t, x}_s)| |N_s^{t, x}| ds \right] \\
	&\leq C e^{C\epsilon} A_i + C e^{C\epsilon} \sqrt{\epsilon} \|V_x^{n+1,2}\|^{(0)}_{[t_{i-1}, t_i]\times\mathbb{R}^d}.
\end{align*}
Rearranging terms with $C e^{C\epsilon} \sqrt{\epsilon} \leq \frac{1}{3} \leq \frac{1}{2}$, we find
$\|V_x^{n+1,2}\|^{(0)}_{[t_{i-1}, t_i]\times\mathbb{R}^d} \leq 2 C e^{C\epsilon} A_i$. 
{\color{black}Thanks to the local bound $A_{K-1} \leq 2Ce^{C\epsilon}$, backward induction yields
\begin{equation*}
	\sup_{n \geq 1} \|V_x^{n,2}\|^{(0)} = A_0 \leq \left(2Ce^{C\epsilon}\right)^{K-1} A_{K-1} \leq \left(2Ce^{C\epsilon}\right)^K \leq C(T) < \infty.
\end{equation*}}
On the other hand, by Lemma \ref{lemma:H_estimates} and $V^{n+1,1}_{xy}=V^{n+1,1}_{yx}$ (from part (i)), $f^{n+1,1}$ in \eqref{eq:f1} satisfies
\begin{equation*}
	|f^{n+1,1}_{y}(\tau, t, y, x)| \leq C \left( 1 + |V^{n+1,1}_{yx}(\tau, t, y, x)| \right),\quad (\tau, t, y, x) \in \Delta[0, T] \times \mathbb{R}^d \times \mathbb{R}^d.
\end{equation*}
By plugging this into \eqref{rep_first1y}, an analogous backward induction argument yields
\begin{equation*}
	\sup_{n \geq 1} [V_{xy}^{n,1}]^{(0)}_{[0,T]} =\sup_{n \geq 1} [V_{yx}^{n,1}]^{(0)}_{[0,T]}{\color{black} \leq C(T) }< \infty.
\end{equation*}

Next, we aim for a uniform bound for $[V^{n,1}_{x}]^{(0)}_{[0,T]}$. 
For any $(\tau, s, y, x) \in \Delta[0, T] \times \mathbb{R}^d \times \mathbb{R}^d$, by Lemma \ref{lemma:H_estimates}, Assumption \ref{assump:regularity}, and the established uniform bounds for $\|V^{n,2}\|^{(0)}$ and $\|V^{n,2}_x\|^{(0)}$, we have
\begin{align}
	|f^{n+1,1}(\tau, s,y, x)| &\leq |H_z^n(s, x)|  [ |V^{n+1,1}_{x}(\tau, s, y, x)| +\|\delta\|^{(0)} |Z^n(s,x)| ] + \|\delta\|^{(0)}( | H^n(s, x)|+2C_0) \nonumber \\
	&\leq C \left(1+ |V^{n+1,1}_{x}(\tau, s, y, x)| + |Z^n(s,x)| \right)\nonumber\\
	&\leq {\color{black}C(T)} + C \left( |V^{n+1,1}_{x}(\tau, s, y, x)| + |V^{n,1}_{x}(s, s, x, x)| \right). \label{eq:source_growth}
\end{align}
Combining \eqref{rep_first1}, \eqref{estimate_1}, and \eqref{eq:source_growth}, we see that for any $(\tau, t, y, x)$ in the local domain $\Delta[T-\epsilon, T] \times \mathbb{R}^d \times \mathbb{R}^d$,
\begin{align*}
	| V^{n+1,1}_x(\tau, t, y, x)| &\leq \mathbb{E}[|\nabla X_T^{t, x}|] [F_x]^{(0)}_{[0, T]} + \int_t^T \mathbb{E}\left[ |f^{n+1,1}(\tau, s, y, X_s^{t, x})| |N_s^{t, x}| \right] ds \nonumber \\
	&\leq C e^{C\epsilon} [F_x]^{(0)}_{[0, T]} + \int_t^T \left[ {\color{black}C(T)} + C \left( [V^{n+1,1}_x]^{(0)}_{[T-\epsilon, T]} + [V^{n,1}_x]^{(0)}_{[T-\epsilon, T]} \right) \right] \frac{C e^{C(s-t)}}{\sqrt{s-t}} ds \\
	&\leq {\color{black}C(T)} + C e^{C\epsilon} \sqrt{\epsilon} \left( [ V^{n+1,1}_x]^{(0)}_{[T-\epsilon, T]} + [V^{n,1}_x]^{(0)}_{[T-\epsilon, T]} \right).
\end{align*}
Decreasing $\epsilon > 0$ if necessary to ensure that $C e^{C\epsilon} \sqrt{\epsilon} \leq \frac{1}{3}$, we have $\frac{C e^{C\epsilon}\sqrt{\epsilon}}{1 - Ce^{C\epsilon}\sqrt{\epsilon}} \leq \frac{1/3}{2/3} = \frac{1}{2}$. Rearranging terms yields the recursive inequality
\begin{equation*}
	[V^{n+1,1}_x]^{(0)}_{[T-\epsilon, T]} \leq \frac{1}{2} [ V^{n,1}_x]^{(0)}_{[T-\epsilon, T]} + {\color{black}C(T)}.
\end{equation*}
By iterating this inequality from $n$ to $0$, we obtain
\begin{align*}
	[V_x^{n,1}]^{(0)}_{[T-\epsilon, T]} &\leq \left( \frac{1}{2} \right)^n  [ V_x^{0,1}]^{(0)}_{[T-\epsilon, T]} + \sum_{k=0}^{n-1} \left( \frac{1}{2} \right)^k {\color{black}C(T)} \leq  [V_x^{0,1}]^{(0)}_{[0, T]} + 2{\color{black}C(T)} \leq {\color{black}C(T)}, 
\end{align*}
which gives a uniform bound on the final time segment. The extension of this to the global domain $\Delta[0, T] \times \mathbb{R}^d \times \mathbb{R}^d$ follows from the same backward induction argument used above for $\|V^{n,2}_x\|^{(0)}$, thereby giving $\sup_{n\ge 1}[V_x^{n,1}]^{(0)}_{[0,T]} \leq {\color{black}C(T)} < \infty$. This, along with \eqref{eq:rep1} and \eqref{eq:source_growth}, also implies $\sup_{n \geq 1} [V^{n,1}]^{(0)}_{[0,T]} \leq {\color{black}C(T)} < \infty$.

Next, we aim for uniform bounds for $[V_{xx}^{n,1}]^{(0)}_{[0,T]}$ and $\|V_{xx}^{n,2}\|^{(0)}$ respectively. Note that
\begin{align}
	f_x^{n+1,1}(\tau, s, y, x) &= \left[ H_{zx}^n(s,x) + H_{zz}^n(s,x)\cdot Z^n_x(s,x) \right] \cdot \left[ V_x^{n+1,1}(\tau, s, y, x) - \delta(s-\tau) Z^n(s,x) \right] \nonumber \notag\\
	&\quad + H_z^n(s, x) \cdot \left[ V_{xx}^{n+1,1}(\tau, s, y, x) - \delta(s-\tau)Z^n_x(s,x) \right] \nonumber \notag\\
	&\quad + \delta(s-\tau) \bigg[ H_x^n(s, x) + H_z^n(s, x)\cdot Z_x^n(s,  x)\notag\\	&\quad+\int_{\mathbb{A}}(r_x(y,s,x,a)-r_y(x,s,x,a)-r_x(x,s,x,a))\Gamma^n(s,x,a)da\notag\\	&\quad+\int_{\mathbb{A}}(r(y,s,x,a)-r(x,s,x,a))\partial_x\Gamma^n(s,x,a)da \bigg],\label{f^1_x}\\
	f_x^{n+1,2}(s,x) &= \left[ H_{zx}^n(s,x) + H_{zz}^n(s,x)\cdot Z^n_x(s,x) \right] \cdot  V_x^{n+1,2}(s, x) 
	 + H_z^n(s, x) \cdot V_{xx}^{n+1,2}(s,x).\label{f^2_x}
\end{align}
Also, with the established uniform bounds for $\|V^{n,2}\|^{(0)}$, $\|V^{n,2}_x\|^{(0)}$, $[V^{n,1}_x]_{[0,T]}^{(0)}$, and $[V^{n,1}_{xy}]_{[0,T]}^{(0)}$, we have 
\begin{align*}
	|Z^n(s,x)|&= |V_x^{n,1}(s, s, x, x) + G_z(s, x, V^{n,2}(s, x))\cdot  V_x^{n,2}(s, x)|\leq {\color{black}C(T)},\\
	|Z^n_x(s,x)|&=|V_{xy}^{n,1}(s, s, x, x)+V_{xx}^{n,1}(s, s, x, x)+ [G_{zx}(s, x, V^{n,2}(s, x))\\&\quad+G_{zz}(s, x, V^{n,2}(s, x))\cdot V_x^{n,2}(s, x)]\cdot  V_x^{n,2}(s, x) + G_z(s, x, V^{n,2}(s, x))\otimes V_{xx}^{n,2}(s, x)|\\&\leq {\color{black}C(T)}(1+|V_{xx}^{n,1}(s, s, x, x)|+|V_{xx}^{n,2}(s, x)|).
\end{align*}
Using the above estimates, Lemma \ref{lemma:H_estimates}, and \eqref{eq:Gamma_x_est}, we deduce from \eqref{f^1_x}-\eqref{f^2_x} that
\begin{align*}
	&|f_x^{n+1,1}(\tau, s, y, x)| \leq {\color{black}C(T)} \left[ 1 + |V_{xx}^{n,1}(s, s, x, x)| + |V_{xx}^{n,2}(s, x)| \right] + C |V_{xx}^{n+1,1}(\tau, s, y, x)|, \\
	&|f_x^{n+1,2}(s, x)| \leq {\color{black}C(T)} \left[ 1 + |V_{xx}^{n,1}(s, s, x, x)| + |V_{xx}^{n,2}(s, x)| \right] + C |V_{xx}^{n+1,2}(s, x)|,\\
	&|f^{n+1,1}(\tau, s, y, x)| + |f^{n+1,2}(s, x)| \leq {\color{black}C(T)}.
\end{align*}
Plugging these into \eqref{rep_second1} and \eqref{rep_second2} and utilizing the estimates \eqref{estimate_1}, we have, on $\Delta[T-\epsilon, T] \times \Rb^d$ for $0<\epsilon\leq T$, that
\begin{align*}
	|V_{xx}^{n+1,1}(\tau, t, y, x)| &\leq C e^{C\epsilon} \left( [F_{xx}]_{[0, T]}^{(0)} + [F_{x}]_{[0, T]}^{(0)} \right) + \int_t^T \left( \mathbb{E}[|N^{t, x}_s| |f_x^{n+1,1}(\tau, s, y,X_s^{t, x})| |\nabla X^{t, x}_s|]\right.\\&\left. \quad + \mathbb{E}[|f^{n+1,1}(\tau, s,y, X_s^{t, x})| |\nabla N^{t, x}_s|] \right) ds \\
	&\leq {\color{black}C(T)} + {\color{black}C(T)} e^{C\epsilon} \sqrt{\epsilon} \left( [V_{xx}^{n,1}]^{(0)}_{[T-\epsilon, T]} + \|V_{xx}^{n,2}\|^{(0)}_{[T-\epsilon, T]} \right) + C e^{C\epsilon} \sqrt{\epsilon} [V_{xx}^{n+1,1}]^{(0)}_{[T-\epsilon, T]},\\
	|V_{xx}^{n+1,2}(t, x)| &\leq C e^{C\epsilon} \left( \|h_{xx}\|^{(0)} + \|h_x\|^{(0)} \right) + \int_t^T \left( \mathbb{E}[|N^{t, x}_s| |f_x^{n+1,2}(s, X_s^{t, x})| |\nabla X^{t, x}_s|]\right.\\&\left. \quad + \mathbb{E}[|f^{n+1,2}(s, X_s^{t, x})| |\nabla N^{t, x}_s|] \right) ds \\
	&\leq {\color{black}C(T)} + {\color{black}C(T)} e^{C\epsilon} \sqrt{\epsilon} \left( [V_{xx}^{n,1}]^{(0)}_{[T-\epsilon, T]} + \|V_{xx}^{n,2}\|^{(0)}_{[T-\epsilon, T]} \right) + C e^{C\epsilon} \sqrt{\epsilon} \|V_{xx}^{n+1,2}\|^{(0)}_{[T-\epsilon, T]}.
\end{align*}
{\color{black}With the constant $C(T) < \infty$ fixed, choose $0 < \epsilon \leq \min\big\{ T, \frac{1}{C}, \frac{1}{36 e^2 (C(T))^2} \big\}$, such that $C(T) e^{C\epsilon} \sqrt{\epsilon} \leq \frac{1}{6}$ (which implies $C e^{C\epsilon}\sqrt{\epsilon} \leq \frac{1}{6}$ as $C \leq C(T)$) and $\frac{2C(T) e^{C\epsilon} \sqrt{\epsilon}}{1 - Ce^{C\epsilon}\sqrt{\epsilon}} \leq \frac{2/6}{5/6} = \frac{2}{5} \leq \frac{1}{2}$. Summing the above two inequalities and rearranging terms yields}
\begin{align*}
[V_{xx}^{n+1,1}]^{(0)}_{[T-\epsilon, T]}+\|V_{xx}^{n+1,2}\|^{(0)}_{[T-\epsilon, T]\times\Rb^d} \leq \frac{1}{2} \left( [V_{xx}^{n,1}]^{(0)}_{[T-\epsilon, T]}+\|V_{xx}^{n,2}\|^{(0)}_{[T-\epsilon, T]\times\Rb^d} \right)+ {\color{black}C(T)},
\end{align*}
which gives a uniform bound for the second derivatives on the final time segment. The extension of this to the entire domain follows from the same backward induction arguments used for the first derivatives, yielding $\sup_{n \geq 1} \left( [V_{xx}^{n,1}]^{(0)}_{[0,T]} + \|V_{xx}^{n,2}\|^{(0)} \right) \leq {\color{black}C(T)} < \infty$. 

Finally, uniform boundedness of the time derivatives $V_t^{n+1,1}$ and $V_t^{n+1,2}$ follows from the recursive linear PDE system \eqref{eq:recursive}, given that all spatial derivatives have been shown to be uniformly bounded by {\color{black}$C(T)$}. Combining all the established uniform bounds yields \eqref{uniformbounded}. 

\section{The Convergence of the PIA}\label{sec:convergence}
This section is devoted to Theorem~\ref{thm:convergence}, 
which establishes the PIA's convergence under time inconsistency. Compared with convergence results in the standard time-consistent case (see e.g., \cite{huang2025convergence, tran2025policy, ma2025convergence}), Theorem~\ref{thm:convergence} is distinct in two ways. First, it does not require any knowledge of the target value function---which is now an {\it equilibrium value function}, not even known to exist a priori. This is in contrast to the time-consistent case, where the target is the optimal value function, clearly-specified with well-understood regularity and PDE characterization. Second, Theorem~\ref{thm:convergence} establishes convergence {\it without} the ``policy improvement'' property (that the PIA keeps generating increasingly better values). This property holds in the time-consistent case, underlying the convergence analysis therein, but fails in general under time inconsistency (as noted in \cite{dai2023learning}). 


The proof of Theorem~\ref{thm:convergence} (see Section~\ref{subsect:proof of thm} below) shows that the value iterates $\{(V^{n,1}, V^{n,2})\}_{n \geq 1}$ in our PIA forms a Cauchy sequence (in particular, the distance between $(V^{n,1}, V^{n,2})$ and $(V^{n+1,1}, V^{n+1,2})$ decays exponentially in $n$) in a suitable complete space, thereby admitting a limit without the need of a known target value function or policy improvement. The precise result is as follows. 


\begin{theorem} \label{thm:convergence}
	Assume Assumption \ref{assump:regularity}. Initiate the PIA (Section~\ref{subsect:pia}) with an arbitrary $(V^{0,1}, V^{0,2})$ satisfying \eqref{V^0 condition}. 
    Then, $\{(V^{n,1}, V^{n,2})\}_{n\ge 1}$ converges in the space $\Theta_{[0, T]}^{(2)} \times C^{2}([0, T] \times \mathbb{R}^d)$ to some $(V^{*,1}, V^{*,2})$, and $\{\pi^n\}_{n\ge 1}$ converges uniformly on $[0,T]\times\mathbb R^d\times\mathbb A$ to some $\pi^*\in\mathcal A$, both at an exponential rate. 
    {\color{black}Specifically, there exists $C(T)>0$ such that}
	\begin{equation}\label{exp convergence V and pi}
		[V^{n,1} - V^{*,1}]_{[0, T]}^{(2)} + \|V^{n,2} - V^{*,2}\|^{(2)} + \|\pi^n - \pi^*\|^{(0)} \leq C(T) {\color{black}\left(\frac{3}{4}\right)^n}, \quad \forall n \geq 1.
	\end{equation}
    Moreover, $(V^{*,1}, V^{*,2})$ is the unique solution in $\Theta_{[0, T]}^{(2)} \times C^{2}([0, T] \times \mathbb{R}^d)$ to the EEHJB equation \eqref{eq:EEHJB} and $\pi^*\in\mathcal A$ is an equilibrium policy (Definition~\ref{def:equilibrium}) that takes the form 
	\begin{equation}\label{pi^*}
		\pi^*(t,x)(a)=\Gamma\bigl(t,x,Z^*(t,x),a\bigr),
	\end{equation}
	with 
	$
	Z^*(t,x):= V^{*,1}_x(t,t,x,x) + G_z(t,x,V^{*,2}(t,x)) \cdot V_x^{*,2}(t,x),
	$
	for all $(t,x,a)\in[0,T]\times\mathbb{R}^d\times\mathbb{A}$. 
\end{theorem}
To facilitate the proof of Theorem \ref{thm:convergence}, we first establish the following technical lemma.
\begin{lemma} \label{lemma:seq_decay}
	Let $C > 0$ and $0 < \rho < q < 1$. If a sequence of non-negative real numbers $\{a_n\}_{n \geq 0}$ satisfies 
	\begin{equation}\label{lemma:rec}
		a_n \leq \rho a_{n-1} + Cq^n, \quad n \geq 1,
	\end{equation}
	then
	\begin{equation}\label{eq:seq_decay_bound}
		a_n \leq \left( a_0 + \frac{Cq}{q - \rho} \right) q^n, \quad \forall n \geq 0.
	\end{equation}
\end{lemma}
\begin{proof}
	Iterating \eqref{lemma:rec} $n$ times yields
	\begin{equation*}
		a_n \leq \rho^n a_0 + C \sum_{k=1}^n \rho^{n-k} q^k = \rho^n a_0 + \frac{Cq}{q - \rho}(q^n - \rho^n) \leq \left( a_0 + \frac{Cq}{q - \rho} \right) q^n,
	\end{equation*}
	which yields \eqref{eq:seq_decay_bound}.
\end{proof}

\subsection{Proof of Theorem \ref{thm:convergence}}\label{subsect:proof of thm}
{\color{black}Throughout this proof, let $C \ge 1$ denote a generic constant that is independent of both $n$ and $T$, which may change from line to line; it depends only on $d, m, C_0$. When a constant depends on $T$, we denote it by $C(T)$. In Steps 1--3, $C(T)$ remains independent of $n$ and may additionally depend on $\lambda, |\mathbb A|, \|\delta\|^{(0)}, \|V^{0,1}\|_{[0,T]}^{(2)}, \|V^{0,2}\|^{(2)}$, the uniform bound in Proposition~\ref{prop:1}~(ii), and the bounds in Assumption~\ref{assump:regularity}~(iii) on $|z|\leq\max\{\|V^{0,2}\|^{(0)},\|h\|^{(0)}\}$.}

{\bf Step 1: Show that $\{(V^{n,1}, V^{n,2})\}_{n \geq 1}$ is Cauchy in the complete space $\Theta^{(2)}_{[0, T]} \times C^{2}([0, T] \times \mathbb{R}^d)$.}
Define $\Delta V^{n,i} := V^{n+1,i} - V^{n,i}$ and $\Delta f^{n,i} := f^{n+1,i} - f^{n,i}$, where $V^{n,i}$ and $f^{n,i}$ are given by \eqref{eq:rep1}, \eqref{eq:rep2}, \eqref{eq:f1}, and \eqref{eq:f2}, respectively, for $i=1,2$. Partition the time horizon $[0, T]$ into $K$ segments $0 = t_0 < t_1 < \dots < t_K = T$ such that $t_i - t_{i-1} \leq \epsilon$ for some sufficiently small $\epsilon > 0$ to be determined. For each  $i \in \{1, \dots, K\}$, we introduce the notations
    \begin{align*}
    \mathcal{E}_{0, i}^{n,1} &:= [\Delta V^{n,1}]^{(0)}_{[t_{i-1}, t_i]}, & 
    \mathcal{E}_{1, i}^{n,1} &:= [\Delta V^{n,1}_{x}]^{(0)}_{[t_{i-1}, t_i]}, \\
    \mathcal{E}_{2, i}^{n,1} &:= [\Delta  V^{n,1}_{xx}]^{(0)}_{[t_{i-1}, t_i]}, & 
    \mathcal{E}_{3, i}^{n,1} &:= [\Delta V^{n,1}_{xy}]^{(0)}_{[t_{i-1}, t_i]}, \\[1ex]
    \mathcal{E}_{0, i}^{n,2} &:= \|\Delta V^{n,2}\|^{(0)}_{[t_{i-1}, t_i]\times\mathbb{R}^d}, & 
    \mathcal{E}_{1, i}^{n,2} &:= \|\Delta V_x^{n,2}\|^{(0)}_{[t_{i-1}, t_i]\times\mathbb{R}^d}, \\
    \mathcal{E}_{2, i}^{n,2} &:= \|\Delta V_{xx}^{n,2}\|^{(0)}_{[t_{i-1}, t_i]\times\mathbb{R}^d}.
\end{align*}
By Lemma \ref{lemma:H_estimates}, the uniform bound \eqref{uniformbounded} for $\{(V^{n,1}, V^{n,2})\}_{n\ge 1}$, and Assumption \ref{assump:regularity} (iii), we have
\begin{align}
	|Z^n(s,x) - Z^{n-1}(s,x)|
	&\leq {\color{black}C(T)} \left[ |\Delta V^{n-1,1}_{x}(s,s,x,x)| + |G_z(s,x, V^{n,2}(s,x)) - G_z(s,x, V^{n-1,2}(s,x))| \right. \nonumber \\
	&\qquad \left. + |\Delta V^{n-1,2}_{x}(s,x)| \right] \nonumber \\
	&\leq {\color{black}C(T)} \left[ |\Delta V^{n-1,1}_{x}(s,s,x,x)| + |\Delta V^{n-1,2}(s,x)| + |\Delta V^{n-1,2}_{x}(s,x)| \right]. \label{estimate:zndiff}
\end{align}
Consequently, combining Lemma \ref{lemma:H_estimates} with \eqref{uniformbounded} and \eqref{estimate:zndiff}, we deduce that
{\small
\begin{align}
	|\Delta f^{n,1}(\tau, s, y, x)| &= \bigg| H_z^n(s, x) \cdot \left[ V_x^{n+1,1}(\tau, s, y, x) - \delta(s-\tau) Z^n(s, x) \right] \nonumber \\
	&\qquad - H_z^{n-1}(s, x) \cdot \left[ V_x^{n,1}(\tau, s, y, x) - \delta(s-\tau) Z^{n-1}(s, x) \right] \nonumber \\
	&\qquad + \delta(s-\tau) \left[ H^n(s, x) - H^{n-1}(s, x) \right] \nonumber \\
	&\qquad + \delta(s-\tau) \left[ \int_{\mathbb{A}} (r(y, s, x, a) - r(x, s, x, a)) (\Gamma^n - \Gamma^{n-1})(s, x, a) \, da \right] \bigg| \nonumber \\
	&\leq |H_z^n(s, x) - H_z^{n-1}(s, x)|  |V_x^{n+1,1}(\tau, s, y, x) - \delta(s-\tau) Z^n(s, x)| \nonumber \\
	&\qquad + |H_z^{n-1}(s, x)|  \left| \Delta V_x^{n,1}(\tau, s, y, x) - \delta(s-\tau)(Z^n - Z^{n-1})(s, x) \right| \nonumber \\
	&\qquad + |\delta(s-\tau)|  |H^n(s, x) - H^{n-1}(s, x)| \nonumber \\
	&\qquad + |\delta(s-\tau)| \, \bigg| \int_{\mathbb{A}} (r(y, s, x, a) - r(x, s, x, a)) (\Gamma^n - \Gamma^{n-1})(s, x, a) \, da \bigg| \nonumber \\
	&\leq {\color{black}C(T)} \left( |\Delta V_x^{n,1}(\tau, s, y, x)| + |Z^n(s, x) - Z^{n-1}(s, x)| \right) \nonumber \\
	&\leq {\color{black}C(T)} \Big( |\Delta V_x^{n,1}(\tau, s, y, x)| + |\Delta V^{n-1,1}_{x}(s, s, x, x)| \nonumber \\
	&\qquad \quad + |\Delta V^{n-1,2}(s, x)| + |\Delta V^{n-1,2}_{x}(s, x)| \Big), \label{eq:df1_bound} \\[1.5ex]
	|\Delta f^{n,2}(s, x)| &= \left| H_z^n(s, x) \cdot V_x^{n+1,2}(s, x) - H_z^{n-1}(s, x) \cdot V_x^{n,2}(s, x) \right| \nonumber \\
	&\leq |H_z^n(s, x) - H_z^{n-1}(s, x)|  |V_x^{n+1,2}(s, x)| + |H_z^{n-1}(s, x)|  |\Delta V_x^{n,2}(s, x)| \nonumber \\
	&\leq {\color{black}C(T)} \left( |\Delta V_x^{n,2}(s, x)| + |Z^n(s, x) - Z^{n-1}(s, x)| \right) \nonumber \\
	&\leq {\color{black}C(T)} \Big( |\Delta V_x^{n,2}(s, x)| + |\Delta V^{n-1,1}_{x}(s, s, x, x)| \nonumber \\
	&\qquad \quad + |\Delta V^{n-1,2}(s, x)| + |\Delta V^{n-1,2}_{x}(s, x)| \Big). \label{eq:df2_bound}
\end{align}
}

Then, utilizing the representation formulas \eqref{rep_first1}, \eqref{rep_first2} and \eqref{eq:rep2}, together with the  estimates \eqref{eq:df1_bound} and \eqref{eq:df2_bound}, we find that for any $(\tau, t, y, x) \in \Delta[t_{i-1}, t_i] \times \mathbb{R}^d \times \mathbb{R}^d$,
		\begin{align*}
			|\Delta V^{n,1}_x(\tau, t, y, x)| &\leq \mathbb{E} \left[ |\nabla X_{t_i}^{t, x}| |\Delta V_{x}^{n,1}(\tau, t_i, y, X_{t_i}^{t, x})| \right] + \mathbb{E} \left[ \int_t^{t_i} |\Delta f^{n,1}(\tau, s, y, X_s^{t, x})| |N_s^{t, x}| ds \right] \\
			&\leq {\color{black}C(T)} e^{C\epsilon} \bigg( \mathcal{E}_{1, i+1}^{n,1} + \sqrt{\epsilon} \left( \mathcal{E}_{1, i}^{n,1} + \mathcal{E}_{1, i}^{n-1,1} + \mathcal{E}_{1, i}^{n-1,2} + \mathcal{E}_{0, i}^{n-1,2} \right) \bigg), \\
			|\Delta V^{n,2}_x(t, x)| &\leq \mathbb{E} \left[ |\nabla X_{t_i}^{t, x}| |\Delta V_{x}^{n,2}(t_i, X_{t_i}^{t, x})| \right] + \mathbb{E} \left[ \int_t^{t_i} |\Delta f^{n,2}(s, X_s^{t, x})| |N_s^{t, x}| ds \right] \\
			&\leq {\color{black}C(T)} e^{C\epsilon} \bigg( \mathcal{E}_{1, i+1}^{n,2} + \sqrt{\epsilon} \left( \mathcal{E}_{1, i}^{n,2} + \mathcal{E}_{1, i}^{n-1,1} + \mathcal{E}_{1, i}^{n-1,2} + \mathcal{E}_{0, i}^{n-1,2} \right) \bigg), \\
			|\Delta V^{n,2}(t, x)| &\leq \mathbb{E} \left[ |\Delta V^{n,2}(t_i, X_{t_i}^{t, x})| \right] + \mathbb{E} \left[ \int_t^{t_i} |\Delta f^{n,2}(s, X_s^{t, x})| ds \right] \\
			&\leq {\color{black}C(T)} e^{C\epsilon} \bigg( \mathcal{E}_{0, i+1}^{n,2} + \epsilon \left( \mathcal{E}_{1, i}^{n,2} + \mathcal{E}_{1, i}^{n-1,1} + \mathcal{E}_{1, i}^{n-1,2} + \mathcal{E}_{0, i}^{n-1,2} \right) \bigg).
		\end{align*}
		Summing these estimates and noting that $\epsilon < \sqrt{\epsilon}$ for small $\epsilon$, we obtain
		\begin{align*}
			\mathcal{E}_{1, i}^{n,1} + \mathcal{E}_{1, i}^{n,2} + \mathcal{E}_{0, i}^{n,2} &\leq {\color{black}C(T)} e^{C\epsilon} \bigg( \mathcal{E}_{1, i+1}^{n,1} + \mathcal{E}_{1, i+1}^{n,2} + \mathcal{E}_{0, i+1}^{n,2} \\
			&\quad + \sqrt{\epsilon} \left( \mathcal{E}_{1, i}^{n,1} + \mathcal{E}_{1, i}^{n,2} + \mathcal{E}_{0, i}^{n,2} + \mathcal{E}_{1, i}^{n-1,1} + \mathcal{E}_{1, i}^{n-1,2} + \mathcal{E}_{0, i}^{n-1,2} \right) \bigg).
		\end{align*}
		{\color{black}By choosing $0 < \epsilon \leq \min\big\{ T, \frac{1}{C}, \frac{1}{9e^2(C(T))^2} \big\}$, we have $C(T)e^{C\epsilon}\sqrt\epsilon\leq1/3$. Rearranging terms yields the recursive relation}
		\begin{equation} \label{eq:local_recursion_grad}
			\mathcal{E}_{1, i}^{n,1} + \mathcal{E}_{1, i}^{n,2} + \mathcal{E}_{0, i}^{n,2} \leq \frac{1}{2} \left( \mathcal{E}_{1, i}^{n-1,1} + \mathcal{E}_{1, i}^{n-1,2} + \mathcal{E}_{0, i}^{n-1,2} \right) + {\color{black}C(T)} \left( \mathcal{E}_{1, i+1}^{n,1} + \mathcal{E}_{1, i+1}^{n,2} + \mathcal{E}_{0, i+1}^{n,2} \right).
		\end{equation}
		{\color{black}On the last subinterval, the terminal errors vanish for $n\geq1$, so the errors are bounded by $C(T)(3/4)^n$. Moreover, applying Lemma~\ref{lemma:seq_decay} with $\rho=1/2$ and $q=3/4$ preserves the rate on every preceding subinterval:
		\begin{equation*}
			a_n\leq\frac12a_{n-1}+C(T)\left(\frac{3}{4}\right)^n
			\quad\Longrightarrow\quad
			a_n\leq\big(a_0+3C(T)\big)\left(\frac{3}{4}\right)^n.
		\end{equation*}
		Backward induction over the finite partition leads to}
\begin{equation} \label{eq:cauchy1}
	[V_x^{n+1,1} - V_x^{n,1}]^{(0)}_{[0,T]} + \|V_x^{n+1,2} - V_x^{n,2}\|^{(0)} + \|V^{n+1,2} - V^{n,2}\|^{(0)}\leq {\color{black}C(T) \left(\frac{3}{4}\right)^n}.
\end{equation}
Moreover, we deduce from the Feynman--Kac representation \eqref{eq:rep1} that
\begin{equation} \label{eq:val_cauchy}
	|\Delta V^{n,1}(\tau, t, y, x)| \leq \mathbb{E} \left[ \int_t^T |\Delta f^{n,1}(\tau, s, y, X^{t,x}_s)| ds \right] \leq {\color{black}C(T) \left(\frac{3}{4}\right)^{n-1}},
\end{equation}
where the second inequality follows  from  \eqref{eq:df1_bound} and \eqref{eq:cauchy1}.

Similarly, for $\Delta V_{xy}^{n,1}$, we can first deduce the bound for $\Delta f_{y}^{n,1}$, by Lemma \ref{lemma:H_estimates} and \eqref{uniformbounded}, as follows:
\begin{align}
	|\Delta f_{y}^{n,1}(\tau, s, y, x)|&=|(H_{z}^n(s, x)-H_{z}^{n-1}(s, x))\cdot V^{n+1,1}_{xy}(\tau,s,y,x)+H_{z}^{n-1}(s, x)\cdot \Delta V^{n,1}_{xy}(\tau,s,y,x)\notag\\
	&\quad+\delta(s-\tau)\int_{\mathbb{A}}r_y(y,s,x,a)\left(\Gamma^n-\Gamma^{n-1}\right)(s,x,a)da|\notag\\&\leq {\color{black}C(T)}\left[|Z^n(s,x)-Z^{n-1}(s,x)|+| \Delta V^{n,1}_{xy}(\tau,s,y,x)|\right]\notag\\
	&\leq {\color{black}C(T)}\Big(|\Delta V^{n-1,1}_{x}(s,s,x,x)|+ |\Delta V^{n-1,2}(s,x)|\notag\\
	&\qquad\quad+ |\Delta V^{n-1,2}_{x}(s,x)|+| \Delta V^{n,1}_{xy}(\tau,s,y,x)|\Big).\label{eq:df1y_bound}
\end{align}
 Substituting this into the formula \eqref{rep_first1y} and utilizing the established convergence rates \eqref{eq:cauchy1}, we have
\begin{align*}
	|\Delta V_{xy}^{n,1}(\tau, t,y, x)| &\leq \mathbb{E} \left[ |\nabla X_{t_i}^{t, x}| |\Delta V_{xy}^{n,1}(\tau, t_i, y, X_{t_i}^{t, x})| \right] + \mathbb{E} \left[ \int_t^{t_i} |\Delta f^{n,1}_y(\tau, s, y, X_s^{t, x})| |N_s^{t, x}| ds \right] \\
	&\leq {\color{black}C(T)} e^{C\epsilon} \bigg( \mathcal{E}_{3, i+1}^{n,1} + \sqrt{\epsilon} \left( \mathcal{E}_{3, i}^{n,1} + \mathcal{E}_{1, i}^{n-1,1} + \mathcal{E}_{1, i}^{n-1,2} + \mathcal{E}_{0, i}^{n-1,2} \right) \bigg)\\
	&\leq {\color{black}C(T)} e^{C\epsilon} \bigg( \mathcal{E}_{3, i+1}^{n,1} + \sqrt{\epsilon} \left( \mathcal{E}_{3, i}^{n,1} + {\color{black}\left(\frac{3}{4}\right)^{n-1}} \right) \bigg).
\end{align*}
{\color{black}By choosing $\epsilon>0$ smaller if necessary, with $0<\epsilon\leq \min\big\{ T, \frac1C, \frac{1}{9e^2(C(T))^2} \big\}$, and refining the partition accordingly, we have $C(T)e^{C\epsilon}\sqrt\epsilon\leq1/3$. Rearranging terms yields the recurrence}
\begin{equation*}
	\mathcal{E}_{3, i}^{n,1} \leq {\color{black}C(T)} \left( \mathcal{E}_{3, i+1}^{n,1} + {\color{black}\left(\frac{3}{4}\right)^{n-1}} \right).
\end{equation*}
For any $n\ge1$, given that $\mathcal{E}_{3, K+1}^{n,1}=0$, a backward induction from $i=K$ to $i=1$ leads to
\begin{equation}\label{eq:cauchy2}
	[V_{xy}^{n+1,1} - V_{xy}^{n,1}]^{(0)}_{[0,T]} \leq {\color{black}C(T) \left(\frac{3}{4}\right)^{n-1}}.
\end{equation}		
	Next, we proceed to estimate $\Delta V_{xx}^{n,1}$ and $\Delta V_{xx}^{n,2}$. Following the same reasoning in \eqref{estimate:zndiff}, we have
{\small\begin{align}\label{estimate:zxndiff}
	|Z_x^n(s,x)-Z_x^{n-1}(s,x)|&\leq|\Delta V^{n-1,1}_{xy}(s,s,x,x)|+|\Delta V^{n-1,1}_{xx}(s,s,x,x)|\nonumber\\&+|G_{zx}(s,x,V^{n,2}(s,x))-G_{zx}(s,x,V^{n-1,2}(s,x))||V^{n,2}_{x}(s,x)|\nonumber\\&+|G_{zz}(s,x,V^{n,2}(s,x))\cdot V^{n,2}_{x}(s,x)-G_{zz}(s,x,V^{n-1,2}(s,x))\cdot V^{n-1,2}_{x}(s,x)||V^{n,2}_{x}(s,x)|\nonumber\\&+|G_{zx}(s,x,V^{n-1,2}(s,x))+G_{zz}(s,x,V^{n-1,2}(s,x))\cdot V^{n-1,2}_{x}(s,x)||\Delta V^{n-1,2}_{x}(s,x)|\nonumber\\&+|G_{z}(s,x,V^{n,2}(s,x))\otimes V^{n,2}_{xx}(s,x)-G_{z}(s,x,V^{n-1,2}(s,x))\otimes V^{n-1,2}_{xx}(s,x)|\nonumber\\&\leq {\color{black}C(T)}\bigg(|\Delta V^{n-1,1}_{xy}(s,s,x,x)|+|\Delta V^{n-1,1}_{xx}(s,s,x,x)|+ |\Delta V^{n-1,2}(s,x)|\nonumber\\&+|\Delta V^{n-1,2}_{x}(s,x)|+|\Delta V^{n-1,2}_{xx}(s,x)|\bigg).
\end{align}}
	Combining \eqref{uniformbounded}, \eqref{estimate:zndiff}, \eqref{estimate:zxndiff}, Lemma \ref{lemma:H_estimates}, and Assumption \ref{assump:regularity} (iii), we obtain
        {\small
\begin{align}
	|\Delta f_x^{n,1}(\tau, s, y, x)| 
	&\leq \Big| (H_{zx}^n(s, x) + H_{zz}^n(s, x)\cdot Z_x^n(s, x)) - (H_{zx}^{n-1}(s, x) + H_{zz}^{n-1}(s, x)\cdot Z_x^{n-1}(s, x)) \Big| \nonumber \\
	&\qquad \times |V^{n+1,1}_x(\tau, s, y,x) - \delta(s-\tau) Z^n(s, x) | \nonumber \\
	&\quad + | H_{zx}^{n-1}(s, x) + H_{zz}^{n-1}(s, x)\cdot Z_{x}^{n-1}(s, x) | \nonumber \\
	&\qquad \times \Big| \Delta V_{x}^{n,1}(\tau, s, y, x) - \delta(s-\tau) (Z^n-Z^{n-1})(s, x) \Big| \nonumber \\
	&\quad + | H_z^n(s, x) | \Big| \Delta V_{xx}^{n,1}(\tau, s, y, x) - \delta(s-\tau)(Z_{x}^{n}(s, x) -Z_{x}^{n-1}(s, x)) \Big| \nonumber \\
	&\quad + | H_z^n(s, x) - H_z^{n-1}(s, x) |  | V_{xx}^{n,1}(\tau, s, y, x) - \delta(s-\tau) Z_x^{n-1}(s, x) | \nonumber \\
	&\quad + \delta(s-\tau) \Big| (H_x^n(s, x) + H_z^n(s, x)\cdot Z_x^n(s, x)) - (H_x^{n-1}(s, x) + H_z^{n-1}(s, x)\cdot Z_x^{n-1}(s, x)) \Big| \nonumber \\
	&\quad + \delta(s-\tau) \Big| \int_{\mathbb{A}}(r_x(y,s,x,a)-r_y(x,s,x,a)-r_x(x,s,x,a))(\Gamma^n-\Gamma^{n-1})(s,x,a)da \Big| \nonumber \\
	&\quad + \delta(s-\tau) \Big| \int_{\mathbb{A}}(r(y,s,x,a)-r(x,s,x,a))(\partial_x\Gamma^n-\partial_x\Gamma^{n-1})(s,x,a)da \Big| \nonumber \\
	&\leq {\color{black}C(T)} \bigg( |Z^n(s,x)-Z^{n-1}(s,x)| + |Z_x^n(s,x)-Z_x^{n-1}(s,x)| \nonumber \\
	&\qquad \quad + |\Delta V_{x}^{n,1}(\tau, s, y, x)| + |\Delta V_{xx}^{n,1}(\tau, s, y, x)| \bigg) \nonumber \\
	&\leq {\color{black}C(T)} \bigg( |\Delta V_{x}^{n,1}(\tau, s, y, x)| + |\Delta V_{xx}^{n,1}(\tau, s, y, x)| + |\Delta V^{n-1,1}_{xy}(s,s,x,x)| \nonumber \\
	&\qquad \quad + |\Delta V^{n-1,1}_{xx}(s,s,x,x)| + |\Delta V^{n-1,1}_{x}(s,s,x,x)| + |\Delta V^{n-1,2}(s,x)| \nonumber \\
	&\qquad \quad + |\Delta V^{n-1,2}_{x}(s,x)| + |\Delta V^{n-1,2}_{xx}(s,x)| \bigg), \label{eq:df1x_bound} \\[1.5ex]
	|\Delta f_x^{n,2}(s, x)| 
	&\leq \Big| (H_{zx}^n(s, x) + H_{zz}^n(s, x)\cdot Z_{x}^n(s, x)) - (H_{zx}^{n-1}(s, x) + H_{zz}^{n-1}(s, x)\cdot Z_{x}^{n-1}(s, x)) \Big| \nonumber \\
	&\qquad \times | V^{n+1,2}_x(s, x) | \nonumber \\
	&\quad + | H_{zx}^{n-1}(s, x) + H_{zz}^{n-1}(s, x)\cdot Z_{x}^{n-1}(s, x) |  \Big| \Delta V_x^{n,2}(s, x) \Big| \nonumber \\
	&\quad + | H_z^n(s, x) |  \Big| \Delta V_{xx}^{n,2}(s, x) \Big| + | H_z^n(s, x) - H_z^{n-1}(s, x) |  | V_{xx}^{n,2}(s, x) | \nonumber \\
	&\leq {\color{black}C(T)} \bigg( |Z^n(s,x)-Z^{n-1}(s,x)| + |Z_x^n(s,x)-Z_x^{n-1}(s,x)| \nonumber \\
	&\qquad \quad + |\Delta V_{x}^{n,2}(s,x)| + |\Delta V_{xx}^{n,2}(s,x)| \bigg) \nonumber \\
	&\leq {\color{black}C(T)} \bigg( |\Delta V_{x}^{n,2}(s,x)| + |\Delta V_{xx}^{n,2}(s,x)| + |\Delta V^{n-1,1}_{xy}(s,s,x,x)| \nonumber \\
	&\qquad \quad + |\Delta V^{n-1,1}_{xx}(s,s,x,x)| + |\Delta V^{n-1,1}_{x}(s,s,x,x)| + |\Delta V^{n-1,2}(s,x)| \nonumber \\
	&\qquad \quad + |\Delta V^{n-1,2}_{x}(s,x)| + |\Delta V^{n-1,2}_{xx}(s,x)| \bigg). \label{eq:df2x_bound}
\end{align}
}Now, for any $(\tau, t, y, x) \in \Delta[t_{i-1}, t_i] \times \mathbb{R}^d \times \mathbb{R}^d$, the representations in \eqref{rep_second1} and \eqref{rep_second2} yield
		\begin{align*}
			|\Delta V_{xx}^{n,1}(\tau, t, y, x)| &\leq \mathbb{E} \left[ | (\nabla X_{t_i}^{t, x})^\top (\Delta V_{xx}^{n,1}(\tau, t_i, y, X^{t, x}_{t_i})) \nabla X_{t_i}^{t, x} + \Delta V_x^{n,1}(\tau, t_i, y, X^{t, x}_{t_i}) \otimes \nabla^2 X_{t_i}^{t, x} | \right] \\
			&\quad + \mathbb{E} \left[ \int_t^{t_i} \left( | N_s^{t, x} (\Delta f_x^{n,1}\left(\tau,s, y, X_s^{t, x}\right))^\top \nabla X_s^{t, x} | + | \Delta f^{n,1}\left(\tau,s,y, X_s^{t, x}\right) \nabla N_s^{t, x} | \right) ds \right],\\
			|\Delta V_{xx}^{n,2}(t, x)| &\leq \mathbb{E} \left[ | (\nabla X_{t_i}^{t, x})^\top (\Delta V_{xx}^{n,2}(t_i, X^{t, x}_{t_i})) \nabla X_{t_i}^{t, x} + \Delta V_x^{n,2}(t_i, X^{t, x}_{t_i}) \otimes \nabla^2 X_{t_i}^{t, x} | \right] \\
			&\quad + \mathbb{E} \left[ \int_t^{t_i} \left( | N_s^{t, x} (\Delta f_x^{n,2}\left(s, X_s^{t, x}\right))^\top \nabla X_s^{t, x} | + | \Delta f^{n,2}\left(s, X_s^{t, x}\right) \nabla N_s^{t, x} | \right) ds \right].
		\end{align*}
        Plugging \eqref{estimate_1}, \eqref{eq:df1_bound}, \eqref{eq:df2_bound}, \eqref{eq:df1x_bound} and \eqref{eq:df2x_bound} into the above estimates,  we obtain
		\begin{align*}
			\mathcal{E}_{2, i}^{n,1} \leq C e^{C\epsilon} (\mathcal{E}_{2, i+1}^{n,1} + \mathcal{E}_{1, i+1}^{n,1}) &+ {\color{black}C(T)} e^{C\epsilon} \sqrt{\epsilon} \bigg( \mathcal{E}_{1, i}^{n,1} + \mathcal{E}_{2, i}^{n,1} + \mathcal{E}_{1, i}^{n-1,1} + \mathcal{E}_{0, i}^{n-1,2}\\&+\mathcal{E}_{1, i}^{n-1,2} + \mathcal{E}_{2, i}^{n-1,1} + \mathcal{E}_{2, i}^{n-1,2} + \mathcal{E}_{3, i}^{n-1,1} \bigg),\\
				\mathcal{E}_{2, i}^{n,2} \leq C e^{C\epsilon} (\mathcal{E}_{2, i+1}^{n,2} + \mathcal{E}_{1, i+1}^{n,2}) &+ {\color{black}C(T)} e^{C\epsilon} \sqrt{\epsilon} \bigg( \mathcal{E}_{1, i}^{n,2} + \mathcal{E}_{2, i}^{n,2} + \mathcal{E}_{1, i}^{n-1,1} + \mathcal{E}_{0, i}^{n-1,2}\\&+\mathcal{E}_{1, i}^{n-1,2} + \mathcal{E}_{2, i}^{n-1,1} + \mathcal{E}_{2, i}^{n-1,2} + \mathcal{E}_{3, i}^{n-1,1} \bigg).
		\end{align*}
By summing these two inequalities and using \eqref{eq:cauchy1} and \eqref{eq:cauchy2}, we get
			\begin{align*}
		\mathcal{E}_{2, i}^{n,1}+\mathcal{E}_{2, i}^{n,2} &\leq C e^{C\epsilon} (\mathcal{E}_{2, i+1}^{n,1} + \mathcal{E}_{2, i+1}^{n,2} + \mathcal{E}_{1, i+1}^{n,1}+ \mathcal{E}_{1, i+1}^{n,2}) + {\color{black}C(T)} e^{C\epsilon} \sqrt{\epsilon} \bigg(\mathcal{E}_{2, i}^{n,1}+ \mathcal{E}_{2, i}^{n,2} +\mathcal{E}_{2, i}^{n-1,1} \\&\quad+ \mathcal{E}_{2, i}^{n-1,2} + \mathcal{E}_{3, i}^{n-1,1}+\mathcal{E}_{1, i}^{n,1}+ \mathcal{E}_{1, i}^{n,2} +\mathcal{E}_{1, i}^{n-1,1} +\mathcal{E}_{1, i}^{n-1,2}+\mathcal{E}_{0, i}^{n-1,2}\bigg)\\
		&\leq  {\color{black}C(T)} e^{C\epsilon} \left(\mathcal{E}_{2, i+1}^{n,1} + \mathcal{E}_{2, i+1}^{n,2} + {\color{black}\left(\frac{3}{4}\right)^n}\right) + {\color{black}C(T)} e^{C\epsilon} \sqrt{\epsilon} \bigg(\mathcal{E}_{2, i}^{n,1}+ \mathcal{E}_{2, i}^{n,2} +\mathcal{E}_{2, i}^{n-1,1} + \mathcal{E}_{2, i}^{n-1,2} + {\color{black}\left(\frac{3}{4}\right)^n}\bigg).
		\end{align*}
		{\color{black}By choosing $\epsilon>0$ smaller if necessary, with $0<\epsilon\leq \min\big\{ T, \frac1C, \frac{1}{9e^2(C(T))^2} \big\}$, and refining the partition accordingly, we have $C(T)e^{C\epsilon}\sqrt\epsilon\leq1/3$. We deduce by rearranging terms that}
		\begin{equation*}
			\mathcal{E}_{2, i}^{n,1}+\mathcal{E}_{2, i}^{n,2} \leq \frac{1}{2}(\mathcal{E}_{2, i}^{n-1,1} + \mathcal{E}_{2, i}^{n-1,2} ) + {\color{black}C(T)} \left( \mathcal{E}_{2, i+1}^{n,1} + \mathcal{E}_{2, i+1}^{n,2} + {\color{black}\left(\frac{3}{4}\right)^n} \right).
		\end{equation*}
		{\color{black}Note that the terminal errors vanish for $n\geq1$, backward induction and Lemma~\ref{lemma:seq_decay}, with $\rho=1/2$ and $q=3/4$, retain the rate $3/4$ such that}
		\begin{equation}
			[V_{xx}^{n+1,1} - V_{xx}^{n,1}]^{(0)}_{[0,T]} + 	\|V_{xx}^{n+1,2} - V_{xx}^{n,2}\|^{(0)} \leq {\color{black}C(T) \left(\frac{3}{4}\right)^n}. \label{eq:hess_cauchy}
		\end{equation}

Finally, we deal with the time derivatives $V_t^{n,1}$ and $V_t^{n,2}$. From the recursive linear PDE \eqref{eq:recursive},
\begin{align*}
	|\Delta V_t^{n,1}(\tau, t, y, x)| 
	&\leq \frac{1}{2} \left| \text{tr}\left(\sigma \sigma^\top(t, x)  \Delta V_{xx}^{n,1}(\tau, t,y, x)\right) \right| + |\Delta f^{n,1}(\tau, t, y, x)|,\\	
	|\Delta V_t^{n,2}(t, x)| 
	&\leq \frac{1}{2} \left| \text{tr}\left(\sigma \sigma^\top(t, x)  \Delta V_{xx}^{n,2}(t,x)\right) \right| + |\Delta f^{n,2}(t,x)|.
\end{align*}
This, along with the exponential decay in \eqref{eq:cauchy1} and \eqref{eq:hess_cauchy} and the  estimates \eqref{eq:df1_bound} and \eqref{eq:df2_bound}, implies
\begin{equation}
	[ V_t^{n+1,1} -  V_t^{n,1}]^{(0)}_{[0,T]} +	\| V_t^{n+1,2} -  V_t^{n,2}\|^{(0)} \leq {\color{black}C(T) \left(\frac{3}{4}\right)^n}. \label{eq:time_cauchy}
\end{equation}
	
    Collectively, the estimates \eqref{eq:cauchy1}, \eqref{eq:val_cauchy}, \eqref{eq:hess_cauchy}, and \eqref{eq:time_cauchy} yield
	\begin{equation}\label{exp decay}
		[V^{n+1,1} - V^{n,1}]^{(2)}_{[0,T]} + \|V^{n+1,2} - V^{n,2}\|^{(2)} \leq {\color{black}C(T) \left(\frac{3}{4}\right)^n}.
	\end{equation}
	Hence, $\{(V^{n,1}, V^{n,2})\}_{n \geq 1}$ is a Cauchy sequence in the Banach space $\Theta_{[0, T]}^{(2)} \times C^{2}([0, T] \times \mathbb{R}^d)$.

    {\bf Step 2: Show the exponential convergence \eqref{exp convergence V and pi}.} 
    As $\{(V^{n,1}, V^{n,2})\}_{n \geq 1}$ is Cauchy in $\Theta_{[0, T]}^{(2)} \times C^{2}([0, T] \times \mathbb{R}^d)$, it admits a limit $(V^{*,1}, V^{*,2})$ in the same space, i.e., 
	\begin{equation}\label{convergence in Banach}
		\lim_{n \to \infty} \left( [V^{n,1} - V^{*,1}]^{(2)}_{[0,T]} + \|V^{n,2} - V^{*,2}\|^{(2)} \right) = 0.
	\end{equation}
    Furthermore, summing up the geometric increments in \eqref{exp decay} yields the exponential  convergence
	\begin{equation}\label{exp_conv1}
    [V^{n,1} - V^{*,1}]^{(2)}_{[0,T]} + \|V^{n,2} - V^{*,2}\|^{(2)} \leq {\color{black}C(T) \left(\frac{3}{4}\right)^n}.
	\end{equation}
    On the other hand, consider $\pi^*$ defined by \eqref{pi^*}. Given the uniform bound \eqref{uniformbounded} for $\{(V^{n,1}, V^{n,2})\}_{n \geq 1}$ and Assumption \ref{assump:regularity} (iii), it can be checked directly that
\begin{align*}
	|\pi^n(t, x)(a) &- \pi^*(t, x)(a)| \leq {\color{black}C(T)} \left| Z^{n-1}(t, x) - Z^*(t, x) \right| \\
	&\leq {\color{black}C(T)} \Big[ |V_{x}^{n-1,1}(t, t, x, x) - V_{x}^{*,1}(t, t, x, x)| \\
	&\qquad + |G_{z}(t, x, V^{n-1,2}(t, x)) - G_{z}(t, x, V^{*,2}(t, x))| 
    + |V_{x}^{n-1,2}(t, x) - V_{x}^{*,2}(t, x)| \Big].
\end{align*}
Applying \eqref{exp_conv1} to the right-hand side above yields
 $   \|\pi^n - \pi^*\|^{(0)} \leq {\color{black}C(T) \left(\frac{3}{4}\right)^{n-1}}$.

{{\bf Step 3: Show that $\pi^*$ is an equilibrium policy.} 
Let us first prove $\pi^*\in\mathcal A$. As $(V^{*,1}, V^{*,2}) \in \Theta_{[0, T]}^{(2)} \times C^{2}([0, T] \times \mathbb{R}^d)$, $Z^*(t, x)$ is jointly continuous in $(t, x)$, so that $\pi^*(t, x, a) = \Gamma(t, x, Z^*(t, x), a)$ is also jointly continuous in $(t, x)$ for each $a \in \mathbb{A}$. For any $x, x' \in \mathbb{R}^d$ and fixed $t \in [0, T]$, it follows from Assumption \ref{assump:regularity}, Lemma \ref{lemma:H_estimates}, and the boundedness of the limit $(V^{*,1}, V^{*,2})$ that
\begin{align}
	|\tilde{b}(t, x, \pi^*(t,x)) - \tilde{b}(t, x', \pi^*(t,x'))| 
	&\leq \int_{\mathbb{A}} |b(t, x, a) - b(t, x', a)| \Gamma(t, x, Z^*(t,x), a) da \notag\\
	&\quad + \int_{\mathbb{A}} |b(t, x', a)| \left| \Gamma(t, x, Z^*(t,x), a) - \Gamma(t, x', Z^*(t,x), a) \right| da \notag\\
    &\quad + \int_{\mathbb{A}} |b(t, x', a)| \left| \Gamma(t, x', Z^*(t,x), a) - \Gamma(t, x', Z^*(t,x'), a) \right| da \notag\\
	&\leq {\color{black}C(T)} \left( |x-x'| + |Z^*(t,x) - Z^*(t,x')| \right).\label{b-b}
\end{align}
Note that on the right-hand side above,
\begin{align}
    |Z^*(t,x) &- Z^*(t,x')| \leq | V_x^{*,1}(t, t, x, x) -  V_x^{*,1}(t, t, x', x')| \notag\\
    &\quad + \left| G_z(t, x, V^{*,2}(t, x))\cdot V_x^{*,2}(t, x) - G_z(t, x', V^{*,2}(t, x'))\cdot  V_x^{*,2}(t, x') \right|.\label{Z-Z}
\end{align}
Also, as the sequence $\{(V^{n,1}, V^{n,2})\}_{n \geq 1}$ is uniformly bounded in the sense of \eqref{uniformbounded}, for any $n\ge 1$,
\begin{align*}
    &|V_x^{n,1}(t,t,x,x) -  V_x^{n,1}(t,t,x',x')| \\
    &\leq | V_x^{n,1}(t,t,x,x) - V_x^{n,1}(t,t,x',x)| + | V_x^{n,1}(t,t,x',x) -  V_x^{n,1}(t,t,x',x')| \\
    &\leq \left( [ V_{xy}^{n,1}]^{(0)}_{[0,T]} + [ V_{xx}^{n,1}]^{(0)}_{[0,T]} \right) |x-x'| 
    \leq {\color{black}C(T)} |x-x'|.
\end{align*}
 As $n \to \infty$, we get
 $    |V_x^{*,1}(t,t,x,x) -  V_x^{*,1}(t,t,x',x')| \leq {\color{black}C(T)} |x-x'|$.
Combining this with \eqref{Z-Z}, Assumption \ref{assump:regularity} (iii), and $\|V^{*,2}\|^{(2)} \leq {\color{black}C(T)}$, we deduce from \eqref{b-b} that the drift $\tilde{b}(t, x, \pi^*(t,x))$ is Lipschitz continuous in $x$ uniformly in $t$. This ensures the existence of a unique strong solution to \eqref{sde_1} under $\pi^*$. {\color{black}Finally, by the same arguments in \eqref{admiss_0} and \eqref{admiss_1} (with $\pi^1$ replaced by $\pi^*$), we see that \eqref{eq:entropy_integrability} and \eqref{eq:local_unif_bound} are fulfilled. Hence, we conclude $\pi^*\in\mathcal A$ (Definition \ref{def:adm}).

We next show that the perturbed policy $\pi^*_{t,\epsilon,\pi_0}$ defined in \eqref{perturbed} belongs to $\mathcal A$ for every $t\in[0,T)$, $\pi_0\in\mathcal A$, and $\epsilon\in(0,T-t)$. The definition \eqref{perturbed} ensures that $\pi^*_{t,\epsilon,\pi_0}$ inherits the measurability and continuity requirements of Definition~\ref{def:adm}~(i) from $\pi_0$ on $[0,t+\epsilon)$ and from $\pi^*$ on $[t+\epsilon,T]$. Moreover, for any initial pair $(s_0, x_0)$ with $s_0 \geq t+\epsilon$, since $\pi^*_{t,\epsilon,\pi_0}$ coincides with $\pi^*$ on $[s_0, T]$, conditions (ii)--(iv) of Definition~\ref{def:adm} hold directly from $\pi^*\in\mathcal A$. It thus remains to verify conditions (ii)--(iv) for initial pairs $(s_0, x_0) \in [0, t+\epsilon) \times \mathbb R^d$.

Regarding condition~(ii), since $\pi_0\in\mathcal A$, SDE \eqref{sde_1} under $\pi_0$ admits a unique strong solution $\{X_s^{\pi_0}\}_{s\in[s_0, t+\epsilon]}$; restarting the equation at time $t+\epsilon$ from initial state $X_{t+\epsilon}^{\pi_0}$, it extends uniquely to $[t+\epsilon,T]$ under $\pi^*\in\mathcal A$, yielding the unique strong solution $\{X_s^{\pi^*_{t,\epsilon,\pi_0}}\}_{s\in[s_0,T]}$. For condition~(iii), by \eqref{eq:entropy_integrability}, splitting the time integral at $t+\epsilon$ yields
\begin{align*}
	&\mathbb{E}_{s_0,x_0}\left[ \int_{s_0}^T \int_{\mathbb{A}} \pi^*_{t,\epsilon,\pi_0}(s, X_s^{\pi^*_{t,\epsilon,\pi_0}}, a) \left|\ln \pi^*_{t,\epsilon,\pi_0}(s, X_s^{\pi^*_{t,\epsilon,\pi_0}}, a)\right| da \, ds \right] \\
	&= \mathbb{E}_{s_0,x_0}\left[ \int_{s_0}^{t+\epsilon} \int_{\mathbb{A}} \pi_0(s, X_s^{\pi_0}, a) \left|\ln \pi_0(s, X_s^{\pi_0}, a)\right| da \, ds \right] \\
	&\quad + \mathbb{E}_{s_0,x_0}\left[ \mathbb{E}_{t+\epsilon, X_{t+\epsilon}^{\pi_0}}\left[ \int_{t+\epsilon}^T \int_{\mathbb{A}} \pi^*(s, X_s^{\pi^*}, a) \left|\ln \pi^*(s, X_s^{\pi^*}, a)\right| da \, ds \right] \right] \\
	&\leq \mathbb{E}_{s_0,x_0}\left[ \int_{s_0}^T \int_{\mathbb{A}} \pi_0(s, X_s^{\pi_0}, a) \left|\ln \pi_0(s, X_s^{\pi_0}, a)\right| da \, ds \right] + C(T)\left(1 + \|Z^*\|^{(0)}\right) < \infty,
\end{align*}
where the first term is finite by $\pi_0\in\mathcal A$ via \eqref{eq:entropy_integrability}, and the second term is bounded by applying \eqref{admiss_0} (with $\pi^1, Z^0$ replaced by $\pi^*, Z^*$) to the inner conditional expectation. Finally, for condition~(iv), since $\pi_0\in\mathcal A$, there exists $\tilde t_0 \in (s_0, T]$ such that \eqref{eq:local_unif_bound} holds for $\pi_0$. Setting $\tilde s := \tilde t_0 \wedge \frac{s_0+t+\epsilon}{2} \in (s_0, T]$, since $s_0 < \tilde s < t+\epsilon$, $\pi^*_{t,\epsilon,\pi_0}$ and $X^{\pi^*_{t,\epsilon,\pi_0}}$ coincide with $\pi_0$ and $X^{\pi_0}$ on $[s_0, \tilde s]$. Hence, \eqref{eq:local_unif_bound} holds for $\pi^*_{t,\epsilon,\pi_0}$ (with $\tilde t = \tilde s$). Thus, $\pi^*_{t,\epsilon,\pi_0}\in\mathcal A$.

Passing to the limit in \eqref{eq:recursive} with the help of \eqref{convergence in Banach}, we deduce that $(V^{*,1},V^{*,2})$ solves \eqref{eq:EEHJB}. The Feynman--Kac formula then yields $V^{*,1}=V^{\pi^*,1}$ and $V^{*,2}=V^{\pi^*,2}$. Together with the established regularity and the admissibility of $\pi^*$ and $\pi^*_{t,\epsilon,\pi_0}$, this identification rigorously justifies the formal derivation in Section~\ref{subsect:eehjb} (see \eqref{I}--\eqref{Z} particularly) as a standard verification argument, which shows that $\pi^*$ satisfies the equilibrium condition \eqref{equilibrium_con} for every $\pi_0\in\mathcal A$.}}

{\bf Step 4: Show that $(V^{*,1}, V^{*,2})$ is the unique solution in $\Theta_{[0, T]}^{(2)} \times C^{2}([0, T] \times \mathbb{R}^d)$ to the EEHJB equation \eqref{eq:EEHJB}.} 
It remains to establish uniqueness. Let $(V^{1,1}, V^{1,2})$ and $(V^{2,1}, V^{2,2})$ be two solutions to \eqref{eq:EEHJB} in $\Theta_{[0, T]}^{(2)} \times C^{2}([0, T] \times \mathbb{R}^d)$. {\color{black}In this step, $C(T)$ may also depend on the norms of  two solutions and the bounds in Assumption~\ref{assump:regularity}~(iii) on $|z|\leq\max_{i=1,2}\|V^{i,2}\|^{(0)}$.} For $i=1,2$, let $f^{i,1}$ and $f^{i,2}$ denote the functions in the representations \eqref{eq:rep1} and \eqref{eq:rep2} associated with the $i$-th solution. We define $\Delta V^1 := V^{1,1} - V^{2,1}$, $\Delta V^2 := V^{1,2} - V^{2,2}$, and $\Delta f^1 := f^{1,1} - f^{2,1}$. Adopting the temporal partition strategy from Step 1 above, we define
	\begin{align*}
		\mathcal{E}_{1, i}^{1} &:= [\Delta V_x^1]^{(0)}_{[t_{i-1}, t_i]}, \quad
		\mathcal{E}_{1, i}^{2} := \|\Delta V_x^2\|^{(0)}_{[t_{i-1}, t_i]\times\mathbb{R}^d}, \quad
		\mathcal{E}_{0, i}^{2} := \|\Delta V^2\|^{(0)}_{[t_{i-1}, t_i]\times\mathbb{R}^d}.
	\end{align*}
	Proceeding analogously to the derivation of \eqref{eq:local_recursion_grad}, we obtain
	\begin{align*}
		\mathcal{E}_{1, i}^{1} + \mathcal{E}_{1, i}^{2} + \mathcal{E}_{0, i}^{2} &\leq {\color{black}C(T)} e^{C\epsilon} \left( \mathcal{E}_{1, i+1}^{1} + \mathcal{E}_{1, i+1}^{2} + \mathcal{E}_{0, i+1}^{2} \right) + {\color{black}C(T)} e^{C\epsilon} \sqrt{\epsilon} \left( \mathcal{E}_{1, i}^{1} + \mathcal{E}_{1, i}^{2} + \mathcal{E}_{0, i}^{2} \right).
	\end{align*}
	{\color{black}By choosing $0 < \epsilon \leq \min\big\{ T, \frac{1}{C}, \frac{1}{4e^2(C(T))^2} \big\}$, we have $C(T)e^{C\epsilon}\sqrt\epsilon\leq 1/2$}. Rearranging terms yields the recursive relation
	\begin{equation*}
		\mathcal{E}_{1, i}^{1} + \mathcal{E}_{1, i}^{2} + \mathcal{E}_{0, i}^{2} \leq 2{\color{black}C(T)} \left( \mathcal{E}_{1, i+1}^{1} + \mathcal{E}_{1, i+1}^{2} + \mathcal{E}_{0, i+1}^{2} \right).
	\end{equation*}
	Noting that the terminal conditions are identical, we have $\mathcal{E}_{1, K+1}^{1} + \mathcal{E}_{1, K+1}^{2} + \mathcal{E}_{0, K+1}^{2} = 0$. Backward induction from $i=K$ down to $i=1$ then yields
	\begin{equation*}
		[V_x^{1,1} - V_x^{2,1}]^{(0)}_{[0,T]} + \|V_x^{1,2} - V_x^{2,2}\|^{(0)} + \|V^{1,2} - V^{2,2}\|^{(0)} = 0.
	\end{equation*}
	This implies that $V_x^{1,1} = V_x^{2,1}$, $V^{1,2} = V^{2,2}$, and $V_x^{1,2} = V_x^{2,2}$, which guarantee that $Z^1 = Z^2$. As $f^{i,1}$ depends solely on $V_x^{i,1}$ and $Z^i$ for $i=1,2$, it follows immediately that $\Delta f^1 = 0$. Finally, by the Feynman--Kac's representation \eqref{eq:rep1}, it holds that
	\begin{equation*}
		|V^{1,1}(\tau, t, y, x) - V^{2,1}(\tau, t, y, x)| \leq \mathbb{E} \left[ \int_t^T |\Delta f^1(\tau, s, y, X^{t,x}_s)| ds \right] = 0.
	\end{equation*}
	Thus, $V^{1,1} = V^{2,1}$. We thus obtain $(V^{1,1}, V^{1,2})=(V^{2,1}, V^{2,2})$, showing the desired uniqueness.



\subsection{Discussions}
Several aspects of Theorem~\ref{thm:convergence} are worth detailed discussions. First, the probabilistic representation formulas in its proof, while inspired by \cite{ma2025convergence}, are used in a different way due to time inconsistency. 

\begin{remark}[Comparison with \cite{ma2025convergence}]\label{remark:comparison_Ma}
In the time-consistent case of \cite{ma2025convergence}, the PIA has a well-defined target, i.e., the optimal value function $V^*$, whose regularity and PDE characterization are readily available. Hence, \cite{ma2025convergence} directly estimates the error between the $n^{th}$ iterate $V^n$ and $V^*$ in a known function space, using the representation formulas. Under time inconsistency, the target $V^{\hat\pi}$ of our PIA is unknown---it is supposed to be of the form \eqref{V^pi}, with $(V^{\hat\pi,1},V^{\hat\pi,2})$ therein satisfying the EEHJB equation \eqref{eq:EEHJB}, but whether \eqref{eq:EEHJB} actually admits a solution is unknown from the literature. In response to this, we estimate the increments $[V^{n+1,1} - V^{n,1}]^{(2)}_{[0,T]}$ and $\|V^{n+1,2} - V^{n,2}\|^{(2)}$ along our PIA, in a suitable function space identified through Proposition~\ref{prop:1}, using the representation formulas.

{\color{black}In addition to these differences in the convergence analysis, new technical issues arise in verifying the well-posedness of the PIA. In the time-consistent setting of \cite{ma2025convergence}, each iterate $V^n(t, x)$ involves only the flow variables $(t, x)$, and the well-posedness of the PIA follows directly from standard PDE results; see Proposition~2.3(ii) therein. In our time-inconsistent setting, however, to show the admissibility of $\pi^{n+1}$ in \eqref{updated} and ensure that the linear PDEs in \eqref{eq:recursive} preserve the required regularity across iterations,  we need to tackle the regularity of $Z^n(t, x)$ that involves the diagonal term $V_x^{n,1}(t, t, x, x)$. We are thus called to establish the regularity of $V^{n,1}(\tau, t, y, x)$ with respect to the reference variables $(\tau, y)$. That is, the well-posedness of the linear PDEs in \eqref{eq:recursive} for each fixed pair of reference variables $(\tau, y)$ is not sufficient to ensure that the PIA is well-defined. These additional regularity requirements are addressed in Proposition~\ref{prop:1}~(i).}

{\color{black}Furthermore, the nonlinear expectation term introduces the second auxiliary function $V^{n,2}$, which must be iterated alongside $V^{n,1}$ via \eqref{eq:recursive}. Although the two linear PDEs in \eqref{eq:recursive} are decoupled in each Policy Evaluation step, the estimates for their solutions remain coupled across iterations, as $Z^n(t, x)$ involves both $V_x^{n,1}(t, t, x, x)$ and $G_z(t, x, V^{n,2}(t, x)) \cdot V_x^{n,2}(t, x)$. Establishing the uniform bounds in Proposition~\ref{prop:1}~(ii) and the increment estimates in Theorem~\ref{thm:convergence} therefore requires careful treatment of this nonlinear coupling together with the diagonal dependence.}
\end{remark}


Second, as a byproduct, the convergence of our PIA gives a constructive proof of the existence and uniqueness of classical solutions to the EEHJB equation \eqref{eq:EEHJB}. Well-posedness of this type of PDE systems is in general elusive, due to the involved non-localness. Our PIA, as it turns out, circumvents precisely the non-localness. 

\begin{remark}\label{rem:PIA non-local}
An {\it equilibrium HJB} equation (see e.g., \cite{Yong2012, lei_nonlocal_2023, lei_nonlocality_2024, liang2025}) is known to be non-local, which complicates the study of its solutions. Well-posedness is established when the equation only allows dependence on initial time. Our equilibrium HJB equation (i.e., \eqref{eq:EEHJB}) allows dependence on initial time and state as well as additional nonlinearity. This is supposed to pose significant analytical complexity, but it is circumvented by our PIA: while \eqref{eq:EEHJB} involves the non-local term $Z(t,x)$, 
in the Policy Evaluation step of our PIA (Section~\ref{subsect:pia}), $Z(t,x)$ is fixed as $Z^n(t,x)$, obtained from the previous iteration step. Remarkably, this reduces the non-local system to a family of \textit{decoupled linear parabolic PDEs}. 
That is, the PIA successfully eliminates the non-localness in every iteration. This points to a new PIA-based method for solving a general equilibrium HJB equation, which is worth a further comprehensive investigation. 
\end{remark}

Third, Theorem~\ref{thm:convergence} suggests a new way of solving time-inconsistent stochastic control problems that {\it takes away any case-by-case guessing}.

\begin{remark}\label{rem:fixed-point approach}
A mainstream approach to time-inconsistent stochastic control problems is to derive an extended HJB equation (see Remark~\ref{rem:bjork}) and solve it case by case through a cleverly guessed ansatz; see \cite{Bjork2017,BjorkKhapkoMurgoci2021,BMZ14, EP08,EMP12,dong2014time,HS26}, among many others. By contrast, under Theorem~\ref{thm:convergence}, we just arbitrarily pick the initial functions, and the PIA takes over the rest to produce an equilibrium policy---with no need of any case-by-case guessing. This is reminiscent of the fixed-point approach for time-inconsistent stopping problems in \cite{HN18,HNZ20,HY21,HZ20,HZ22,DYZ25}, where recursive game-theoretic reasoning (framed as fixed-point iterations) yields intra-personal equilibrium stopping rules. Economically, our PIA performs the same recursive intra-personal game-theoretic reasoning, but for entropy-regularized stochastic control problems. This hints at a unified fixed-point iterative approach for general time-inconsistent problems (of both stopping and control) and we will leave it for future research. 
\end{remark}

Finally, let us point out that our focus on the EEHJB equation \eqref{eq:EEHJB}, instead of the extended HJB equation (which consists of both \eqref{eq:Bjork_structure} and \eqref{eq:EEHJB}; recall Remark~\ref{rem:bjork}), provides several technical conveniences. 

\begin{remark}\label{remark:adv}
The extended HJB equation couples \eqref{eq:Bjork_structure} for $(J^{\hat{\pi}},V^{\hat{\pi},1},V^{\hat{\pi},2})$ with \eqref{eq:EEHJB} for $(V^{\hat{\pi},1},V^{\hat{\pi},2})$, where \eqref{eq:Bjork_structure} serves to characterize an equilibrium policy $\hat\pi$. But as $\hat\pi$ can be alternatively determined by $(V^{\hat{\pi},1},V^{\hat{\pi},2})$ through \eqref{I}, we choose to focus solely on \eqref{eq:EEHJB}. This brings a few technical advantages. 
	\begin{enumerate}
		\item \textbf{Less regularity required:} As the extended HJB equation involves \eqref{eq:Bjork_structure}, the classical verification theorem (e.g., \cite[Theorem 15.2]{BjorkKhapkoMurgoci2021}) requires $J^{\hat{\pi}}\in C^{1,2}$, which implicitly demands $V^{\hat{\pi},1}\in C^{1,2}$ w.r.t.\ the reference variables $(\tau,y)$; recall Remark~\ref{rem:bjork}. 
        Without considering $J^{\hat{\pi}}$, our verification does not require such regularity of $V^{\hat{\pi},1}$. 
		\item \textbf{Simplified convergence analysis:} If the PIA is built from the extended HJB equation, it will iterate $(J^n, V^{n,1}, V^{n,2})$ jointly, with $J^n(t, x) = V^{n,1}(t, t, x, x) + G(t, x, V^{n,2}(t, x))$.  
        In particular, the Policy Update step (see Section~\ref{subsect:pia}) will depend on $J_x^n(t, x)$, such that the convergence of policies $\{\pi^n\}_{n\ge 1}$ will hinge on the convergence of the derivatives $\{V_y^{n,1}\}_{n\ge 1}$. Our PIA, which involves only $(V^{n,1}, V^{n,2})$, does not require such derivative convergence. 
	\end{enumerate}
\end{remark}

{\color{black}We conclude this subsection by discussing the main analytical obstacle to extending our probabilistic approach to the controlled diffusion case.

\begin{remark}\label{rem:controlled_diffusion}
When the diffusion coefficient $\sigma(t,x,a)$ also depends on the control $a$, the policy update generally involves the second-order spatial derivatives $V_{xx}^{n,1}(t,t,x,x)$ and $V_{xx}^{n,2}(t,x)$.  Consequently, the diffusion coefficient associated with the updated policy $\pi^{n+1}$ in \eqref{eq:recursive}, given by
\[
\sigma^{n+1}(t,x)
:=
\left(
\int_{\mathbb A}
\sigma\sigma^\top(t,x,a)\,
\pi^{n+1}(t,x,a)\,da
\right)^{1/2},
\]
changes across iterations and depends on these second-order derivatives.

To retain a common auxiliary state process in the probabilistic representations \eqref{eq:rep1}--\eqref{rep_second2}, we may fix a smooth, uniformly non-degenerate coefficient $\sigma_0(t,x)$, independent of $n$, and use the auxiliary process driven by $\sigma_0$. The second-order term in the equation for $V^{n+1,1}$ can then be decomposed as
\[
\begin{aligned}
&\frac{1}{2}\text{tr}\!\left(
\sigma^{n+1}(\sigma^{n+1})^\top V_{xx}^{n+1,1}
\right)=
\frac{1}{2}\text{tr}\!\left(
\sigma_0\sigma_0^\top V_{xx}^{n+1,1}
\right)
+
\frac{1}{2}\text{tr}\!\left(
\left[\sigma^{n+1}(\sigma^{n+1})^\top
-\sigma_0\sigma_0^\top\right]
V_{xx}^{n+1,1}
\right).
\end{aligned}
\]
The first term is absorbed into the generator of the chosen auxiliary state process, while the second enters $f^{n+1,1}$ in \eqref{eq:f1} and appears under the time integral in \eqref{eq:rep1}. Consequently, estimating first-order derivatives via \eqref{rep_first1} already involves second-order derivatives of the iterates. Moreover, estimating second-order derivatives via \eqref{rep_second1} requires differentiating this additional term in $x$, which generally introduces third-order derivatives. The same issue arises for $V^{n+1,2}$ in \eqref{rep_first2}--\eqref{rep_second2}. This leads to a loss of derivatives, preventing the recursive estimates in Proposition~\ref{prop:1}~(ii) and Theorem~\ref{thm:convergence} from closing at the second-order level.

Alternatively, we may use a different auxiliary state process at each iteration:
\[
X_s^{n+1,t,x}
=
x+\int_t^s
\sigma^{n+1}\!\left(l,X_l^{n+1,t,x}\right)dW_l,
\qquad s\in[t,T].
\]
This choice avoids the additional second-order term above, but $\nabla X_s^{n+1,t,x}$, $\nabla^2 X_s^{n+1,t,x}$, $N_s^{n+1,t,x}$, and $\nabla N_s^{n+1,t,x}$ in \eqref{rep_first1}--\eqref{rep_second2} now depend on $n$. To obtain the estimates \eqref{estimate_1} with constants independent of $n$ in Proposition~\ref{prop:1}~(ii) and to compare the auxiliary processes across consecutive steps in Theorem~\ref{thm:convergence}, we would need control of $\sigma_x^{n+1}$ and $\sigma_{xx}^{n+1}$. Since $\sigma^{n+1}$ depends on $\pi^{n+1}$, which itself involves the second-order derivatives of $(V^{n,1}, V^{n,2})$, differentiating $\sigma^{n+1}$ in $x$ generally introduces higher-order derivatives of the preceding iterates. Thus, a similar loss of derivatives occurs: bounding the second-order derivatives at step $n+1$ requires higher-order derivatives from step $n$.

In either case, the probabilistic approach based on the Bismut--Elworthy--Li formula encounters a loss of derivatives and cannot close the recursive estimates at the second-order level. In fact, even in the time-consistent setting, treating controlled diffusion requires more restrictive conditions: for instance, \cite{tran2025policy} imposes a smallness condition on the control dependence of $\sigma\sigma^\top$, while \cite{ma2025convergence} considers a one-dimensional infinite-horizon problem, where the scalar structure and the absence of a time derivative allow the second-order derivative of the current iterate to be expressed in terms of its value and first-order derivative. Thus, extending our convergence analysis to the controlled diffusion case would require a different treatment and remains a question for future research.
\end{remark}}

\subsection{\texorpdfstring{{\color{black}Extension to the Infinite-Horizon Case}}{An Infinite-Horizon Extension}}\label{subsect:infinite_horizon}
{\color{black}
We next extend our convergence analysis to the infinite-horizon setting ($T = \infty$), where $F\equiv 0$ and $G\equiv 0$, so that the objective function \eqref{J^pi} takes the form
\begin{equation}\label{eq:ih_objective}
J^\pi(t,x)
:=
\mathbb E_{t,x}\left[
\int_t^\infty
\delta(s-t)
\left(
\tilde r(x,s,X_s^\pi,\pi(s,X_s^\pi))
+\lambda\mathcal H(\pi(s,X_s^\pi))
\right)ds
\right],
\quad (t,x)\in[0,\infty)\times\mathbb R^d.
\end{equation}
We still use $\mathcal A$ to denote the set of admissible policies on the infinite horizon: conditions~(i), (ii), and (iv) of Definition~\ref{def:adm} are imposed on $[0,\infty)$, and condition~(iii) is replaced, for every $(t,x)\in[0,\infty)\times\mathbb R^d$, by
\begin{equation}\label{eq:ih_adm_integrability}
\mathbb E_{t,x}\left[
\int_t^\infty \delta(s-t)
\int_{\mathbb A}
\pi(s,X_s^\pi,a)
\Bigl(
|r(x,s,X_s^\pi,a)|
+\lambda|\ln\pi(s,X_s^\pi,a)|
\Bigr)\,da\,ds
\right]<\infty.
\end{equation}
An equilibrium policy $\hat\pi\in\mathcal A$ is defined as in Definition~\ref{def:equilibrium}, with the objective function given by \eqref{eq:ih_objective}.

Analogously to \eqref{V^1, V^2}, for each $\pi\in\mathcal A$, we define the auxiliary function on $\Delta[0,\infty)\times\mathbb R^d\times\mathbb R^d$, where $\Delta[0,\infty) := \{ (\tau, t) : 0 \leq \tau \leq t < \infty \}$:
\begin{equation*}\label{eq:ih_auxiliary}
V^{\pi,1}(\tau,t,y,x)
:=
\mathbb E_{t,x}\left[
\int_t^\infty
\delta(s-\tau)
\left(
\tilde r(y,s,X_s^\pi,\pi(s,X_s^\pi))
+\lambda\mathcal H(\pi(s,X_s^\pi))
\right)ds
\right].
\end{equation*}
To analyze the PIA in the infinite-horizon setting, we proceed via finite-horizon truncations. For each $T>0$ and $\pi\in\mathcal A$, the truncated auxiliary function on $\Delta[0,T]\times\mathbb R^d\times\mathbb R^d$ is given by
\[
V^{\pi,1;T}(\tau,t,y,x)
:=
\mathbb E_{t,x}\left[
\int_t^T
\delta(s-\tau)
\left(
\tilde r(y,s,X_s^\pi,\pi(s,X_s^\pi))
+\lambda\mathcal H(\pi(s,X_s^\pi))
\right)ds
\right].
\]
Given an initial function $V^{0,1}$ on $\Delta[0,\infty)\times\mathbb R^d\times\mathbb R^d$, we set $V^{0,1;T}:=V^{0,1}|_{\Delta[0,T]\times\mathbb R^d\times\mathbb R^d}$ and apply the PIA from Section~\ref{subsect:pia} on $[0,T]$ with $F\equiv 0$ and $G\equiv 0$. In this case, only the first auxiliary function is needed:
\begin{equation}\label{eq:ih_finite_pia}
\pi^{n+1,T}(t,x)(a)
:=
\Gamma\big(
t,x,V_x^{n,1;T}(t,t,x,x),a
\big),
\qquad
V^{n+1,1;T}
:=V^{\pi^{n+1,T},1;T},
\qquad n\geq0,
\end{equation}
where $\Gamma$ is defined in \eqref{eq:gibbs_policy}. In particular, $V^{n+1,1;T}(\tau,T,y,x)=0$.

We impose the following conditions for the infinite-horizon problem.

\begin{assumption}\label{assump:ih}
The following conditions hold:
\begin{enumerate}
\item[(i)]
Assumption~\ref{assump:regularity}~(i)--(ii) holds on $[0,\infty)\times\mathbb R^d$ with the same constant $C_0>0$, in the sense that each function $\phi$ in Assumption~\ref{assump:regularity}~(i) satisfies $\sup_{S>0}\|\phi\|_{[0,S]\times\mathbb R^d}^{(2)}\leq C_0$, and $\sigma\sigma^\top(t,x)\geq \frac{1}{C_0}I_d$ on $[0,\infty)\times\mathbb R^d$.

\item[(ii)]
The discount function $\delta$ satisfies $\delta(0)=1$, $\delta|_{[0,T]}\in C^2([0,T])$ for every $T>0$, and there exist constants $C_\delta>0$ and $\rho>0$ such that
\begin{equation}\label{eq:ih_discount}
0<\delta(s)\leq C_\delta e^{-\rho s},
\qquad \forall s\geq0,
\end{equation}
where $C_\delta$ is independent of $\rho$.

\item[(iii)]
For every $T>0$, $V^{0,1;T}\in\Xi_{[0,T]}^{(2+\alpha)}\cap\Theta_{[0,T]}^{(2)}$ with $V_{x\tau}^{0,1;T}=V_{\tau x}^{0,1;T}$ and $V_{xy}^{0,1;T}=V_{yx}^{0,1;T}$, for some $\alpha\in(0,1)$. Moreover, there exists a constant $K_0\geq1$, independent of $\rho$, such that
\begin{equation}\label{eq:ih_initial}
\sup_{t\geq0,\;x\in\mathbb R^d}
\Bigl\{
|V_x^{0,1}(t,t,x,x)|
+
\bigl|
V_{xx}^{0,1}(t,t,x,x)
+
V_{xy}^{0,1}(t,t,x,x)
\bigr|
\Bigr\}
\leq K_0.
\end{equation}
\end{enumerate}
\end{assumption}

Assumption~\ref{assump:ih} naturally extends the hypotheses of Theorem~\ref{thm:convergence} to the infinite-horizon setting.
The condition \eqref{eq:ih_discount} is satisfied, for example, by mixtures of exponential discounts of the form
$\delta(s)=\int_{[\rho,\infty)}e^{-rs}\,\mu(dr)$,
where $\rho>0$ and $\mu$ is a probability measure on $[\rho,\infty)$ with finite second moment. The moment condition ensures that \(\delta\) is twice continuously differentiable, while the lower bound on the discount rates gives \(\delta(s)\leq e^{-\rho s}\) for all \(s\geq0\), with \(C_\delta=1\).
Under sufficiently strong discounting, the convergence of the infinite-horizon PIA is established in the following result.

\begin{prop}\label{prop:ih}
Assume Assumption~\ref{assump:ih}. There exists $\rho_0>0$, depending only on $d,C_0,\lambda,|\mathbb A|,C_\delta$, and $K_0$, such that for all $\rho\geq\rho_0$, the following assertions hold:
\begin{enumerate}
\item[(i)] For each $n\geq1$ and every finite interval $[0,S]$, as $T\to\infty$, $V^{n,1;T}$ converges in $\Theta_{[0,S]}^{(2)}$ to a limit $V^{n,1}$, and $\pi^{n,T}$ converges uniformly on $[0,S]\times\mathbb R^d\times\mathbb A$ to a limit $\pi^n$. Moreover, for each $n\geq0$, $\pi^{n+1}\in\mathcal A$, and $\{(V^{n,1},\pi^n)\}_{n\geq 1}$ satisfies the infinite-horizon PIA:
\begin{align}
\pi^{n+1}(t,x)(a)
&=
\Gamma\big(t,x,V_x^{n,1}(t,t,x,x),a\big),
\label{eq:ih_update}\\
V^{n+1,1}(\tau,t,y,x)
&=
\mathbb E_{t,x}\left[
\int_t^\infty
\delta(s-\tau)
\Bigl(
\tilde r(y,s,X_s^{n+1},\pi^{n+1}(s,X_s^{n+1}))
+\lambda\mathcal H(\pi^{n+1}(s,X_s^{n+1}))
\Bigr)\,ds
\right],
\label{eq:ih_evaluation}
\end{align}
where $\{X_s^{n+1}\}_{s\geq t}$ solves \eqref{sde_1} on $[t,\infty)$ under $\pi^{n+1}$.
\item[(ii)] There exist $V^{*,1}$ with $[V^{*,1}]_{[0,\infty)}^{(2)}<\infty$ and an equilibrium policy $\pi^*\in\mathcal A$, given by $\pi^*(t,x)(a)=\Gamma\big(t,x,V_x^{*,1}(t,t,x,x),a\big)$ for all $(t,x,a)\in[0,\infty)\times\mathbb R^d\times\mathbb A$, such that
\begin{equation}\label{eq:ih_convergence}
[V^{n,1}-V^{*,1}]_{[0,\infty)}^{(2)}
+
\|\pi^n-\pi^*\|^{(0)}
\leq
C\left(\frac34\right)^n,
\qquad \forall n\geq1,
\end{equation}
where $C>0$ is independent of $n$, $[\psi]_{[0,\infty)}^{(2)} := \sup_{S>0}[\psi]_{[0,S]}^{(2)}$, and $\|\pi\|^{(0)}$ denotes the supremum norm of $\pi$ on $[0,\infty)\times\mathbb R^d\times\mathbb A$.
\end{enumerate}
\end{prop}

\begin{remark}\label{rem:ih_horizon}
In the finite-horizon setting of Theorem~\ref{thm:convergence}, no decay on $\delta$ is required because the estimates rely on backward induction over a fine time partition of $[0, T]$; however, the constant $C(T)$ in \eqref{exp convergence V and pi} may grow with $T$. Consequently, an infinite-horizon convergence result cannot be deduced simply by sending $T\to\infty$ in Theorem~\ref{thm:convergence}. Under sufficiently strong discounting ($\rho \geq \rho_0$), Proposition~\ref{prop:ih} overcomes this difficulty by establishing estimates on $[0, T]$ that are uniform in the truncation horizon, from which the convergence of the infinite-horizon PIA is obtained as $T\to\infty$.
\end{remark}

\begin{proof}
Throughout this proof, let $C\geq1$ denote a generic constant that is independent of $n,T$, and $\rho$, which may change from line to line; it depends only on $d,C_0,\lambda,|\mathbb A|,C_\delta$, and $K_0$.

For each $T>0$ and $n\geq0$, define
\begin{align*}
Z^{n,T}(t,x)&:=V_x^{n,1;T}(t,t,x,x), &
\Gamma^{n,T}(t,x,a)&:=\Gamma\big(t,x,Z^{n,T}(t,x),a\big),\\
H^{n,T}(t,x)&:=H\big(t,x,Z^{n,T}(t,x)\big), &
H_z^{n,T}(t,x)&:=H_z\big(t,x,Z^{n,T}(t,x)\big).
\end{align*}
By Proposition~\ref{prop:1}~(i), $V^{n+1,1;T}$ satisfies the recursive linear PDE
\begin{equation}\label{eq:ih_recursive}
\begin{cases}
V_t^{n+1,1;T}(\tau,t,y,x)
+
\frac12\text{tr}\big(\sigma\sigma^\top(t,x) V_{xx}^{n+1,1;T}(\tau,t,y,x)\big)
+
f^{n+1,1;T}(\tau,t,y,x)
=0,\\
V^{n+1,1;T}(\tau,T,y,x)=0,
\end{cases}
\end{equation}
where, corresponding to \eqref{eq:f1},
\begin{align}\label{eq:ih_f}
f^{n+1,1;T}(\tau,t,y,x)
&:=
H_z^{n,T}(t,x)\cdot
\left[
V_x^{n+1,1;T}(\tau,t,y,x)
-\delta(t-\tau)Z^{n,T}(t,x)
\right]\notag\\
&\quad+
\delta(t-\tau)
\left[
H^{n,T}(t,x)
+
\int_{\mathbb A}
\big(r(y,t,x,a)-r(x,t,x,a)\big)
\Gamma^{n,T}(t,x,a)\,da
\right].
\end{align}

{\bf Step 1: Establish uniform weighted estimates for $V^{n,1;T}$.}
Define the weight $\omega(\tau,t):=e^{\rho(t-\tau)}$ on $\Delta[0,\infty)$. For $T>0$, set $L_{0,T}:=K_0$, $Q_{0,T}:=K_0$, and for $n\geq1$, define
\[
L_{n,T}
:=
[\omega V_x^{n,1;T}]_{[0,T]}^{(0)},\qquad
Q_{n,T}
:=
[\omega V_{xx}^{n,1;T}]_{[0,T]}^{(0)},\qquad
M_{n,T}
:=
[\omega V_{xy}^{n,1;T}]_{[0,T]}^{(0)},
\]
which are finite for each $(n,T)$ by Proposition~\ref{prop:1}~(i). In view of \eqref{eq:ih_initial}, $|Z^{n,T}(s,x)|\leq L_{n,T}$ on $[0,T]\times\mathbb R^d$ for all $n\geq0$, and $|Z_x^{0,T}(s,x)|\leq Q_{0,T}$ on $[0,T]\times\mathbb R^d$.

By Lemma~\ref{lemma:H_estimates} and \eqref{eq:ih_discount}, \eqref{eq:ih_f} yields, for each $n\geq0$,
\begin{align}\label{eq:ih_f_bound}
|f^{n+1,1;T}(\tau,s,y,x)|
&\leq
C|V_x^{n+1,1;T}(\tau,s,y,x)|
+
C\delta(s-\tau)\big(1+|Z^{n,T}(s,x)|\big)\notag\\
&\leq
Ce^{-\rho(s-\tau)}\big(1+L_{n+1,T}+L_{n,T}\big).
\end{align}
Applying \eqref{rep_first1} (with $F\equiv 0$) and \eqref{estimate_1}, we obtain, for each $n\geq0$ and $\rho>C$,
\begin{align*}
e^{\rho(t-\tau)}
|V_x^{n+1,1;T}(\tau,t,y,x)|
&\leq
C\big(1+L_{n+1,T}+L_{n,T}\big)
\int_t^T
\frac{e^{-(\rho-C)(s-t)}}{\sqrt{s-t}}\,ds\\
&\leq
C\big(1+L_{n+1,T}+L_{n,T}\big)
\int_t^\infty
\frac{e^{-(\rho-C)(s-t)}}{\sqrt{s-t}}\,ds\\
&=
\frac{C\sqrt\pi}{\sqrt{\rho-C}}
\big(1+L_{n+1,T}+L_{n,T}\big).
\end{align*}
Taking the supremum over $\Delta[0,T]\times\mathbb R^d\times\mathbb R^d$ yields, for all $n\geq0$,
\[
L_{n+1,T}
\leq
\frac{C\sqrt\pi}{\sqrt{\rho-C}}(1+L_{n+1,T}+L_{n,T}).
\]
To estimate $V_{xy}^{n+1,1;T}=V_{yx}^{n+1,1;T}$, differentiating \eqref{eq:ih_f} with respect to $y$ gives
\begin{equation}\label{eq:ih_fy}
f_y^{n+1,1;T}(\tau,s,y,x)
=
H_z^{n,T}(s,x)\cdot V_{yx}^{n+1,1;T}(\tau,s,y,x)
+
\delta(s-\tau)
\int_{\mathbb A}r_y(y,s,x,a)\Gamma^{n,T}(s,x,a)\,da,
\end{equation}
which satisfies $|f_y^{n+1,1;T}(\tau,s,y,x)|\leq Ce^{-\rho(s-\tau)}(M_{n+1,T}+1)$ for all $n\geq0$. Applying \eqref{rep_first1y} (with $F\equiv 0$) and \eqref{estimate_1} yields
\[
M_{n+1,T}
\leq
\frac{C\sqrt\pi}{\sqrt{\rho-C}}(M_{n+1,T}+1),
\qquad n\geq0.
\]
Choosing $\rho\geq C+9\pi C^2$ so that $\frac{C\sqrt\pi}{\sqrt{\rho-C}}\leq \frac13$, rearranging terms gives
\[
L_{n+1,T}\leq\frac12L_{n,T}+\frac12,\qquad
M_{n+1,T}\leq\frac12,
\qquad n\geq0,
\]
and induction on $n$ with $L_{0,T}=K_0\geq1$ yields $L_{n,T}\leq 1+(\frac12)^n(K_0-1)\leq K_0$, so that $\sup_{n\geq0,\,T>0}L_{n,T}\leq K_0$ and $\sup_{n\geq1,\,T>0}M_{n,T}\leq\frac12$.

We now estimate $V_{xx}^{n,1;T}$. In \eqref{f^1_x}, \eqref{eq:ih_discount} and the bound $\sup_{n\geq0,\,T>0}L_{n,T}\leq K_0$ yield
\[
|V_x^{n+1,1;T}(\tau,s,y,x)-\delta(s-\tau)Z^{n,T}(s,x)|
\leq
e^{-\rho(s-\tau)}(L_{n+1,T}+C_\delta L_{n,T})
\leq
Ce^{-\rho(s-\tau)}.
\]
Since $|Z^{n,T}(s,x)|\leq L_{n,T}\leq K_0$, Lemma~\ref{lemma:H_estimates} gives $|H_{zx}^{n,T}(s,x)+H_{zz}^{n,T}(s,x)\cdot Z_x^{n,T}(s,x)|\leq C(1+|Z_x^{n,T}(s,x)|)$, and all remaining terms in \eqref{f^1_x} except $H_z^{n,T}(s,x)\cdot V_{xx}^{n+1,1;T}(\tau,s,y,x)$ carry the factor $\delta(s-\tau)\leq C_\delta e^{-\rho(s-\tau)}$ and are bounded by $Ce^{-\rho(s-\tau)}(1+|Z_x^{n,T}(s,x)|)$ via Lemma~\ref{lemma:H_estimates} and \eqref{eq:Gamma_x_est}. Consequently,
\[
|f_x^{n+1,1;T}(\tau,s,y,x)|
\leq
C|V_{xx}^{n+1,1;T}(\tau,s,y,x)|
+
Ce^{-\rho(s-\tau)}\big(1+|Z_x^{n,T}(s,x)|\big).
\]
Using $|V_{xx}^{n+1,1;T}(\tau,s,y,x)|\leq e^{-\rho(s-\tau)}Q_{n+1,T}$ and $|Z_x^{n,T}(s,x)|\leq M_{n,T}+Q_{n,T}\leq \frac12+Q_{n,T}$ for $n\geq1$ (with $|Z_x^{0,T}(s,x)|\leq Q_{0,T}=K_0$), this yields, for all $n\geq0$,
\[
|f_x^{n+1,1;T}(\tau,s,y,x)|
\leq
Ce^{-\rho(s-\tau)}(1+Q_{n+1,T}+Q_{n,T}).
\]
Moreover, \eqref{eq:ih_f_bound} and $\sup_{n\geq0,\,T>0}L_{n,T}\leq K_0$ yield $|f^{n+1,1;T}(\tau,s,y,x)|\leq Ce^{-\rho(s-\tau)}$. Substituting these bounds into \eqref{rep_second1} (with $F\equiv 0$) and applying \eqref{estimate_1} yields, for all $n\geq0$,
\[
Q_{n+1,T}
\leq
\frac{C\sqrt\pi}{\sqrt{\rho-C}}
(1+Q_{n+1,T}+Q_{n,T}).
\]
For $\rho\geq C+9\pi C^2$, rearranging terms gives $Q_{n+1,T}\leq\frac12 Q_{n,T}+\frac12$, so that $\sup_{n\geq1,\,T>0}Q_{n,T}\leq K_0$. In view of \eqref{eq:rep1} (with $F\equiv 0$),  $|f^{n+1,1;T}(\tau,s,y,x)|\leq Ce^{-\rho(s-\tau)}$, and \eqref{eq:ih_recursive},
\begin{align*}
e^{\rho(t-\tau)}|V^{n+1,1;T}(\tau,t,y,x)|
&\leq
C\int_t^T e^{-\rho(s-t)}\,ds
\leq
\frac C\rho,\\
e^{\rho(t-\tau)}|V_t^{n+1,1;T}(\tau,t,y,x)|
&\leq
C(Q_{n+1,T}+1)
\leq
C.
\end{align*}
Combining these with the bounds on $L_{n,T}$, $Q_{n,T}$, and $M_{n,T}$ yields
\begin{equation}\label{eq:ih_step1_bounds}
\sup_{n\geq1,\;T>0}
\Bigl(
[\omega V_x^{n,1;T}]_{[0,T]}^{(0)}
+[\omega V_{xx}^{n,1;T}]_{[0,T]}^{(0)}
+[\omega V_{xy}^{n,1;T}]_{[0,T]}^{(0)}
+[\omega V_t^{n,1;T}]_{[0,T]}^{(0)}
\Bigr)
\leq C,
\quad
\sup_{n\geq1,\;T>0}
[\omega V^{n,1;T}]_{[0,T]}^{(0)}
\leq\frac C\rho.
\end{equation}

{\bf Step 2: Proof of assertion~(i).}
Fix $T_2>T_1>0$ and set, on $\Delta[0,T_1]\times\mathbb R^d\times\mathbb R^d$,
\begin{align*}
\widehat V^{n,1}&:=V^{n,1;T_2}-V^{n,1;T_1},\qquad
\widehat f^{n+1,1}:=f^{n+1,1;T_2}-f^{n+1,1;T_1},\\
\widehat Z^n&:=Z^{n,T_2}-Z^{n,T_1},\qquad
\widehat Z_x^n:=Z_x^{n,T_2}-Z_x^{n,T_1}.
\end{align*}
In particular, $\widehat Z^n(s,x)=\widehat V_x^{n,1}(s,s,x,x)$ and $\widehat Z_x^n(s,x)=\widehat V_{xx}^{n,1}(s,s,x,x)+\widehat V_{xy}^{n,1}(s,s,x,x)$. With the weight $\omega_{T_1}(\tau,t):=e^{\rho(t-\tau)+\frac{\rho}{2}(T_1-t)}$, define
\[
A_{1,n}
:=
[\omega_{T_1}\widehat V_x^{n,1}]_{[0,T_1]}^{(0)},\qquad
A_{2,n}
:=
[\omega_{T_1}\widehat V_{xx}^{n,1}]_{[0,T_1]}^{(0)},\qquad
A_{3,n}
:=
[\omega_{T_1}\widehat V_{xy}^{n,1}]_{[0,T_1]}^{(0)}.
\]
Since $V^{0,1;T_2}=V^{0,1;T_1}$ on $\Delta[0,T_1]\times\mathbb R^d\times\mathbb R^d$, we have $A_{1,0}=A_{2,0}=A_{3,0}=0$.

We now estimate $\widehat f^{n+1,1}$, $\widehat f_y^{n+1,1}$, and $\widehat f_x^{n+1,1}$ on $\Delta[0,T_1]\times\mathbb R^d\times\mathbb R^d$. First, subtracting \eqref{eq:ih_f} between $T_2$ and $T_1$, a calculation analogous to \eqref{eq:df1_bound} and Lemma~\ref{lemma:H_estimates} yield
\begin{align}\label{eq:ih_difference_f}
|\widehat f^{n+1,1}(\tau,s,y,x)|
&\leq
|H_z^{n,T_1}(s,x)||\widehat V_x^{n+1,1}(\tau,s,y,x)|\notag\\
&\quad+
|H_z^{n,T_2}(s,x)-H_z^{n,T_1}(s,x)||V_x^{n+1,1;T_2}(\tau,s,y,x)-\delta(s-\tau)Z^{n,T_2}(s,x)|\notag\\
&\quad+
\delta(s-\tau)C|\widehat Z^n(s,x)|\notag\\
&\leq
C|\widehat V_x^{n+1,1}(\tau,s,y,x)|
+
Ce^{-\rho(s-\tau)}|\widehat Z^n(s,x)|.
\end{align}
Similarly, subtracting \eqref{eq:ih_fy} between $T_2$ and $T_1$ and estimating as in \eqref{eq:df1y_bound} give
\begin{align}\label{eq:ih_difference_fy}
|\widehat f_y^{n+1,1}(\tau,s,y,x)|
&\leq
|H_z^{n,T_1}(s,x)||\widehat V_{yx}^{n+1,1}(\tau,s,y,x)|\notag\\
&\quad+
|H_z^{n,T_2}(s,x)-H_z^{n,T_1}(s,x)||V_{yx}^{n+1,1;T_2}(\tau,s,y,x)|
+
\delta(s-\tau)C|\widehat Z^n(s,x)|\notag\\
&\leq
C|\widehat V_{yx}^{n+1,1}(\tau,s,y,x)|
+
Ce^{-\rho(s-\tau)}|\widehat Z^n(s,x)|,
\end{align}
where the second inequalities in both \eqref{eq:ih_difference_f} and \eqref{eq:ih_difference_fy} follow from \eqref{eq:ih_step1_bounds}, Lemma~\ref{lemma:H_estimates}, and $\delta(s-\tau)\leq C_\delta e^{-\rho(s-\tau)}$.

Finally, subtracting \eqref{f^1_x} between $T_2$ and $T_1$, a calculation analogous to \eqref{eq:df1x_bound} yields
\begin{align}\label{eq:ih_difference_fx}
|\widehat f_x^{n+1,1}(\tau,s,y,x)|
&\leq
|H_z^{n,T_1}(s,x)||\widehat V_{xx}^{n+1,1}(\tau,s,y,x)|
+
|(H_{zx}+H_{zz}\cdot Z_x)^{n,T_1}(s,x)||\widehat V_x^{n+1,1}(\tau,s,y,x)|\notag\\
&\quad+
Ce^{-\rho(s-\tau)}\Bigl(|\widehat Z^n(s,x)|+|\widehat Z_x^n(s,x)|\Bigr)\notag\\
&\leq
C\Bigl(
|\widehat V_{xx}^{n+1,1}(\tau,s,y,x)|
+
|\widehat V_x^{n+1,1}(\tau,s,y,x)|
\Bigr)
+
Ce^{-\rho(s-\tau)}
\Bigl(
|\widehat Z^n(s,x)|+|\widehat Z_x^n(s,x)|
\Bigr),
\end{align}
where the first inequality uses Lemma~\ref{lemma:H_estimates}, \eqref{eq:ih_step1_bounds}, and $\delta(s-\tau)\leq C_\delta e^{-\rho(s-\tau)}$, and the second follows from Lemma~\ref{lemma:H_estimates} and \eqref{eq:ih_step1_bounds}.

At $t=T_1$, since $V^{n+1,1;T_1}(\tau,T_1,y,x)=0$, \eqref{eq:ih_step1_bounds} yields
\[
|\widehat V_x^{n+1,1}(\tau,T_1,y,x)|
=
|V_x^{n+1,1;T_2}(\tau,T_1,y,x)|
\leq
Ce^{-\rho(T_1-\tau)}.
\]
Since $\omega_{T_1}(\tau,t)=e^{-\frac{\rho}{2}(s-t)}\omega_{T_1}(\tau,s)$ and $\omega_{T_1}(\tau,t)e^{-\rho(s-\tau)}=e^{-\frac{\rho}{2}(s-t)}\omega_{T_1}(s,s)$, \eqref{eq:ih_difference_f} gives
\[
\omega_{T_1}(\tau,t)|\widehat f^{n+1,1}(\tau,s,y,x)|
\leq
Ce^{-\frac{\rho}{2}(s-t)}(A_{1,n+1}+A_{1,n}).
\]
Multiplying \eqref{rep_first1} on $[t,T_1]$ by $\omega_{T_1}(\tau,t)$ and applying \eqref{estimate_1}, we obtain, for $\rho/2>C$,
\begin{align*}
\omega_{T_1}(\tau,t)|\widehat V_x^{n+1,1}(\tau,t,y,x)|
&\leq
Ce^{-(\rho/2-C)(T_1-t)}
+
C(A_{1,n+1}+A_{1,n})\int_t^{T_1}\frac{e^{-(\rho/2-C)(s-t)}}{\sqrt{s-t}}\,ds\\
&\leq
C+\frac{C\sqrt\pi}{\sqrt{\rho/2-C}}(A_{1,n+1}+A_{1,n}).
\end{align*}
Taking the supremum over $\Delta[0,T_1]\times\mathbb R^d\times\mathbb R^d$ and estimating $A_{3,n+1}$ and $A_{2,n+1}$ analogously via \eqref{rep_first1y}, \eqref{rep_second1}, \eqref{estimate_1}, \eqref{eq:ih_step1_bounds}, and \eqref{eq:ih_difference_fy}--\eqref{eq:ih_difference_fx} yields
\begin{align*}
A_{1,n+1}
&\leq C+\frac{C\sqrt\pi}{\sqrt{\rho/2-C}}(A_{1,n+1}+A_{1,n}),\qquad
A_{3,n+1}
\leq C+\frac{C\sqrt\pi}{\sqrt{\rho/2-C}}(A_{3,n+1}+A_{1,n}),\\
A_{2,n+1}
&\leq C+\frac{C\sqrt\pi}{\sqrt{\rho/2-C}}(A_{2,n+1}+A_{2,n}+A_{3,n}+A_{1,n+1}+A_{1,n}).
\end{align*}
For $\rho\geq 2C+18\pi C^2$, we have $\frac{C\sqrt\pi}{\sqrt{\rho/2-C}}\leq\frac13$, so rearranging terms gives
\[
A_{1,n+1}\leq\frac12A_{1,n}+C,\qquad
A_{3,n+1}\leq\frac12A_{1,n}+C,\qquad
A_{2,n+1}\leq\frac12\big(A_{2,n}+A_{3,n}+A_{1,n+1}+A_{1,n}\big)+C.
\]
Inducting on $n$ in the first relation yields $\sup_{n\geq1,\,T_2>T_1}A_{1,n}\leq C$, and substituting this into the remaining two relations gives $\sup_{n\geq1,\,T_2>T_1}(A_{2,n}+A_{3,n})\leq C$. Furthermore, applying \eqref{eq:rep1} on $[t,T_1]$, using $|V^{n+1,1;T_2}(\tau,T_1,y,x)|\leq Ce^{-\rho(T_1-\tau)}$ from \eqref{eq:ih_step1_bounds}, and subtracting \eqref{eq:ih_recursive} between $T_2$ and $T_1$, we obtain
\begin{align*}
\omega_{T_1}(\tau,t)|\widehat V^{n+1,1}(\tau,t,y,x)|
&\leq
Ce^{-\frac{\rho}{2}(T_1-t)}
+
C(A_{1,n+1}+A_{1,n})\int_t^{T_1}e^{-\frac{\rho}{2}(s-t)}\,ds
\leq C,\\
\omega_{T_1}(\tau,t)|\widehat V_t^{n+1,1}(\tau,t,y,x)|
&\leq
C\big(A_{2,n+1}+A_{1,n+1}+A_{1,n}\big)
\leq C.
\end{align*}
For any $S<T_1$, since $\omega_{T_1}(\tau,t)=e^{\rho(t-\tau)+\frac{\rho}{2}(T_1-t)}\geq e^{\frac{\rho}{2}(T_1-S)}$ on $\Delta[0,S]$, combining the above bounds yields
\begin{equation}\label{eq:ih_horizon_comparison}
\sup_{n\geq1}[V^{n,1;T_2}-V^{n,1;T_1}]_{[0,S]}^{(2)}
\leq Ce^{-\frac{\rho}{2}(T_1-S)}.
\end{equation}
By \eqref{eq:ih_horizon_comparison} and the completeness of $\Theta_{[0,S]}^{(2)}$, for each $n\geq1$, as $T\to\infty$, $V^{n,1;T}$ converges on every finite interval $[0,S]$ in $\Theta_{[0,S]}^{(2)}$ to a limit $V^{n,1}\in\Theta_{[0,S]}^{(2)}$. Moreover, passing to the limit in \eqref{eq:ih_step1_bounds} gives
\begin{equation}\label{eq:ih_limit_bounds}
|V^{n,1}(\tau,t,y,x)| + |V_x^{n,1}(\tau,t,y,x)| \leq Ce^{-\rho(t-\tau)},
\qquad \forall (\tau,t,y,x)\in\Delta[0,\infty)\times\mathbb R^d\times\mathbb R^d.
\end{equation}
Setting $\pi^{n+1}(t,x)(a):=\Gamma\big(t,x,V_x^{n,1}(t,t,x,x),a\big)$ as in \eqref{eq:ih_update}, Lemma~\ref{lemma:H_estimates}~(ii) and \eqref{eq:ih_finite_pia} give $\|\pi^{n+1,T}-\pi^{n+1}\|_{[0,S]\times\mathbb R^d\times\mathbb A}^{(0)}\leq C[V^{n,1;T}-V^{n,1}]_{[0,S]}^{(2)}\to 0$ for all $n\geq 0$.

To verify $\pi^{n+1}\in\mathcal A$ for all $n\geq 0$, note that $\pi^{n+1}$ is continuous in $(t,x)$ for each $a\in\mathbb A$ by Lemma~\ref{lemma:H_estimates}~(ii). Next, for $n\geq 1$, as in \eqref{b-b}--\eqref{Z-Z}, for each fixed $S>0$ and $T>S$, \eqref{eq:ih_step1_bounds} yields
\[
|V_x^{n,1;T}(t,t,x,x)-V_x^{n,1;T}(t,t,x',x')|
\leq
\big([V_{xx}^{n,1;T}]_{[0,S]}^{(0)}+[V_{xy}^{n,1;T}]_{[0,S]}^{(0)}\big)|x-x'|
\leq
C|x-x'|,
\qquad \forall t\in[0,S].
\]
Passing to the limit as $T\to\infty$ and then taking the supremum over $S>0$ shows that $x\mapsto V_x^{n,1}(t,t,x,x)$ is Lipschitz continuous uniformly in $t\in[0,\infty)$. For $n=0$, \eqref{eq:ih_initial} directly gives
\[
|V_x^{0,1}(t,t,x,x)|\leq K_0,
\qquad
|V_x^{0,1}(t,t,x,x)-V_x^{0,1}(t,t,x',x')|\leq K_0|x-x'|,
\qquad \forall t\geq0,\;x,x'\in\mathbb R^d.
\]
Consequently, for all $n\geq 0$, an estimate analogous to \eqref{b-b} yields
\[
|\tilde b(t,x,\pi^{n+1}(t,x))-\tilde b(t,x',\pi^{n+1}(t,x'))|
\leq
C\Bigl(|x-x'| + |V_x^{n,1}(t,t,x,x)-V_x^{n,1}(t,t,x',x')|\Bigr)
\leq
C|x-x'|,
\]
uniformly in $t\in[0,\infty)$, ensuring that SDE \eqref{sde_1} under $\pi^{n+1}$ admits a unique strong solution $\{X_s^{n+1}\}_{s\geq t}$ on $[t,\infty)$.

Furthermore, condition~(iii) (i.e., \eqref{eq:ih_adm_integrability}) is satisfied since
\begin{align*}\label{eq:ih_adm_int_est}
&\mathbb E_{t,x}\left[
\int_t^\infty \delta(s-t)
\int_{\mathbb A}
\pi^{n+1}(s,X_s^{n+1},a)
\Bigl(
|r(x,s,X_s^{n+1},a)|
+\lambda|\ln\pi^{n+1}(s,X_s^{n+1},a)|
\Bigr)\,da\,ds
\right]\notag\\
&\quad\leq
\mathbb E_{t,x}\left[
\int_t^\infty \delta(s-t)
\int_{\mathbb A}
\pi^{n+1}(s,X_s^{n+1},a)
C\big(1+|Z^n(s,X_s^{n+1})|\big)\,da\,ds
\right]\notag\\
&\quad\leq
C\int_t^\infty e^{-\rho(s-t)}\,ds
=
\frac C\rho
<\infty,
\end{align*}
where the first inequality uses \eqref{eq:ln_Gamma}, Assumption~\ref{assump:ih}~(i), and Lemma~\ref{lemma:H_estimates}~(i), and the second follows from $|Z^n(s,x)|\leq C$ (using \eqref{eq:ih_limit_bounds} for $n\geq 1$ and \eqref{eq:ih_initial} for $n=0$) and $\delta(s-t)\leq C_\delta e^{-\rho(s-t)}$. Finally, condition~(iv) holds by choosing $\tilde t=t+1$: estimating as in \eqref{admiss_1} using $|Z^n(s,x)|\leq C$ yields
\[
\mathbb E_{t,x}\left[
\int_{\mathbb A}\sup_{s\in[t,t+1]}
\left\{
\pi^{n+1}(s,X_s^{n+1},a)
\bigl(1+|\ln\pi^{n+1}(s,X_s^{n+1},a)|\bigr)
\right\}\,da
\right]
\leq
C
<\infty.
\]
Hence, $\pi^{n+1}\in\mathcal A$ for all $n\geq 0$.

Passing to the limit as $T\to\infty$ in \eqref{eq:ih_recursive} on each $\Delta[0,S]\times\mathbb R^d\times\mathbb R^d$ (for any $S>0$), we obtain that $V^{n+1,1}$ satisfies
\begin{align*}
&V_t^{n+1,1}(\tau,t,y,x)
+
\frac12\text{tr}\big(\sigma\sigma^\top(t,x)V_{xx}^{n+1,1}(\tau,t,y,x)\big)
+
\tilde b(t,x,\pi^{n+1}(t,x))\cdot V_x^{n+1,1}(\tau,t,y,x)\\
&\quad+\,
\delta(t-\tau)
\Bigl(
\tilde r(y,t,x,\pi^{n+1}(t,x))
+\lambda\mathcal H(\pi^{n+1}(t,x))
\Bigr)
=0,
\end{align*}
on $\Delta[0,\infty)\times\mathbb R^d\times\mathbb R^d$.
Applying It\^o's formula to $V^{n+1,1}(\tau,\cdot,y,X_\cdot^{n+1})$ on $[t,M]$, taking expectations (where the stochastic integral vanishes by the boundedness of $\sigma$ and $V_x^{n+1,1}$ on $[t,M]$), and letting $M\to\infty$ using $|V^{n+1,1}(\tau,M,y,X_M^{n+1})|\leq Ce^{-\rho(M-\tau)}\to 0$ from \eqref{eq:ih_limit_bounds}, we obtain \eqref{eq:ih_evaluation}.

{\bf Step 3: Establish the exponential convergence \eqref{eq:ih_convergence}.}
For $n\geq1$ and $T>0$, set
\[
\Delta V^{n,1;T}:=V^{n+1,1;T}-V^{n,1;T},\qquad
\Delta Z^{n,T}:=Z^{n+1,T}-Z^{n,T},\qquad
\Delta f^{n,1;T}:=f^{n+1,1;T}-f^{n,1;T},
\]
and define
\[
E_{1,n,T}
:=
[\omega\Delta V_x^{n,1;T}]_{[0,T]}^{(0)},\qquad
E_{2,n,T}
:=
[\omega\Delta V_{xx}^{n,1;T}]_{[0,T]}^{(0)},\qquad
E_{3,n,T}
:=
[\omega\Delta V_{xy}^{n,1;T}]_{[0,T]}^{(0)}.
\]
By \eqref{eq:ih_step1_bounds}, $E_{1,1,T}+E_{2,1,T}+E_{3,1,T}\leq C$. For $n\geq 2$, the same estimation procedure as in \eqref{eq:ih_difference_f}--\eqref{eq:ih_difference_fx} yields
\begin{align*}
|\Delta f^{n,1;T}(\tau,s,y,x)|
&\leq
C|\Delta V_x^{n,1;T}(\tau,s,y,x)|
+
Ce^{-\rho(s-\tau)}|\Delta Z^{n-1,T}(s,x)|,\\
|\Delta f_y^{n,1;T}(\tau,s,y,x)|
&\leq
C|\Delta V_{yx}^{n,1;T}(\tau,s,y,x)|
+
Ce^{-\rho(s-\tau)}|\Delta Z^{n-1,T}(s,x)|,\\
|\Delta f_x^{n,1;T}(\tau,s,y,x)|
&\leq
C\big(|\Delta V_{xx}^{n,1;T}|+|\Delta V_x^{n,1;T}|\big)(\tau,s,y,x)
+
Ce^{-\rho(s-\tau)}\big(|\Delta Z^{n-1,T}|+|\Delta Z_x^{n-1,T}|\big)(s,x).
\end{align*}
As $\Delta V^{n,1;T}(\tau,T,y,x)=0$, substituting these bounds into \eqref{rep_first1}, \eqref{rep_first1y}, and \eqref{rep_second1} and then applying \eqref{estimate_1} yields
\begin{align*}
E_{1,n,T}
&\leq\frac{C\sqrt\pi}{\sqrt{\rho-C}}(E_{1,n,T}+E_{1,n-1,T}),\qquad
E_{3,n,T}
\leq\frac{C\sqrt\pi}{\sqrt{\rho-C}}(E_{3,n,T}+E_{1,n-1,T}),\\
E_{2,n,T}
&\leq\frac{C\sqrt\pi}{\sqrt{\rho-C}}\big(E_{2,n,T}+E_{2,n-1,T}+E_{3,n-1,T}+E_{1,n,T}+E_{1,n-1,T}\big).
\end{align*}
For $\rho\geq C+9\pi C^2$, we have $\frac{C\sqrt\pi}{\sqrt{\rho-C}}\leq\frac13$. Rearranging the terms gives
\[
E_{1,n,T}\leq\frac12 E_{1,n-1,T},\qquad
E_{3,n,T}\leq\frac12 E_{1,n-1,T},\qquad
E_{2,n,T}\leq\frac12 E_{2,n-1,T}+\frac12\big(E_{3,n-1,T}+E_{1,n,T}+E_{1,n-1,T}\big).
\]
Inducting on $n$ in the first two inequalities yields $E_{1,n,T}+E_{3,n,T}\leq C(1/2)^n$. Substituting this into the third inequality gives
\[
E_{2,n,T}\leq\frac12 E_{2,n-1,T}+C\left(\frac34\right)^n.
\]
It follows from Lemma~\ref{lemma:seq_decay} (with $\rho=1/2$ and $q=3/4$) that $E_{2,n,T}\leq C(3/4)^n$. In light of \eqref{eq:rep1} and \eqref{eq:ih_recursive}, we deduce the same bound for $[\omega\Delta V^{n,1;T}]_{[0,T]}^{(0)}$ and $[\omega\Delta V_t^{n,1;T}]_{[0,T]}^{(0)}$. As $\omega(\tau,t)\geq 1$ on $\Delta[0,T]$, summing these components gives
\[
[\Delta V^{n,1;T}]_{[0,T]}^{(2)}
\leq C\left(\frac34\right)^n,
\qquad \forall T>0.
\]
For any fixed $S>0$ and all $T\geq S$, we have $[\Delta V^{n,1;T}]_{[0,S]}^{(2)}\leq [\Delta V^{n,1;T}]_{[0,T]}^{(2)}$. Letting $T\to\infty$ via the convergence of $V^{n,1;T}$ in $\Theta_{[0,S]}^{(2)}$ from assertion~(i) yields $[V^{n+1,1}-V^{n,1}]_{[0,S]}^{(2)}\leq C(3/4)^n$. As $C$ is independent of $S$, it holds by taking the supremum over $S>0$ that
\[
[V^{n+1,1}-V^{n,1}]_{[0,\infty)}^{(2)}
\leq C\left(\frac34\right)^n,
\qquad \forall n\geq 1.
\]
Hence, $\{V^{n,1}\}_{n\geq 1}$ is a Cauchy sequence in $\Theta_{[0,S]}^{(2)}$ for every $S>0$, and admits a unique limit $V^{*,1}$ satisfying $[V^{*,1}]_{[0,\infty)}^{(2)}<\infty$ with $[V^{n,1}-V^{*,1}]_{[0,\infty)}^{(2)}\leq C(3/4)^n$. Setting $\pi^*(t,x)(a):=\Gamma\big(t,x,V_x^{*,1}(t,t,x,x),a\big)$, by Lemma~\ref{lemma:H_estimates}~(ii), we get
\[
\|\pi^n-\pi^*\|^{(0)}
\leq C\sup_{t\geq0,\,x\in\mathbb R^d}|V_x^{n-1,1}(t,t,x,x)-V_x^{*,1}(t,t,x,x)|
\leq C\left(\frac34\right)^n,
\]
which establishes \eqref{eq:ih_convergence}.

{\bf Step 4: Show that $\pi^*$ is an equilibrium policy.}
Passing to the limit as $n\to\infty$ in \eqref{eq:ih_limit_bounds} and in the Lipschitz estimate for $V_x^{n,1}$ from Step~2, we obtain $|V_x^{*,1}(t,t,x,x)|\leq C$ and $|V_x^{*,1}(t,t,x,x)-V_x^{*,1}(t,t,x',x')|\leq C|x-x'|$ on $[0,\infty)\times\mathbb R^d$. Similar to Step~2, $x\mapsto\tilde b(t,x,\pi^*(t,x))$ is Lipschitz continuous uniformly in $t\in[0,\infty)$, ensuring condition~(ii) (a unique strong solution $\{X_s^{\pi^*}\}_{s\geq t}$ to SDE \eqref{sde_1} on $[t,\infty)$), and conditions~(iii) and (iv) hold as in Step~2. Hence, $\pi^*\in\mathcal A$. Moreover, for any $(t,x)\in[0,\infty)\times\mathbb R^d$, $\pi_0\in\mathcal A$, and $\epsilon>0$, the concatenation argument from Step~3 in the proof of Theorem~\ref{thm:convergence} with the fact $\delta(s-t)\leq C_\delta e^{-\rho(s-t)}$ verifies $\pi^*_{t,\epsilon,\pi_0}\in\mathcal A$.

Passing to the limit as $n\to\infty$ in the linear PDE of $V^{n+1,1}$, we get that $V^{*,1}$ satisfies the infinite-horizon EEHJB equation
\begin{align*}
&V_t^{*,1}(\tau,t,y,x)
+
\frac12\text{tr}\big(\sigma\sigma^\top(t,x)V_{xx}^{*,1}(\tau,t,y,x)\big)
+
\tilde b(t,x,\pi^*(t,x))\cdot V_x^{*,1}(\tau,t,y,x)\\
&\quad+\,
\delta(t-\tau)
\Bigl(
\tilde r(y,t,x,\pi^*(t,x))
+\lambda\mathcal H(\pi^*(t,x))
\Bigr)
=0,
\end{align*}
on $\Delta[0,\infty)\times\mathbb R^d\times\mathbb R^d$, with $|V^{*,1}(\tau,t,y,x)|+|V_x^{*,1}(\tau,t,y,x)|\leq Ce^{-\rho(t-\tau)}$ by \eqref{eq:ih_limit_bounds}. Applying It\^o's formula to $V^{*,1}(\tau,\cdot,y,X_\cdot^{\pi^*})$ on $[t,M]$ and letting $M\to\infty$ as in Step~2 yields $V^{*,1}=V^{\pi^*,1}$. Similar to Step~3 of the proof of Theorem~\ref{thm:convergence}, this identification, together with the admissibility of $\pi^*$ and $\pi^*_{t,\epsilon,\pi_0}$, completes the verification argument of Section~\ref{subsect:eehjb} on $[t,t+\epsilon]$, concluding that $\pi^*$ satisfies \eqref{equilibrium_con} for every $\pi_0\in\mathcal A$.
\end{proof}
}

{\color{black}
\section{Numerical Examples}\label{sec:numerical}
\newcommand{\numfigure}[2]{%
    \IfFileExists{#2}{\includegraphics[width=#1]{#2}}%
    {\fbox{\parbox[c][1.25in][c]{0.84\textwidth}{\centering Figure to be inserted}}}}
This section illustrates the numerical performance of the policy iteration algorithm via two financial applications under non-exponential discounting.

\subsection{Equilibrium Effort under a Given Incentive Contract}\label{subsect:num_effort}
Let us consider an investor who hires a manager to operate a project, with the manager's compensation determined by the project's terminal performance. The manager must decide how much costly effort to exert today to improve the prospect of a higher payment at maturity. A closely related problem appears in Section~2.1.1 of \cite{hernandezpossamai2024contract}, where a risk-neutral manager chooses effort for a given project and terminal payment contract under non-exponential discounting. Here we take the project and the compensation schedule as given and compute the manager's equilibrium effort policy. Specifically, the manager chooses an effort level $a\in\mathbb A:=[0,\bar a]$ according to a relaxed policy $\pi$, where \(\bar a>0\) reflects the manager’s capacity constraint. Note that such a constraint is common in continuous-time principal-agent problems; see, e.g., \cite{sannikov2008continuous, possamai2025golden}, where $\mathbb A$ is a compact subset of $[0,\infty)$ containing $0$. The cumulative project performance $X^\pi$ evolves as
\begin{equation*}
    dX_s^\pi=\left(\mu+\int_{\mathbb A}a\,\pi(s,X_s^\pi)(a)\,da\right)ds+\sigma\,dW_s,
    \qquad X_t^\pi=x\in\mathbb R,
\end{equation*}
where $\mu$ is the baseline drift and $\sigma>0$ is the volatility. Effort increases expected performance but incurs a running cost $\frac{\kappa}{2} a^2$ with $\kappa>0$. At maturity $T$, the contract pays $B(X_T^\pi)$, where
\begin{equation*}
    B(x):=B_{\min}+\frac{B_{\max}-B_{\min}}{1+e^{-k(x-K)}},
\end{equation*}
with $0<B_{\min}<B_{\max}$, sensitivity parameter $k>0$, and performance target $K\in\mathbb R$. This contract sets a floor $B_{\min} > 0$ and a cap $B_{\max}$ on the terminal payment. Since $B$ is strictly increasing, it incentivizes the manager to exert effort; its marginal payment $B'(x)$ attains its maximum at $K$ and vanishes as $x \to \pm\infty$.

The manager evaluates terminal compensation through a strictly increasing, three times continuously differentiable utility function $U:(0,\infty)\to\mathbb R$. To induce time inconsistency, we consider a weighted exponential discount function
\begin{equation}\label{eq:num_discount}
\delta(s)=w_1 e^{-\beta_1 s}+w_2 e^{-\beta_2 s},
\end{equation}
where $w_1,w_2>0$ with $w_1+w_2=1$, and $\beta_1,\beta_2>0$ are distinct discount rates. The entropy-regularized objective function is given by
\begin{equation*}
    J^\pi(t,x)=\mathbb E_{t,x}\left[
        \int_t^T\delta(s-t)\left(
            -\frac{\kappa}{2}\int_{\mathbb A}a^2\,\pi(s,X_s^\pi)(a)\,da
            +\lambda\mathcal H\bigl(\pi(s,X_s^\pi)\bigr)
        \right)ds
        +\delta(T-t)U(B(X_T^\pi))
    \right].
\end{equation*}

In the framework of Section~\ref{sec:formulation}, this problem corresponds to $b(t,x,a)=\mu+a$, $r(a)=-\frac{\kappa}{2} a^2$, $F(\tau,y,x)=\delta(T-\tau)U(B(x))$, and $G\equiv h\equiv0$. Because $\overline{B(\mathbb R)}=[B_{\min},B_{\max}]\subset(0,\infty)$, $B$ has bounded derivatives of all orders, and $U$ is three times continuously differentiable on $(0,\infty)$, the composite function $U \circ B$ has bounded derivatives up to order three, which ensures $F \in \Xi_{[0,T]}^{(2+\alpha_0)}$ for any $\alpha_0 \in (0,1)$. The remaining conditions in Assumption~\ref{assump:regularity} are satisfied: $0<|\mathbb A|=\bar a<\infty$, the drift and running reward are bounded and independent of $(t,x,y)$, the volatility $\sigma>0$ is constant, and $\delta$ in \eqref{eq:num_discount} belongs to $C^2([0,T])$. Theorem~\ref{thm:convergence} ensures that the PIA converges to the unique solution of the EEHJB equation \eqref{eq:EEHJB} and an equilibrium policy $\pi^*$.

Since $G\equiv h\equiv0$ and both $r$ and $F$ are independent of $y$, we only need to consider the auxiliary function $V^{n,1}(\tau,t,x)$ (and $V^{*,1}(\tau,t,x)$). Given an initial function $V^{0,1}(\tau,t,x)$, at each iteration $n\ge0$, we set $Z^n(t,x):=V_x^{n,1}(t,t,x)$. The policy update step \eqref{updated} then computes $\pi^{n+1}$ via the Gibbs measure
\begin{equation}\label{eq:num_effort_update}
    \pi^{n+1}(t,x)(a) = \Gamma\bigl(t,x,Z^n(t,x),a\bigr) =
    \frac{\exp\left(\dfrac{1}{\lambda}\bigl(a Z^n(t,x)-\frac{\kappa}{2} a^2\bigr)\right)}
    {\displaystyle\int_{\mathbb A}\exp\left(\dfrac{1}{\lambda}\bigl(a' Z^n(t,x)-\frac{\kappa}{2} (a')^2\bigr)\right)\,da'}.
\end{equation}
For the policy evaluation step \eqref{V^1, V^2, n}, define
\[
    \bar a^{n+1}(t,x):=\int_{\mathbb A}a\,\pi^{n+1}(t,x)(a)\,da,\qquad
    g^{n+1}(t,x):=-\frac{\kappa}{2}\int_{\mathbb A}a^2\,\pi^{n+1}(t,x)(a)\,da
    +\lambda\mathcal H\bigl(\pi^{n+1}(t,x)\bigr).
\]
The updated auxiliary function
\[
    V^{n+1,1}(\tau,t,x) = \mathbb E_{t,x}\!\left[\int_t^T \delta(s-\tau) g^{n+1}(s,X_s^{n+1})\,ds
    +\delta(T-\tau)U(B(X_T^{n+1}))\right],
\]
where $X^{n+1}$ is the state process under $\pi^{n+1}$, then solves the linear PDE on $\Delta[0,T]\times\mathbb R$:
\begin{equation}\label{eq:num_effort_pde3d}
\begin{cases}
    V_t^{n+1,1}(\tau,t,x)+\dfrac{1}{2}\sigma^2 V_{xx}^{n+1,1}(\tau,t,x)
    +\bigl(\mu+\bar a^{n+1}(t,x)\bigr)V_x^{n+1,1}(\tau,t,x)+\delta(t-\tau)g^{n+1}(t,x)=0,\\
    V^{n+1,1}(\tau,T,x)=\delta(T-\tau)U(B(x)).
\end{cases}
\end{equation}
Under \eqref{eq:num_discount}, the identity $\delta(s-\tau)=\sum_{j=1}^2 w_j e^{-\beta_j(t-\tau)}e^{-\beta_j(s-t)}$ for $0\leq\tau\leq t\leq s\leq T$ yields the decomposition
\begin{equation}\label{eq:num_effort_decomp}
    V^{n+1,1}(\tau,t,x)=\sum_{j=1}^2 w_j e^{-\beta_j(t-\tau)}v_j^{n+1}(t,x), \qquad
    Z^{n+1}(t,x)=V_x^{n+1,1}(t,t,x)=\sum_{j=1}^2 w_j (v_j^{n+1})_x(t,x),
\end{equation}
where, for $j=1,2$,
\[
    v_j^{n+1}(t,x):=\mathbb E_{t,x}\!\left[\int_t^T e^{-\beta_j(s-t)}g^{n+1}(s,X_s^{n+1})\,ds
    +e^{-\beta_j(T-t)}U(B(X_T^{n+1}))\right].
\]
Substituting \eqref{eq:num_effort_decomp} into \eqref{eq:num_effort_pde3d} yields
\begin{align*}
    \sum_{j=1}^2 w_j e^{-\beta_j(t-\tau)}\Bigl[&(v_j^{n+1})_t(t,x)+\frac{1}{2}\sigma^2(v_j^{n+1})_{xx}(t,x)\\
    &+\bigl(\mu+\bar a^{n+1}(t,x)\bigr)(v_j^{n+1})_x(t,x)-\beta_j v_j^{n+1}(t,x)+g^{n+1}(t,x)\Bigr]=0.
\end{align*}
Since $\beta_1\neq\beta_2$ and $w_1,w_2>0$, the functions $\tau\mapsto e^{-\beta_j(t-\tau)}$ ($j=1,2$) are linearly independent on $[0,t]$ for each $t>0$. Thus, at each iteration $n\ge0$, solving \eqref{eq:num_effort_pde3d} on $\Delta[0,T]\times\mathbb R$ reduces to solving two decoupled linear parabolic PDEs on $[0,T]\times\mathbb R$:
\begin{equation}\label{eq:num_effort_evaluation}
\begin{cases}
    (v_j^{n+1})_t(t,x)+\dfrac{1}{2}\sigma^2(v_j^{n+1})_{xx}(t,x)
    +\bigl(\mu+\bar a^{n+1}(t,x)\bigr)(v_j^{n+1})_x(t,x)-\beta_jv_j^{n+1}(t,x)+g^{n+1}(t,x)=0,\\
    v_j^{n+1}(T,x)=U(B(x)),\qquad j=1,2.
\end{cases}
\end{equation}
Consequently, choosing an initial guess $V^{0,1}(\tau,t,x)=\sum_{j=1}^2 w_j e^{-\beta_j(t-\tau)}v_j^0(t,x)$, the PIA operates directly on $(v_1^n,v_2^n)$ on $[0,T]\times\mathbb R$. At iteration $n\ge0$, we compute $Z^n(t,x)=\sum_{j=1}^2 w_j (v_j^n)_x(t,x)$, update the policy $\pi^{n+1}$ via \eqref{eq:num_effort_update}, and solve the decoupled system \eqref{eq:num_effort_evaluation} to update $(v_1^{n+1},v_2^{n+1})$.

The same decomposition applies to the equilibrium pair $(V^{*,1},\pi^*)$:
\begin{equation}\label{eq:num_effort_eq_decomp}
    V^{*,1}(\tau,t,x)=\sum_{j=1}^2 w_j e^{-\beta_j(t-\tau)}v_j^*(t,x), \qquad
    Z^*(t,x)=V_x^{*,1}(t,t,x)=\sum_{j=1}^2 w_j (v_j^*)_x(t,x),
\end{equation}
where, for $j=1,2$,
\[
    v_j^*(t,x):=\mathbb E_{t,x}\!\left[\int_t^T e^{-\beta_j(s-t)}g^*(s,X_s^{\pi^*})\,ds
    +e^{-\beta_j(T-t)}U(B(X_T^{\pi^*}))\right],
\]
with
\[
    \bar a^*(t,x):=\int_{\mathbb A}a\,\pi^*(t,x)(a)\,da,\qquad
    g^*(t,x):=-\frac{\kappa}{2}\int_{\mathbb A}a^2\,\pi^*(t,x)(a)\,da
    +\lambda\mathcal H\bigl(\pi^*(t,x)\bigr).
\]
Substituting \eqref{eq:num_effort_eq_decomp} into the EEHJB equation \eqref{eq:EEHJB} and matching the coefficients of $e^{-\beta_j(t-\tau)}$ ($j=1,2$) leads to the coupled system of parabolic equations for $(v_1^*,v_2^*)$ and the equilibrium policy $\pi^*$:
\begin{equation}\label{eq:num_effort_eq_system}
\begin{aligned}
    &(v_j^*)_t(t,x)+\frac{1}{2}\sigma^2(v_j^*)_{xx}(t,x)
    +\bigl(\mu+\bar a^*(t,x)\bigr)(v_j^*)_x(t,x)-\beta_j v_j^*(t,x)+g^*(t,x)=0,\\
    &v_j^*(T,x)=U(B(x)),\qquad j=1,2,\\
    &\pi^*(t,x)(a)=\Gamma\bigl(t,x,Z^*(t,x),a\bigr)
    =\frac{\exp\left(\dfrac{1}{\lambda}\bigl(a Z^*(t,x)-\frac{\kappa}{2} a^2\bigr)\right)}
    {\displaystyle\int_{\mathbb A}\exp\left(\dfrac{1}{\lambda}\bigl(a' Z^*(t,x)-\frac{\kappa}{2} (a')^2\bigr)\right)\,da'}.
\end{aligned}
\end{equation}

In our numerical experiments, we consider three different utility functions:
\[
\text{(i) } U(w) = \ln w, \qquad
\text{(ii) } U(w) = \frac{w^{1-\gamma}-1}{1-\gamma} \text{ with } \gamma = 1/2, \qquad
\text{(iii) } U(w) = w.
\]
For each utility function, we initialize the PIA with $(v_1^0, v_2^0) = (0, 0)$, update $\pi^{n+1}$ via \eqref{eq:num_effort_update}, and solve \eqref{eq:num_effort_evaluation} for $(v_1^{n+1}, v_2^{n+1})$ using a finite-difference scheme. To obtain a reference equilibrium policy, we solve the coupled system \eqref{eq:num_effort_eq_system} for $(v_1^*, v_2^*)$ by finite differences and compute $\pi^*$. In both finite-difference schemes, we truncate the unbounded spatial domain to $(-R, R)$ with $R > 0$, and evaluate the policies on a smaller region $D_0 := [0, T] \times [-R_0, R_0]$ with $0 < R_0 < R$ to reduce the influence of artificial boundary conditions.

\begin{table}[htbp]
\centering
\color{black}
\caption{Parameter specifications for the equilibrium effort problem.}
\label{tab:effort_parameters}
\begin{tabular}{lll}
\hline
Parameter & Symbol & Value \\
\hline
Time horizon & $T$ & 2 \\
Baseline drift & $\mu$ & 0.05 \\
Volatility & $\sigma$ & 0.25 \\
Effort upper bound & $\bar a$ & 0.5 \\
Effort cost parameter & $\kappa$ & 1 \\
Compensation bounds & $(B_{\min}, B_{\max})$ & $(1, 2)$ \\
Performance target & $K$ & 0 \\
Sensitivity parameter & $k$ & 2 \\
Entropy regularization parameter & $\lambda$ & 0.02 \\
Discount weights & $(w_1, w_2)$ & $(0.5, 0.5)$ \\
Discount rates & $(\beta_1, \beta_2)$ & $(0.1, 1)$ \\
\hline
\end{tabular}
\end{table}



Under the parameter specifications summarized in Table~\ref{tab:effort_parameters}, we set $R_0 = 3$ and $R = 6$.\footnote{A larger truncation bound $R$ leaves the computed policies on $D_0$ essentially unchanged.}  The PIA is stopped when the sup-norm difference between successive policy iterates falls below $10^{-8}$.
At iteration $n$, the policy error is evaluated by
\[
\|\pi^n-\pi^*\|_{D_0\times\mathbb A}^{(0)}
:=\sup_{(t,x,a)\in D_0\times\mathbb A}|\pi^n(t,x)(a)-\pi^*(t,x)(a)|.
\]
Figure~\ref{fig:effort_policy_error} shows numerical estimates of the policy error $\|\pi^n-\pi^*\|_{D_0\times\mathbb A}^{(0)}$ across iterations $n$ for the three utility functions.
In all three cases, the numerical results are consistent with the convergence result established in Theorem~\ref{thm:convergence}.

\begin{figure}[htbp]
\centering
\color{black}
\numfigure{0.6\textwidth}{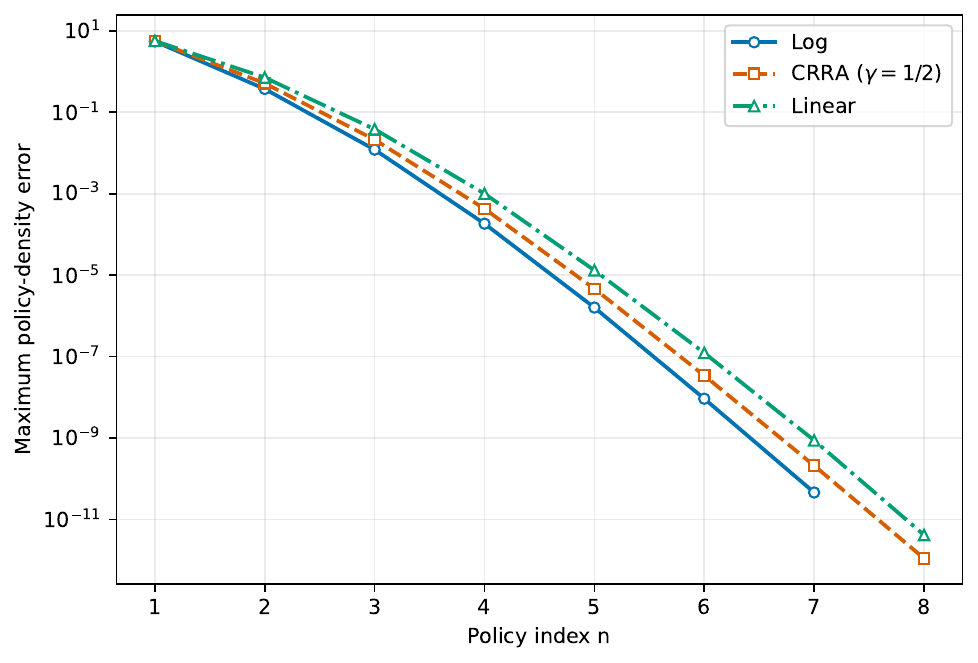}
\caption{Policy error $\|\pi^n-\pi^*\|_{D_0\times\mathbb A}^{(0)}$ versus iteration $n$ for the three utility functions, with a stopping tolerance of $10^{-8}$ for successive policy iterates.}
\label{fig:effort_policy_error}
\end{figure}

Finally, we examine the average effort under the PIA policy $\pi$,
\[
    \bar a(t,x):=\int_{\mathbb A}a\,\pi(t,x)(a)\,da,
\]
plotted in Figure~\ref{fig:effort_mean} for the three utility functions.

\begin{figure}[htbp]
\centering
\color{black}
\numfigure{\textwidth}{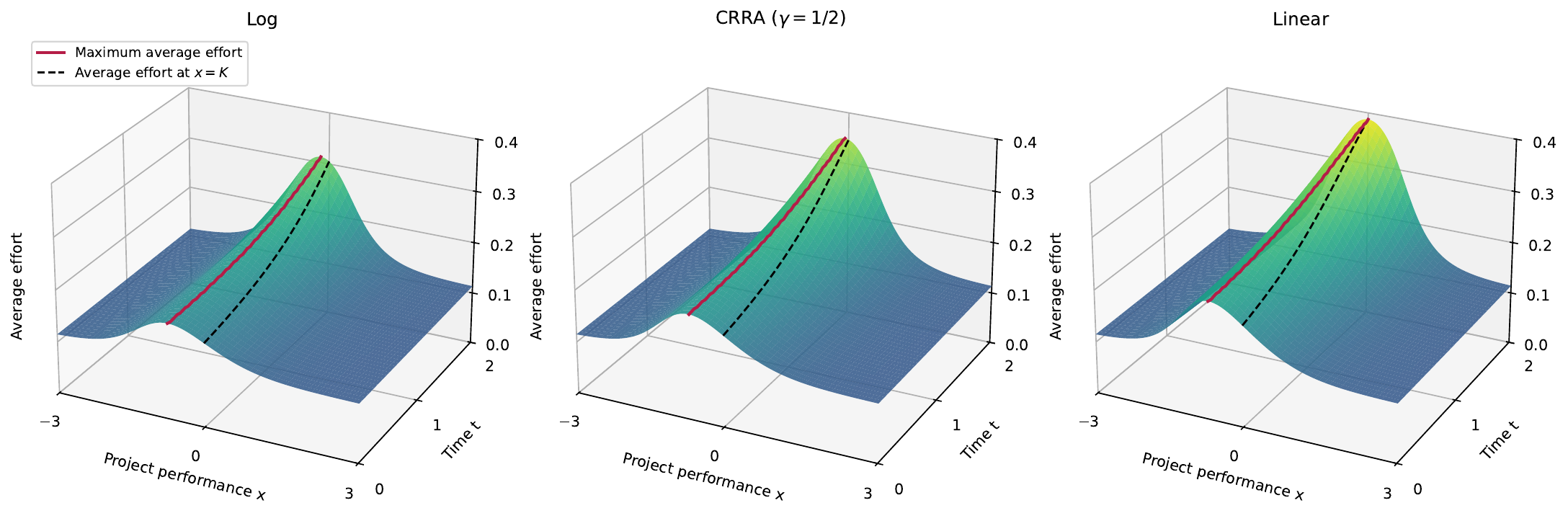}
\caption{Average effort under the PIA policy for logarithmic, power ($\gamma=1/2$), and linear utility. The solid curve traces the maximum average effort over $x\in[-R_0,R_0]$ at each time $t$, and the dashed curve shows the average effort at $x=K$.}
\label{fig:effort_mean}
\end{figure}

For all three utilities, average effort is highest near the performance target $K$ and is lower toward both ends of $[-R_0, R_0]$. This phenomenon is consequent on the structure of the contract. Because effort incurs the quadratic cost $\frac{\kappa}{2}a^2$ while the marginal payment $B'(x)$ vanishes as $x \to \pm\infty$, further improvements in performance far from $K$ have little effect on compensation. By contrast, $B'(x)$ attains its maximum at $K$, where compensation is most sensitive to performance, giving the manager a stronger incentive to exert effort near the target.

Moreover, the solid curve shows that at $t=0$, the maximum average effort is attained below $K$ for all three utilities, and as maturity approaches, the maximizing performance level shifts toward $K$. This pattern is consistent with the positive drift ($\mu + \bar a(t,x) > 0$) of project performance: because compensation depends on terminal performance, starting below $K$ allows terminal performance to reach the region around $K$ where $B'(x)$ is largest, while this drift effect weakens as the remaining horizon shortens. Near maturity, incentives are governed by $(U \circ B)'(x) = U'(B(x))B'(x)$, which attains its maximum at $K$ under linear utility, but strictly below $K$ under logarithmic and power utilities due to diminishing marginal utility (as $U'$ is strictly decreasing).

\subsection{Equilibrium Withdrawals from a Retirement Account}\label{subsect:num_retirement}
Consider a retiree managing withdrawals from a retirement account over an uncertain lifetime and leaving the remaining wealth as a bequest at death. Specifically, the wealth in the account is invested in an asset whose price $S$ follows a geometric Brownian motion
\[
\frac{dS_s}{S_s}=\mu_F\,ds+\sigma\,dW_s,
\]
where $\mu_F$ is the expected growth rate and $\sigma>0$ is the volatility. The retiree chooses a proportional withdrawal rate $a\in\mathbb A:=[0,\bar c]$, with $\bar c>0$, according to a relaxed policy $\pi$. Note that an upper bound $\bar c$ is commonly imposed on withdrawals from a retirement account (see, e.g., the widely accepted ``4\% rule.'') The account wealth process $\mathcal W^\pi$ then evolves as
\[
d\mathcal W_s^\pi=\mathcal W_s^\pi\left(\int_{\mathbb A}(\mu_F-a)\,\pi(s,\mathcal W_s^\pi)(a)\,da\right)ds+\sigma\mathcal W_s^\pi\,dW_s,
\qquad \mathcal W_t^\pi=w>0.
\]
To fit into the model assumption of bounded coefficients in the present work, we define the state variable as the log-wealth $X_s^\pi:=\ln\mathcal W_s^\pi$. By It\^o's formula, expressing the policy directly in terms of log-wealth, the dynamics of $X^\pi$ are given by
\[
dX_s^\pi=\left(\int_{\mathbb A}(c_0-a)\,\pi(s,X_s^\pi)(a)\,da\right)ds+\sigma\,dW_s,
\qquad X_t^\pi=x\in\mathbb R,
\]
where $c_0:=\mu_F-\frac12\sigma^2$.

Let $\zeta$ be the death time, independent of the Brownian motion $W$, with constant mortality intensity $\eta>0$, so that
\[
\mathbb P(\zeta>s\mid \zeta>t)=e^{-\eta(s-t)},\qquad s\geq t.
\]
The retiree derives utility $U(ae^x)$ from withdrawals and utility $\theta U(e^x)$ from the bequest, where $\theta>0$ and $U:[0,\infty)\to\mathbb R$. Under the discount function $\delta$ in \eqref{eq:num_discount}, the entropy-regularized objective function is given by
\begin{align}\label{eq:num_retirement_objective}
J^\pi(t,x)
&=\mathbb E_{t,x}\!\Biggl[
\int_t^\zeta \delta(s-t)\left(\int_{\mathbb A} U(ae^{X_s^\pi})\,\pi(s,X_s^\pi)(a)\,da
+\lambda\mathcal H\bigl(\pi(s,X_s^\pi)\bigr)\right)ds \notag\\
&\qquad\qquad +\delta(\zeta-t)\theta U(e^{X_\zeta^\pi})\,\Bigg|\,\zeta>t\Biggr]\notag\\
&=\mathbb E_{t,x}\!\left[
\int_t^\infty \delta_\eta(s-t)
\left(\int_{\mathbb A}\bigl(U(ae^{X_s^\pi})+\eta\theta U(e^{X_s^\pi})\bigr)\pi(s,X_s^\pi)(a)\,da
+\lambda\mathcal H\bigl(\pi(s,X_s^\pi)\bigr)\right)ds\right],
\end{align}
where $\delta_\eta(u):=e^{-\eta u}\delta(u)=\sum_{j=1}^2 w_j e^{-(\beta_j+\eta)u}$, and the second equality follows from the independence of $\zeta$ and $W$ together with Fubini's theorem (for policies satisfying \eqref{eq:ih_adm_integrability} with discount function $\delta_\eta$). We consider three utility functions, with $\alpha>0$:
\[
\text{(i) } U(c)=-e^{-\alpha c},\qquad
\text{(ii) } U(c)=\frac{1}{1+e^{-\alpha c}},\qquad
\text{(iii) } U(c)=\arctan(\alpha c).
\]

In the framework of Section~\ref{subsect:infinite_horizon}, \eqref{eq:num_retirement_objective} corresponds to $b(t,x,a)=c_0-a$ and $r(x,a)=U(ae^x)+\eta\theta U(e^x)$, where $r$ is independent of $(t,y)$. Here $0<|\mathbb A|=\bar c<\infty$, $b$ is uniformly bounded, and $\sigma>0$ is constant and uniformly non-degenerate. Setting $c:=ae^x\geq0$, we have
\[
\partial_x U(ae^x)=cU'(c),\qquad
\partial_{xx} U(ae^x)=cU'(c)+c^2U''(c).
\]
For each utility function (i)--(iii), $U(c)$, $cU'(c)$, and $c^2U''(c)$ are uniformly bounded on $[0,\infty)$, and the same estimates with $a=1$ yield the uniform boundedness of $\theta U(e^x)$ and its first two derivatives. Hence $\|r(\cdot,a)\|^{(2)}\leq C_0$ on $\mathbb R$ uniformly in $a\in\mathbb A$, so Assumption~\ref{assump:ih}~(i) holds. Moreover, $\delta_\eta(0)=1$, $\delta_\eta|_{[0,T]}\in C^2([0,T])$ for every $T>0$, and $0<\delta_\eta(u)\leq e^{-\rho u}$ for all $u\geq0$ with $C_\delta=1$ and $\rho:=\min_{j=1,2}(\beta_j+\eta)>0$, which verifies Assumption~\ref{assump:ih}~(ii). Finally, taking the initial guess $V^{0,1}\equiv0$ satisfies Assumption~\ref{assump:ih}~(iii).

Similar to Section~\ref{subsect:num_effort}, we only need to consider the auxiliary function $V^{n,1}(\tau,t,x)$ (and $V^{*,1}(\tau,t,x)$). Because $b$, $\sigma$, and $r$ are independent of $t$, the infinite-horizon PIA \eqref{eq:ih_update}--\eqref{eq:ih_evaluation} initiated with $V^{0,1}\equiv0$ preserves stationarity. Specifically, if $V^{n,1}(\tau,t,x)=\sum_{j=1}^2 w_j e^{-(\beta_j+\eta)(t-\tau)}v_j^n(x)$, then $Z^n(x):=V_x^{n,1}(t,t,x)=\sum_{j=1}^2 w_j (v_j^n)'(x)$ depends only on $x$, and the policy update \eqref{eq:ih_update} computes the stationary policy $\pi^{n+1}$ via the Gibbs measure
\begin{equation}\label{eq:num_retirement_update}
\pi^{n+1}(x)(a)=
\frac{\exp\left(\dfrac{1}{\lambda}\bigl[(c_0-a)Z^n(x)+U(ae^x)\bigr]\right)}
{\displaystyle\int_{\mathbb A}\exp\left(\dfrac{1}{\lambda}\bigl[(c_0-a')Z^n(x)+U(a'e^x)\bigr]\right)\,da'},
\end{equation}
where the action-independent bequest term $\eta\theta U(e^x)$ cancels from the numerator and denominator. For the policy evaluation step \eqref{eq:ih_evaluation}, let us define
\[
\bar a^{n+1}(x):=\int_{\mathbb A}a\,\pi^{n+1}(x)(a)\,da,\qquad
g^{n+1}(x):=\int_{\mathbb A} U(ae^x)\,\pi^{n+1}(x)(a)\,da+\lambda\mathcal H\bigl(\pi^{n+1}(x)\bigr).
\]
As in Section~\ref{subsect:num_effort}, the weighted-exponential form of $\delta_\eta$ and the time-homogeneity of the state process $X^{n+1}$ under $\pi^{n+1}$ yield the decomposition
\[
V^{n+1,1}(\tau,t,x)=
\sum_{j=1}^2 w_j e^{-(\beta_j+\eta)(t-\tau)}v_j^{n+1}(x),
\]
where, for $j=1,2$,
\[
v_j^{n+1}(x):=\mathbb E_{0,x}\!\left[
\int_0^\infty e^{-(\beta_j+\eta)s}
\bigl(g^{n+1}(X_s^{n+1})+\eta\theta U(e^{X_s^{n+1}})\bigr)\,ds\right].
\]
Consequently, the policy evaluation step reduces to solving two decoupled linear ordinary differential equations on $\mathbb R$:
\begin{equation}\label{eq:num_retirement_evaluation}
\frac{1}{2}\sigma^2(v_j^{n+1})''(x)
+\bigl(c_0-\bar a^{n+1}(x)\bigr)(v_j^{n+1})'(x)
-(\beta_j+\eta)v_j^{n+1}(x)+g^{n+1}(x)+\eta\theta U(e^x)=0,
\qquad j=1,2.
\end{equation}
A stationary equilibrium $(v_1^*,v_2^*,\pi^*)$ satisfies the corresponding coupled system obtained by replacing $(v_1^{n+1},v_2^{n+1},\pi^{n+1},Z^n)$ in \eqref{eq:num_retirement_update}--\eqref{eq:num_retirement_evaluation} with $(v_1^*,v_2^*,\pi^*,Z^*)$, where $Z^*(x):=\sum_{j=1}^2 w_j(v_j^*)'(x)$.

As in Section~\ref{subsect:num_effort}, for each utility function we initialize the PIA with $(v_1^0,v_2^0)=(0,0)$, update $\pi^{n+1}$ via \eqref{eq:num_retirement_update} and solve \eqref{eq:num_retirement_evaluation} using a finite-difference scheme, and obtain the reference equilibrium policy $\pi^*$ by solving the coupled stationary system by finite differences. In both schemes, we truncate the spatial domain to $(-R,R)$ with $R>0$ and evaluate the policies on $D_0:=[-R_0,R_0]$ with $0<R_0<R$.

\begin{table}[htbp]
\centering
\color{black}
\caption{Parameter specifications for the retirement withdrawal problem.}
\label{tab:parameters}
\begin{tabular}{lll}
\hline
Parameter & Symbol & Value \\
\hline
Expected growth rate & $\mu_F$ & 0.06 \\
Volatility & $\sigma$ & 0.125 \\
Withdrawal rate upper bound & $\bar c$ & 0.25 \\
Mortality intensity & $\eta$ & 0.05 \\
Bequest weight & $\theta$ & 0.5 \\
Utility parameter & $\alpha$ & 1 \\
Entropy regularization parameter & $\lambda$ & 0.02 \\
Discount weights & $(w_1,w_2)$ & $(0.5,0.5)$ \\
Subjective discount rates & $(\beta_1,\beta_2)$ & $(0.1,1)$ \\
\hline
\end{tabular}
\end{table}


Under the parameter specifications summarized in Table~\ref{tab:parameters}, we set $R_0=3$ and $R=8$.\footnote{As in Section~\ref{subsect:num_effort}, increasing $R$ beyond $8$ leaves the computed policies on $D_0$ essentially unchanged.} Similarly, the PIA is stopped when the sup-norm difference between successive policy iterates falls below \(10^{-8}\).
Figure~\ref{fig:policy_error} shows numerical estimates of the policy error $\|\pi^n-\pi^*\|_{D_0\times\mathbb A}^{(0)}$ (defined analogously to Section~\ref{subsect:num_effort}) across iterations $n$.  Although the discounting threshold $\rho_0$ in Proposition~\ref{prop:ih} is not explicitly specified, the PIA iterates converge numerically to $\pi^*$ across all three utilities in our numerical experiments.

\begin{figure}[htbp]
\centering
\color{black}
\numfigure{0.6\textwidth}{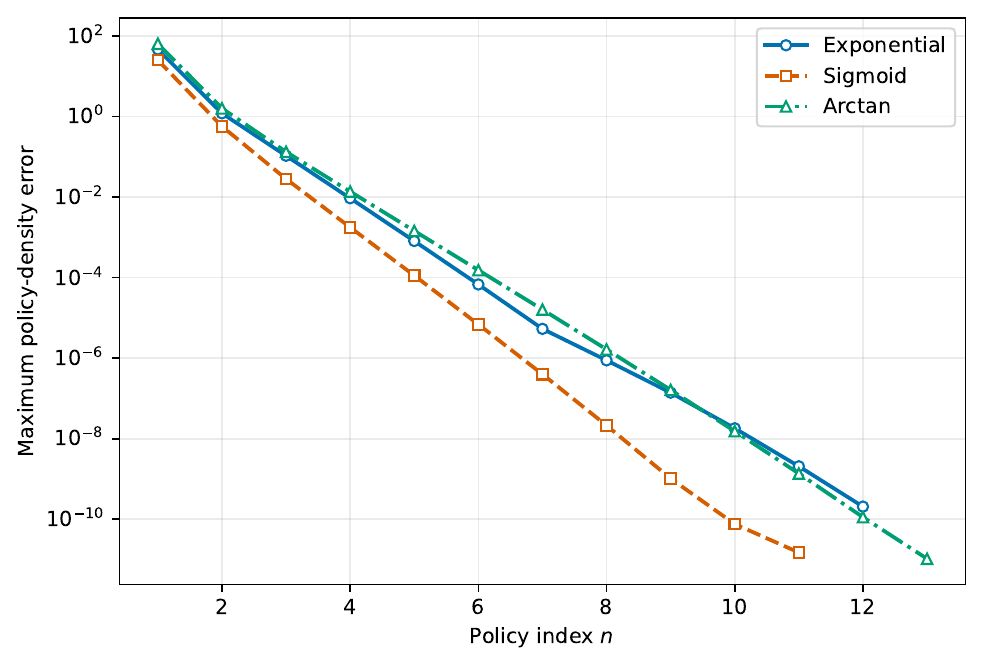}
\caption{Policy error $\|\pi^n-\pi^*\|_{D_0\times\mathbb A}^{(0)}$ versus iteration $n$ for the three utility functions, with a stopping tolerance of $10^{-8}$ for successive policy iterates.}
\label{fig:policy_error}
\end{figure}

}

\ \\
\noindent
\textbf{Acknowledgements}:
X. Yu is supported by the Hong Kong RGC General Research Fund (GRF) under grant no. 15211524 and the Hong Kong Polytechnic University research grant under no. P0045654.

{
\small
\bibliography{ref}
}

\end{document}